\documentclass[11pt, reqno]{amsart}
\usepackage[margin= 2.0cm]{geometry}
\usepackage{amsmath}
\usepackage{amsfonts}
\usepackage{amssymb}
\usepackage{amsthm}
\usepackage{mathtools} % An improvement of amsmath
\usepackage{latexsym}
\usepackage{enumerate}
\usepackage{cancel}
\usepackage{ragged2e}
\usepackage{blindtext}
\usepackage{cases}
\usepackage{empheq}
\usepackage{multicol}

\usepackage{wrapfig}
\usepackage{caption}

\usepackage{subfigure}
\usepackage{float}

\usepackage{booktabs}
\usepackage{array}

\usepackage{color}

\usepackage[color=white,linecolor=black]{todonotes}
\usepackage{mathrsfs}
\usepackage{fontenc} %T1 font encoding
\usepackage{inputenc} %UTF-8 support
\usepackage{enumitem}
\usepackage{verbatim}

\theoremstyle{plain}
\newtheorem{theorem}{Theorem}[section]

\newtheorem{lemma}[theorem]{Lemma}

\theoremstyle{definition}
\newtheorem{definition}[theorem]{Definition}

\theoremstyle{remark}

\newtheorem{hypothesis}{Hypothesis}
\numberwithin{equation}{section} %% Equation numbering control.
\numberwithin{figure}{section}   %% Figure numbering control.

\usepackage[square,comma,numbers,sort]{natbib}
\usepackage[colorlinks=true, pdfborder={ 0 0 0}]{hyperref}
\hypersetup{urlcolor=blue, citecolor=red}
\usepackage{url}
\makeatletter
\renewcommand{\section}{\@startsection{section}{1}{0pt}%
  {1.5ex plus .2ex minus .2ex}%
  {1.0ex plus .2ex}%
  {\normalfont\large\bfseries\raggedright}}
\renewcommand{\subsection}{\@startsection{subsection}{2}{0pt}%
  {1.25ex plus .2ex minus .2ex}%
  {1.0ex plus .2ex}%
  {\normalfont\normalsize\bfseries\raggedright}}
\makeatother
\renewenvironment{abstract}{%
  \begin{center}
    \textbf{\abstractname}
  \end{center}
}{}
\makeatother
\allowdisplaybreaks
\begin{document}
\makeatletter
% Prevent title from being forced to uppercase
\def\@settitle{\begin{center}%
  \normalfont\LARGE\bfseries \@title
  \end{center}%
}
\makeatother

\title[Stochastic MHD]{Stochastic Two-Dimensional Anisotropic Magnetohydrodynamics System with Multiplicative Noise}
% ======================  Author Information ======================
\author[M. M. Rahman]{Mohammad Mahabubur Rahman}
\address[Mohammad Mahabubur Rahman]{School of Mathematical and Statistical Sciences, 
                Clemson University, 
       Clemson, SC, 29634, USA}
\email[Mohammad Mahabubur Rahman]{rahman6@clemson.edu}

\maketitle{}
\begin{abstract}
In this paper, we investigate the stochastic two-dimensional anisotropic magnetohydrodynamics system on $\mathbb T^2$ with velocity dissipation acting only in the horizontal direction. Under multiplicative noise, we prove the existence of global martingale solutions and global probabilistically weak solutions, together with pathwise uniqueness.

\end{abstract}
\section{Introduction}

\noindent
\noindent
The incompressible magnetohydrodynamics (MHD) equations constitute a
fundamental model for the motion of electrically conducting fluids in the
presence of a magnetic field. The mathematical theory of incompressible
MHD was developed in the classical works of Duvaut and Lions
\cite{DuvautLions1972} and Sermange and Temam
\cite{SermangeTemam1983}. An important direction in the study of MHD
equations concerns global regularity under partial or fractional
dissipation and diffusion.

\medskip

\noindent
\noindent
The global regularity problem for MHD equations
has been studied extensively. Among these developments, Wu \cite{Wu2003} considered generalized
MHD equations with fractional dissipation and magnetic diffusion. In two
dimensions, Cao and Wu \cite{CaoWu2011} established global regularity
for MHD equations with mixed partial dissipation and magnetic diffusion. In particular, their result includes
the cases of vertical viscosity with horizontal magnetic diffusion and,
symmetrically, horizontal viscosity with vertical magnetic diffusion.
They also obtained global existence, uniqueness, and conditional
regularity of weak solutions for the two-dimensional MHD equations with
magnetic diffusion and without velocity dissipation. Cao, Regmi and Wu
\cite{CaoRegmiWu2013} studied the two-dimensional MHD equations with
horizontal velocity dissipation and horizontal magnetic diffusion. They
proved global $L^{2r}$ bounds for $1\leq r<\infty$ and established a
conditional global regularity criterion, while the global regularity of
classical solutions for the original system remains open. Cao, Wu and
Yuan \cite{CaoWuYuan2014} considered the two-dimensional MHD equations
without velocity dissipation and with fractional magnetic diffusion
$(-\Delta)^\beta$, and proved global regularity for $\beta>1$. The global regularity of MHD systems under partial, anisotropic, or
weakened dissipation has received considerable attention; see, for
example, \cite{yamazaki2014remarks, yamazaki2015global}
and the references therein. Further
results on two-dimensional MHD equations with partial or fractional
dissipation can be found in
\cite{DuZhou2015,JiuNiuWuXuYu2015,DongJiaLiWu2018,DongLiWu2019}.

\medskip

\noindent
\noindent
\noindent
The case of horizontal dissipation in the velocity equation and
horizontal magnetic diffusion in the magnetic equation is particularly
relevant to the present work. Global regularity of classical solutions
for the corresponding two-dimensional deterministic MHD system remains
open \cite{CaoRegmiWu2013}. This motivates the question of whether,
under multiplicative stochastic forcing, one can establish global
martingale solutions and pathwise uniqueness when only horizontal
dissipation and horizontal magnetic diffusion are present.

\medskip

\noindent

\medskip

\noindent
\noindent
Our first attempt is therefore the stochastic MHD system with horizontal
dissipation in the velocity equation and horizontal magnetic diffusion
in the magnetic equation. At the $L^2$ level, the corresponding
estimates provide control of $\partial_1u$ and $\partial_1B$, while at
the $H^{0,1}$ level they provide control of
$\partial_1\partial_2u$ and $\partial_1\partial_2B$. However, these
derivatives are not sufficient to close the coupled $H^{0,1}$ estimate.
Indeed, the anisotropic inequalities used to estimate the nonlinear MHD
terms require additional vertical control of the magnetic field, which
is not supplied by horizontal magnetic diffusion alone.

\medskip

\noindent
\noindent
For the deterministic MHD equations with mixed partial dissipation and
magnetic diffusion, Cao and Wu \cite{CaoWu2011} obtained higher-order
estimates by taking curl of the velocity and magnetic equations and
working with the vorticity and current density. A direct use of this
procedure under multiplicative stochastic forcing would differentiate
the noise coefficients, producing stochastic terms involving
$\operatorname{curl}\sigma_i(t,u,B)$ and requiring corresponding spatial
regularity assumptions on the noise. We do not impose such assumptions
here and instead work directly with the velocity and magnetic fields in
anisotropic Sobolev spaces.

\medskip

\noindent
This leads us first to the following two-dimensional deterministic
anisotropic MHD system:
\begin{subequations}
\label{MHD-deterministic}
\begin{align}
\dfrac{\partial u}{\partial t}
+(u\cdot\nabla)u
-(B\cdot\nabla)B
-\nu\partial_1^2u
+\nabla p
&=0,
\label{DMHDu}
\\
\dfrac{\partial B}{\partial t}
+(u\cdot\nabla)B
-(B\cdot\nabla)u
-\eta\Delta B
&=0,
\label{DMHDB}
\\
\nabla\cdot u&=0,
\qquad
\nabla\cdot B=0.
\label{DMHD}
\end{align}
\end{subequations}
Here $u,B:\mathbb R_+\times\mathbb T^2\to\mathbb R^2$ denote the
velocity and magnetic fields, respectively, and $\nu>0$ and $\eta>0$
are the viscosity and magnetic diffusivity. The velocity equation retains
only horizontal viscosity, whereas the magnetic equation contains the
full Laplacian.

\medskip

\noindent
For \eqref{MHD-deterministic}, we establish  global
regularity in two dimensions for any initial data. The argument is carried out directly in anisotropic Sobolev spaces.
At the $L^2$ level, the dissipative terms provide control of
$\partial_1u$ and $\nabla B$, and at the $H^{0,1}$ level they provide
control of $\partial_1\partial_2u$ and $\nabla\partial_2B$. In
particular, the full magnetic diffusion supplies the vertical magnetic
regularity that is missing in the horizontal-horizontal system and
allows the coupled anisotropic estimate to close without a smallness
condition on the initial data.

\medskip

\noindent
Stochastic MHD equations have been studied under several forms of random
forcing. Sritharan and Sundar \cite{SritharanSundar1999} considered the
martingale problem associated with stochastic MHD equations. Barbu and
Da Prato \cite{BarbuDaPrato2007} studied two-dimensional MHD equations
with additive Gaussian noise and proved the existence and uniqueness of
an invariant measure. Chueshov and Millet \cite{ChueshovMillet2010}
considered stochastic hydrodynamical systems including two-dimensional
MHD equations with multiplicative Gaussian noise, obtaining results on
existence, uniqueness, and large deviations.  Yamazaki
\cite{yamazaki2016global} established the existence of a global
martingale solution for the stochastic nonhomogeneous MHD system and,
more recently, studied the two-dimensional MHD system driven by
space-time white noise \cite{yamazaki2026remarks}.

\medskip

\noindent
Stochastic MHD equations with fractional or partial dissipation have also
been investigated. Huang and Shen \cite{HuangShen2016} studied a
two-dimensional stochastic fractional anisotropic MHD system with
Gaussian multiplicative noise. Hong, Li and Liu
\cite{HongLiLiu2021} considered stochastic MHD equations with fractional
dissipation and established the existence of martingale solutions,
together with pathwise uniqueness and strong solutions under stronger
conditions on the fractional exponents. Li, Liu and Tang
\cite{LiLiuTang2021} studied stochastic MHD equations with fractional
dissipation in the velocity equation and componentwise partial magnetic
diffusion. They obtained local pathwise existence and uniqueness for
general multiplicative noise and global results for particular linear
multiplicative noises.

\medskip

\noindent
Motivated by these results and by the deterministic system
\eqref{MHD-deterministic}, the main object of the present paper is the
two-dimensional stochastic anisotropic MHD system
\begin{align}
\begin{cases}
du+
\bigl[
(u\cdot\nabla)u-(B\cdot\nabla)B
-\nu\partial_1^2u+\nabla p
\bigr]\,dt
=
\sigma_1(t,u,B)\,dW_1,
\\[1mm]
dB+
\bigl[
(u\cdot\nabla)B-(B\cdot\nabla)u
-\eta\Delta B
\bigr]\,dt
=
\sigma_2(t,u,B)\,dW_2,
\\[1mm]
\nabla\cdot u=0,
\qquad
\nabla\cdot B=0,
\end{cases}
\label{MHD}
\end{align}where $u,B:\mathbb R_+\times\mathbb T^2\to\mathbb R^2$ denote the
velocity and magnetic fields, respectively, and $\nu>0$ and $\eta>0$
are the viscosity and magnetic diffusivity. In addition, $W_1$ and $W_2$ are independent cylindrical
Wiener processes and $\sigma_1$ and $\sigma_2$ are multiplicative noise.

\medskip

\noindent
For \eqref{MHD}, we establish the existence of a global martingale
solution and pathwise uniqueness. The analysis is based on estimates at
the $L^2$ and $H^{0,1}$ levels that preserve the one-directional
viscosity in the velocity equation and exploit the full magnetic
diffusion in the coupled nonlinear estimates. In particular, no vertical
viscosity is imposed on the velocity equation.

\medskip

\noindent \textbf{The remainder of the paper is organized as follows.} We first introduce
the anisotropic functional setting, formulate the assumptions, and state
the main results. We then establish stochastic a priori estimates and prove
the existence of global martingale solutions. Next, using the martingale
representation theorem, we prove the existence of global probabilistically
weak solutions. Finally, we prove pathwise uniqueness.

\medskip

\noindent We begin by fixing the functional setting and notation needed for the subsequent analysis.

\section{Preliminaries and notation}

\noindent Throughout this paper, we work on the two-dimensional periodic torus
$
\mathbb T^2
=
\mathbb R/2\pi\mathbb Z
\times
\mathbb R/2\pi\mathbb Z
=
(\mathbb T_h,\mathbb T_v),
$
where $h$ denotes the horizontal variable $x_1$ and $v$ denotes
the vertical variable $x_2$. We denote by
$$
H
:=
\left\{
u\in L^2:
\nabla\cdot u=0
\right\},
$$
endowed with the $L^2$-inner product and norm. For $s\geq0$, let $H^s$ denote the space of divergence-free
vector fields on $\mathbb T^2$ satisfying
$$
\lVert u\rVert_{H^s}^2
:=
\sum_{k\in\mathbb Z^2}
(1+\lvert k \rvert^2)^s
\lvert \widehat u_k\rvert^2
<\infty,
$$
where $\widehat u_k$ denotes the $k$-th Fourier coefficient of $u$. Due to the anisotropic structure of the system, for $s,s'\geq0$,
we denote the anisotropic Sobolev space of divergence-free vector fields
on $\mathbb T^2$ by $H^{s,s'}$, defined by
$$
\lVert u\rVert_{H^{s,s'}}^2
:=
\sum_{k\in\mathbb Z^2}
(1+\lvert k_1\rvert^2)^s
(1+\lvert k_2\rvert^2)^{s'}
\lvert\widehat u_k\rvert^2,
\qquad
k=(k_1,k_2).
$$
For $p,q\in[1,\infty]$, we denote by
$L_h^p(L_v^q)$ the mixed Lebesgue space equipped with the norm
$$
\lVert u\rVert_{L_h^p(L_v^q)}
:=
\left(
\int_{\mathbb T_h}
\left(
\int_{\mathbb T_v}
\lvert u(x_1,x_2)\rvert^q\,dx_2
\right)^{p/q}
dx_1
\right)^{1/p}.
$$
Similarly,
$$
\lVert u\rVert_{L_v^q(L_h^p)}
:=
\left(
\int_{\mathbb T_v}
\left(
\int_{\mathbb T_h}
\lvert u(x_1,x_2)\rvert^p\,dx_1
\right)^{q/p}
dx_2
\right)^{1/q}.
$$
We denote by $H^{-1}$ the dual space of $H^1$. Its norm is defined by
$$
\lvert f\rVert_{H^{-1}}
:=
\sup_{\substack{\phi\in H^1\\ \phi\neq0}}
\frac{
\lVert
\langle f,\phi\rangle_{H^{-1},H^1}
\rvert
}{
\lVert \phi\rVert_{H^1}
}
=
\sup_{\substack{\phi\in H^1\\ \lvert \phi\rVert_{H^1}\leq1}}
\lvert
\langle f,\phi\rangle_{H^{-1},H^1}
\rvert.
$$
Let $\mathfrak U$ and $X$ be separable Hilbert spaces. We denote
by $L_2(\mathfrak U;X)$ the space of Hilbert-Schmidt operators
from $\mathfrak U$ into $X$, equipped with the norm
$$
\lVert G\rVert_{L_2(\mathfrak U;X)}^2
:=
\sum_{j=1}^{\infty}
\lVert G\psi_j\rVert_X^2,
$$
where $\{\psi_j\}_{j\geq1}$ is any orthonormal basis of
$\mathfrak U$. Let's recall Hilbert-Schmidt inequality,
\begin{align}
\lVert G^*v\rVert_{\mathfrak U}^2
\le
\lVert G\rVert_{L_2(\mathfrak U;L^2)}^2
\lVert v\rVert_{L^2}^2.
\label{LKH8}
\end{align}  In addition, we use
$$
(u,v) 
:=
(u,v)_{L^2}
=
\int_{\mathbb T^2}
u(x)\cdot v(x)\,dx,
\qquad (u,v)_{H^{0,1}}
:=
(u,v)
+
(\partial_2u,\partial_2v),$$
to denote  the inner products in $L^2$ and $H^{0,1}$ respectively.
\medskip

\noindent 
Let $
\mathbf P:L^2\longrightarrow H$
denote the Leray projection onto the space of divergence-free vector
fields. In particular, for every $f\in L^2$ and $v\in H$,
$
(\mathbf P f,v)=(f,v),
$
and
$
\mathbf P\nabla p=0.$
\noindent We now state the assumptions imposed on $\sigma_i, i=1,2 $
that will be used in the existence and pathwise uniqueness.
\begin{hypothesis}
\label{HYp1}
Let
$$
\sigma_1:
\left(
\mathbb R_+\times
H^{1,1}\times H^{1,1},
\mathcal B\bigl(
\mathbb R_+\times
H^{1,1}\times H^{1,1}
\bigr)
\right)
\longrightarrow
\left(
L_2(\mathfrak U_1;H^{1,1}),
\mathcal B\bigl(
L_2(\mathfrak U_1;H^{1,1})
\bigr)
\right)
$$
and
$$
\sigma_2:
\left(
\mathbb R_+\times
H^{1,1}\times H^{1,1},
\mathcal B\bigl(
\mathbb R_+\times
H^{1,1}\times H^{1,1}
\bigr)
\right)
\longrightarrow
\left(
L_2(\mathfrak U_2;H^{1,1}),
\mathcal B\bigl(
L_2(\mathfrak U_2;H^{1,1})
\bigr)
\right)
$$
are measurable mappings such that

\medskip

\noindent
\textbf{(I) Growth assumptions for existence.}

There exist nonnegative constants
$
K_i',\,K_i,\,\widetilde K_i,
\;
i=0,1,2,
$
such that for every
$
t\ge0
$
and every
$
(u,B)\in H^{1,1}\times H^{1,1},
$
the following estimates hold.

\medskip

\noindent
\textbf{(a) $H^{-1}$-growth}
\begin{align}
&
\lVert \sigma_1(t,u,B)\rVert_{L_2(\mathfrak U_1;H^{-1})}^2
+
\lVert \sigma_2(t,u,B)\rVert_{L_2(\mathfrak U_2;H^{-1})}^2
\le
K_0'
+
K_1'
\left(
\lVert u\rVert_{L^2}^2
+
\lVert B\rVert_{L^2}^2
\right).
\label{NH1}
\end{align}

\medskip
\noindent
\textbf{(b) $H^{-1}$-Lipschitz continuity.}
There exists a constant $L_0\geq0$ such that, for every
$t\geq0$ and every
$
(u_1,B_1),(u_2,B_2)\in H^{1,1}\times H^{1,1},
$ $\sigma_i, i=1,2$
\begin{align}
&
\left\lVert 
\sigma_1(t,u_1,B_1)
-
\sigma_1(t,u_2,B_2)
\right\rVert_{L_2(\mathfrak U_1;H^{-1})}^2
\nonumber\\
&\qquad
+
\left\lVert 
\sigma_2(t,u_1,B_1)
-
\sigma_2(t,u_2,B_2)
\right\rVert_{L_2(\mathfrak U_2;H^{-1})}^2
\nonumber\\
&\qquad\leq
L_0
\left(
\lVert u_1-u_2\rVert_{L^2}^2
+
\lVert B_1-B_2\rVert_{L^2}^2
\right).
\label{NlH}
\end{align}
\medskip
\noindent
\textbf{(c) $L^2$-growth}
\begin{align}
\lVert \sigma_1(t,u,B)\rVert_{L_2(\mathfrak U_1;L^2)}^2
+
\lVert \sigma_2(t,u,B)\rVert_{L_2(\mathfrak U_2;L^2)}^2
&\le
K_0
+
K_1
\left(
\lVert u\rVert_{L^2}^2
+
\lVert B\rVert_{L^2}^2
\right)
\nonumber\\
&
+
K_2
\left(
\nu\lVert \partial_1u\rVert_{L^2}^2
+
\eta\lVert \nabla B\rVert_{L^2}^2
\right).
\label{nL2}
\end{align}

\medskip

\noindent
\textbf{(d) $H^{0,1}$-growth}
\begin{align}
\lVert \sigma_1(t,u,B)\rVert_{L_2(\mathfrak U_1;H^{0,1})}^2
+
\lVert \sigma_2(t,u,B)\rVert_{L_2(\mathfrak U_2;H^{0,1})}^2
&\le
\widetilde K_0
+
\widetilde K_1
\left(
\lVert u\rVert_{H^{0,1}}^2
+
\lVert B\rVert_{H^{0,1}}^2
\right)
\nonumber\\
&+
\widetilde K_2
\left(
\nu\lVert \partial_1\partial_2u\rVert_{L^2}^2
+
\eta\lVert \nabla\partial_2B\rVert_{L^2}^2
\right).
\label{nH01}
\end{align}
\noindent \textbf{(II) Lipschitz assumptions for pathwise uniqueness}
Assume there exist constants $L_1,L_2\ge0$ such that, for every
$
(u_1,B_1),(u_2,B_2)\in H^{1,0}\times H^1,
$
the diffusion coefficients satisfy
\begin{align}
\lVert \sigma_1(t,u_1,B_1)-\sigma_1(t,u_2,B_2)\rVert_
{L_2(\mathfrak U_1;L^2)}^2
&+
\lVert \sigma_2(t,u_1,B_1)-\sigma_2(t,u_2,B_2)\rVert_
{L_2(\mathfrak U_2;L^2)}^2
\nonumber\\
&\qquad
\le
L_1
\left(
\lVert u_1-u_2\rVert_{L^2}^2
+
\lVert B_1-B_2\rVert_{L^2}^2
\right)
\nonumber\\
&\qquad\quad
+
L_2
\left(
\nu\lVert \partial_1(u_1-u_2)\rVert_{L^2}^2
+
\eta\lVert \nabla(B_1-B_2)\rVert_{L^2}^2
\right).
\label{LIp}
\end{align}
\end{hypothesis}
\noindent 
Define
\begin{align}
\Sigma_i(t)
:=
\sigma_i\!\left(t,u^{(1)}(t),B^{(1)}(t)\right)
-
\sigma_i\!\left(t,u^{(2)}(t),B^{(2)}(t)\right),
\qquad
i=1,2.
\end{align}
then, applying \eqref{LIp} with
$$
(u_1,B_1)
=
(u^{(1)}(t),B^{(1)}(t)),
\qquad
(u_2,B_2)
=
(u^{(2)}(t),B^{(2)}(t)),
$$
and $U=u_1-u_2$ and $V=B_1-B_2$,
we obtain
\begin{align}
&
\lVert \Sigma_1(t)\rVert_{L_2(\mathfrak U_1;L^2)}^2
+
\lVert \Sigma_2(t)\rVert_{L_2(\mathfrak U_2;L^2)}^2
\nonumber\\
&\qquad
\le
L_1
\left(
\lVert U(t)\rVert_{L^2}^2
+
\lVert V(t)\rVert_{L^2}^2
\right)
+
L_2
\left(
\nu\lVert \partial_1U(t)\rVert_{L^2}^2
+
\eta\lVert \nabla V(t)\rVert_{L^2}^2
\right).
\label{LSI}
\end{align}\medskip
\noindent With the functional setting and the assumptions on the noise coefficients
specified, we now turn to the  stochastic setting and notions of solution leads to construction of approximate solutions by
the Galerkin method.
\subsection{Stochastic setting and notions of solution}
\label{5GHT}

\noindent The Galerkin approximation described the following section \ref{sec 2.2} is being constructed using the standard
divergence-free Fourier basis on the periodic torus
$
\mathbb T^2=(\mathbb R/2\pi\mathbb Z)^2.
$ Let
$$
\mathbb Z^2\setminus\{(0,0)\}
=
\mathbb Z_+^2
\cup
\mathbb Z_-^2,
$$
where
$$
\mathbb Z_+^2
:=
\{
(k_1,k_2)\in\mathbb Z^2:
k_2>0
\}
\cup
\{
(k_1,0)\in\mathbb Z^2:
k_1>0
\},
\qquad
\mathbb Z_-^2
:=
\{
k\in\mathbb Z^2:
-k\in\mathbb Z_+^2
\}.
$$
For
$
k=(k_1,k_2)\in\mathbb Z^2\setminus\{(0,0)\},
$
define
$
k^\perp
:=
(-k_2,k_1),
$
so that
$
k\cdot k^\perp=0.
$ For each
$
k\in\mathbb Z^2\setminus\{(0,0)\},
$
define
$$
e_k(x)
=
\begin{cases}
\dfrac{1}{\sqrt{2}\,\pi}\,
\dfrac{k^\perp}{\lvert k\rvert}
\sin(k\cdot x),
&
k\in\mathbb Z_+^2,
\\[2ex]
\dfrac{1}{\sqrt{2}\,\pi}\,
\dfrac{k^\perp}{\lvert k\rvert}
\cos(k\cdot x),
&
k\in\mathbb Z_-^2.
\end{cases}
$$
where
$
k\cdot x
=
k_1x_1+k_2x_2.
$ Since
$
k^\perp\cdot k=0,
$
we obtain
$
\nabla\cdot e_k
=
0.
$
Moreover,
$
\int_{\mathbb T^2}
e_k(x)\,dx
=
0,
$
so that
$
e_k\in H.
$
Since $H$ does not impose a zero-mean condition, the normalized constant
divergence-free vector fields
$
\frac{1}{2\pi}(1,0),
\;
\frac{1}{2\pi}(0,1),
$
are also included.  Indeed, for
$
k\neq l,
$
we have
$
(e_k,e_l)=0,
$
$
(\partial_1e_k,\partial_1e_l)=0,
$
and
$
(\partial_2e_k,\partial_2e_l)=0.
$
Consequently,
for every
$
k\neq l
$
and
$
i=1,2,
$
integration by parts gives
$$
(\partial_i^2e_k,e_l)
=
-
(\partial_ie_k,\partial_ie_l)
=
0,
$$
which shows that
$
\partial_i^2e_k
$
is orthogonal to every
$
e_l,
\,
l\neq k.
$
Since
$
\{e_k\}_{k\ge1}
$
is an orthonormal basis of $H$, it follows that
$
\partial_i^2e_k
=
\lambda_{i,k}e_k
$
for some real constant
$
\lambda_{i,k}.
$ For the Fourier basis these eigenvalues are explicitly given by
$$
\lambda_{1,k}
=
-k_1^2,
\qquad
\lambda_{2,k},
=
-k_2^2$$
so that
$$
\partial_1^2e_k
=
-k_1^2e_k, \qquad
\partial_2^2e_k
=
-k_2^2e_k,
$$
 therefore
$$
\Delta e_k
=
-(k_1^2+k_2^2)e_k
=
-\lvert k\rvert^2e_k,
$$
so, each Fourier basis function is an eigenfunction of both
$\partial_1^2$
and
$\Delta$.
Consequently, if
$v\in H_n$,
then
$\partial_1^2v\in H_n$
and
$\Delta v\in H_n$.
 Denote by
$
P_n:H\longrightarrow H_n,
\,
P_n^{0,1}:H^{0,1}\longrightarrow H_n,
\,
P_n^{1,1}:H^{1,1}\longrightarrow H_n
$
the orthogonal projections with respect to the inner products of
$H$, $H^{0,1}$, and $H^{1,1}$, respectively. We first observe that
$$
P_nu=P_n^{0,1}u,
\qquad
u\in H^{0,1}.
$$
Indeed, let $u\in H^{0,1}$ and $v\in H_n$. Since
$
\partial_2^2v\in H_n,
$
the definition of the $H$-orthogonal projection gives
$
(P_nu,v)
=
(u,v).
$
Moreover, using integration by parts,
\begin{align*}
(\partial_2P_nu,\partial_2v)=
(\partial_2u,\partial_2v),
\end{align*}
which consequently,
\begin{align*}
(P_nu,v)_{H^{0,1}}
&=
(P_nu,v)
+
(\partial_2P_nu,\partial_2v)_{L^2}=
(u,v)_{H^{0,1}},
\qquad
v\in H_n,
\end{align*}
thus $P_nu$ is also the $H^{0,1}$-orthogonal projection
of $u$ onto $H_n$, and therefore
$$
P_nu=P_n^{0,1}u.
$$
Similarly, using
$
\partial_1^2v,\,
\partial_2^2v,\,
\partial_1^2\partial_2^2v
\in H_n,
$
one proves that
$$
P_nu
=
P_n^{0,1}u
=
P_n^{1,1}u,
\qquad
u\in H^{1,1}.
$$
Finally, we extend $P_n$ to $H^{-1}$ by duality. For
$f\in H^{-1}$, define
$$
P_nf
:=
\sum_{k=1}^n
\langle f,e_k\rangle_{H^{-1},H^1}\,e_k,
$$
where
$\langle\cdot,\cdot\rangle_{H^{-1},H^1}$
denotes the duality pairing between $H^{-1}$ and $H^1$.
\medskip

\noindent Let
$
(\Omega,\mathcal F,\mathbb P)
$
be a probability space, and let
$
W_1=(W_1(t))_{t\ge0},
\;
W_2=(W_2(t))_{t\ge0},
$
be two \emph{independent} cylindrical Wiener processes on the separable
Hilbert spaces $\mathfrak U_1$ and $\mathfrak U_2$, respectively,
defined on $(\Omega,\mathcal F,\mathbb P)$.
\medskip

\noindent Let
$
\{\psi_j^{(1)}\}_{j\ge1},
\;
\{\psi_j^{(2)}\}_{j\ge1},
$
be orthonormal bases of $\mathfrak U_1$ and $\mathfrak U_2$,
respectively. For each
$
i=1,2
$
and
$
n\geq1,
$
let
$
\Pi_n^{(i)}:\mathfrak U_i\to\mathfrak U_i
$
denote the orthogonal projection onto
$
\operatorname{span}
\{\psi_1^{(i)},\ldots,\psi_n^{(i)}\}.
$
It is pretty straightforward
\begin{align}
\Pi_n^{(i)}z
\longrightarrow
z
\qquad
\text{strongly in }\mathfrak U_i
\quad\text{for every }z\in\mathfrak U_i,
\qquad i=1,2.
\label{dsa}
\end{align}
\begin{align}
\lVert \Pi_n^{(i)}z\rVert_{\mathfrak U_i}
\leq
\lVert z\rVert_{\mathfrak U_i},
\qquad
\lVert (I-\Pi_n^{(i)})z\rVert_{\mathfrak U_i}
\leq
\lVert z\rVert_{\mathfrak U_i},
\qquad
z\in\mathfrak U_i.
\label{asd1}
\end{align}
With respect to these orthonormal bases, the cylindrical Wiener
processes $W_1$ and $W_2$ admit the representations
$$
W_i(t)
=
\sum_{j=1}^{\infty}
\psi_j^{(i)}\beta_j^{(i)}(t),
\qquad
i=1,2,
$$
where
$
\{\beta_j^{(i)}\}_{j\ge1},
\; i=1,2,
$
are standard Brownian motions on
$(\Omega,\mathcal F,\mathbb P)$, and the two families
$
\{\beta_j^{(1)}\}_{j\ge1}
$
and
$
\{\beta_j^{(2)}\}_{j\ge1}
$
are independent. For
$
i=1,2
$
and
$
n\ge1,
$
define
$$
W_i^n(t)
:=
\Pi_n^{(i)}W_i(t)
=
\sum_{j=1}^{n}
\psi_j^{(i)}\beta_j^{(i)}(t).
$$
For a Polish space
$
\mathbb V,
$
let
$
\mathcal B(\mathbb V)
$
denote its Borel $\sigma$-algebra, and let
$
\mathcal P(\mathbb V)
$
denote the set of all probability measures on
$
(\mathbb V,\mathcal B(\mathbb V)).
$

\noindent
For notational convenience, we introduce the deterministic operators
$
F_1,F_2:H^{1,1}\times H^{1,1}\rightarrow H^{-1}
$
defined by
\begin{align}
F_1(u,B)
&=
-(u\cdot\nabla)u
+
(B\cdot\nabla)B
+
\nu\partial_1^2u,
\label{def-F1}
\\
F_2(u,B)
&=
-(u\cdot\nabla)B
+
(B\cdot\nabla)u
+
\eta\Delta B.
\label{def-F2}
\end{align}
We next define a probabilistically weak solution.
\begin{definition}[Probabilistically weak solution]
\label{d-w-s}
We say that the pair
$
\bigl((u,B),(W_1,W_2)\bigr)
$
is a \emph{(probabilistically) weak solution} of the stochastic
anisotropic MHD system if there exists a stochastic basis
$
(\Omega,\mathcal F,\{\mathcal F_t\}_{t\ge0},\mathbb P)
$
satisfying the usual conditions such that
$
u=(u(t))_{t\ge0},
\;
B=(B(t))_{t\ge0},
$
are $\{\mathcal F_t\}_{t\ge0}$-progressively measurable processes,
$W_1$ and $W_2$ are independent cylindrical Wiener processes on
$\mathfrak U_1$ and $\mathfrak U_2$, respectively, defined on
$
(\Omega,\mathcal F,\{\mathcal F_t\}_{t\ge0},\mathbb P),
$
and the following conditions hold:

\begin{enumerate}

\item[(i)]
We have
\begin{align*}
u
&\in
L^\infty_{\mathrm{loc}}
\bigl(\mathbb R_+; H^{0,1}\bigr)
\cap
L^2_{\mathrm{loc}}
\bigl(\mathbb R_+; H^{1,1}\bigr)
\cap
C\bigl(\mathbb R_+;H^{-1}\bigr),
\nonumber
\\
B
&\in
L^\infty_{\mathrm{loc}}
\bigl(\mathbb R_+; H^{0,1}\bigr)
\cap
L^2_{\mathrm{loc}}
\bigl(\mathbb R_+; H^{1}\bigr)
\cap
L^2_{\mathrm{loc}}
\bigl(\mathbb R_+; H^{1,1}\bigr)
\cap
L^2_{\mathrm{loc}}
\bigl(\mathbb R_+; H^{0,2}\bigr)
\cap
C\bigl(\mathbb R_+;H^{-1}\bigr),
\end{align*}
$\mathbb P$-a.s..

\item[(ii)]
For every $T>0$,
\begin{align*}
&\int_0^T
\left\lVert 
F_1\bigl(u(s),B(s)\bigr)
\right\rVert_{H^{-1}}\,ds
+
\int_0^T
\left\lVert 
F_2\bigl(u(s),B(s)\bigr)
\right\rVert_{H^{-1}}\,ds
+
\int_0^T
\left\lVert 
\sigma_1\bigl(s,u(s),B(s)\bigr)
\right\rVert_{L_2(\mathfrak U_1;H)}^2\,ds
\nonumber\\
&\quad
+
\int_0^T
\left\lVert 
\sigma_2\bigl(s,u(s),B(s)\bigr)
\right\rVert_{L_2(\mathfrak U_2;H)}^2\,ds
<\infty,
\qquad
\mathbb P\text{-a.s.},
\end{align*}
where
\begin{align*}
F_1(u,B)
&=
\nu\partial_1^2u
-
(u\cdot\nabla)u
+
(B\cdot\nabla)B,
\nonumber\\
F_2(u,B)
&=
\eta\Delta B
-
(u\cdot\nabla)B
+
(B\cdot\nabla)u.
\end{align*}

\item[(iii)]
For every
$
\phi,\psi
\in
C^\infty
$
satisfying
$
\nabla\cdot\phi=0,
\quad
\nabla\cdot\psi=0,
$
we have, $\mathbb P$-a.s.,
$
u(0)=u_0,
\,
B(0)=B_0,$
and, for every $t\ge0$,
\begin{align}
\langle u(t),\phi\rangle
&=
\langle u_0,\phi\rangle
+
\int_0^t
\left\langle
F_1\bigl(u(s),B(s)\bigr),
\phi
\right\rangle\,ds
+
\int_0^t
\langle
\sigma_1\bigl(s,u(s),B(s)\bigr)\,dW_1(s),
\phi
\rangle,
\label{PrW1}
\\
\langle B(t),\psi\rangle
&=
\langle B_0,\psi\rangle
+
\int_0^t
\langle
F_2\bigl(u(s),B(s)\bigr),
\psi
\rangle\,ds
+
\int_0^t
\langle
\sigma_2\bigl(s,u(s),B(s)\bigr)\,dW_2(s),
\psi
\rangle.
\label{PrW2}
\end{align}
\end{enumerate}
Here
$\langle\cdot,\cdot\rangle$
denotes the duality pairing between $H^{-1}$ and $H^1$;
when $u,v\in H$,
$
\langle u,v\rangle=(u,v).
$
\end{definition}
\noindent We next introduce the filtration used in the martingale
formulation. On the path space $\bar\Omega$, for $t\ge0$,
define
\begin{align}
\mathcal F_t
:=
\sigma\left(
u(s),B(s):0\leq s\leq t
\right).
\label{MF1}
\end{align}
\begin{definition}
\label{DEF}[(Global) martingale solution]
A probability measure
$
\mathbb P\in\mathcal P(\bar\Omega)
$
is called a \emph{global martingale solution}
of the stochastic anisotropic MHD system with initial data
$
(u_0,B_0),
$
if

\begin{enumerate}

\item[(M1)]
$
\mathbb P
\bigl(
(u(0),B(0))
=
(u_0,B_0)
\bigr)
=
1,
$
and, for every $T>0$,
$$
\mathbb P
\left(
\sup_{0\le t\le T}
\lVert u(t)\rVert_{H^{0,1}}
<\infty,
\;
\int_0^T
\lVert u(t)\rVert_{H^{1,1}}^2\,dt
<\infty
\right)
=
1,
$$
and
$$
\mathbb P
\left(
\sup_{0\le t\le T}
\lVert B(t)\rVert_{H^{0,1}}
<\infty,
\;
\int_0^T
\lVert B(t)\rVert_{H^{0,2}}^2\,dt
<\infty,
\;
\int_0^T
\lVert B(t)\rVert_{H^{1,1}}^2\,dt
<\infty
\right)
=
1,
$$
together with
$$
\mathbb P
\left(
\int_0^T
\Big(
\lVert F_1(u,B)\rVert_{H^{-1}}
+
\lVert F_2(u,B)\rVert_{H^{-1}}
\Big)\,dt
<\infty
\right)
=
1,
$$
and
$$
\mathbb P
\left(
\int_0^T
\left(
\lVert \sigma_1(t,u,B)\rVert_{L_2(\mathfrak U_1;H)}^2
+
\lVert \sigma_2(t,u,B)\rVert_{L_2(\mathfrak U_2;H)}^2
\right)\,dt
<\infty
\right)
=
1.
$$

\item[(M2)]
For every divergence-free
$
\phi,\psi\in C^\infty,
$
the processes
$$
M_{u}^{\phi}(t)
=
\langle u(t),\phi\rangle
-
\langle u_0,\phi\rangle
-
\int_0^t
\langle F_1(u(s),B(s)),\phi\rangle\,ds,
$$
and
$$
M_B^{\psi}(t)
=
\langle B(t),\psi\rangle
-
\langle B_0,\psi\rangle
-
\int_0^t
\langle F_2(u(s),B(s)),\psi\rangle\,ds,
$$
are continuous square-integrable
$(\mathcal F_t)$-martingales
whose quadratic variation processes are
\begin{align}
\left\langle\!\left\langle
M_u^\phi
\right\rangle\!\right\rangle_t
&=
\int_0^t
\left\lVert 
\sigma_1^*(s,u(s),B(s))(\phi)
\right\rVert_{\mathfrak U_1}^2\,ds,
\label{M2DV}
\\
\left\langle\!\left\langle
M_B^\psi
\right\rangle\!\right\rangle_t
&=
\int_0^t
\left\lVert 
\sigma_2^*(s,u(s),B(s))(\psi)
\right\rVert_{\mathfrak U_2}^2\,ds,
\label{M2DB}
\end{align}
respectively. Moreover, their cross-variation vanishes:
\begin{align}
\left\langle\!\left\langle
M_u^\phi,M_B^\psi
\right\rangle\!\right\rangle_t
=
0,
\qquad
t\geq0.
\label{M2qb}
\end{align}

\item[(M3)]
For every $T>0$,
$$
\mathbb E
\left[
\sup_{0\le t\le T}
\left(
\lVert u(t)\rVert_{L^2}^2
+
\lVert B(t)\rVert_{L^2}^2
\right)
+
\int_0^T
\left(
\nu\lVert \partial_1u(t)\rVert_{L^2}^2
+
\eta\lVert \nabla B(t)\rVert_{L^2}^2
\right)\,dt
\right]
<\infty.
$$
\end{enumerate}
\end{definition}
We next introduce the Galerkin approximation used in the construction of a martingale solution.
\subsection{Galerkin approximation}
\label{sec 2.2}

\noindent For the Galerkin construction, let
$
\{\mathcal F_t\}_{t\geq0}
$
be an increasing, complete, and right-continuous filtration on
$
(\Omega,\mathcal F,\mathbb P)
$
such that
$
W_1
$
and
$
W_2
$
are independent cylindrical Wiener processes with respect to
$
\{\mathcal F_t\}_{t\geq0}.
$ Let
$
H_n=\operatorname{span}\{e_1,\ldots,e_n\},
$
and let
$
P_n:H\to H_n
$
denote the orthogonal projection onto $H_n$. Since $H_n$ is finite
dimensional, every $H_n$-valued process admits a unique expansion with
respect to the basis
$
\{e_1,\ldots,e_n\}.
$
We therefore seek approximate solutions
$
u^n
$
and
$
B^n
$
of the form
\begin{align}
u^n(t)
&=
\sum_{k=1}^{n}
a_k^n(t)e_k,
\label{GAU}
\\
B^n(t)
&=
\sum_{k=1}^{n}
b_k^n(t)e_k,
\label{GAB}
\end{align}
where
$
a_k^n
$
and
$
b_k^n
$
are real-valued
$
\{\mathcal F_t\}_{t\geq0}
$-adapted processes.
\medskip

\noindent 
Let
$
\Pi_n^{(i)}:\mathfrak U_i\to\mathfrak U_i
$
denote the orthogonal projection onto the finite-dimensional subspace
generated by the first $n$ elements of a fixed orthonormal basis of
$
\mathfrak U_i,
$
and set
$
W_i^n=\Pi_n^{(i)}W_i.
$
Then
$
dW_i^n=\Pi_n^{(i)}\,dW_i,
$
for $i=1,2$. The Galerkin approximation of the stochastic MHD system \eqref{MHD} is given by
\begin{align}
du^n
&=
\Bigl[
\nu P_n\partial_1^2u^n
-
P_n\bigl((u^n\cdot\nabla)u^n\bigr)
+
P_n\bigl((B^n\cdot\nabla)B^n\bigr)
\Bigr]\,dt
+
P_n\sigma_1(t,u^n,B^n)\Pi_n^{(1)}\,dW_1,
\label{gVT}
\\
dB^n
&=
\Bigl[
\eta P_n\Delta B^n
-
P_n\bigl((u^n\cdot\nabla)B^n\bigr)
+
P_n\bigl((B^n\cdot\nabla)u^n\bigr)
\Bigr]\,dt
+
P_n\sigma_2(t,u^n,B^n)\Pi_n^{(2)}\,dW_2,
\label{GVL}
\end{align}
with initial conditions
\begin{align}
u^n(0)
=
P_nu_0,
\qquad
B^n(0)
=
P_nB_0.
\label{GIL}
\end{align}
It follows that
\begin{align}
P_nF_1(u^n,B^n)
&=
\nu P_n\partial_1^2u^n
-
P_n\bigl((u^n\cdot\nabla)u^n\bigr)
+
P_n\bigl((B^n\cdot\nabla)B^n\bigr),
\label{KL4}
\\
P_nF_2(u^n,B^n)
&=
\eta P_n\Delta B^n
-
P_n\bigl((u^n\cdot\nabla)B^n\bigr)
+
P_n\bigl((B^n\cdot\nabla)u^n\bigr).
\label{P2K}
\end{align}
Therefore,
\eqref{gVT}-\eqref{GVL}
may equivalently be written as
\begin{align}
du^n
&=
P_nF_1(u^n,B^n)\,dt
+
P_n\sigma_1(t,u^n,B^n)\Pi_n^{(1)}\,dW_1,
\label{HGS}
\\
dB^n
&=
P_nF_2(u^n,B^n)\,dt
+
P_n\sigma_2(t,u^n,B^n)\Pi_n^{(2)}\,dW_2.
\label{qwe}
\end{align}
The Galerkin solution
$
(u^n,B^n)
$
is an
$
\{\mathcal F_t\}_{t\geq0}
$-adapted
$
H_n\times H_n
$-valued process with continuous sample paths. Consequently,
$
u^n
$
and
$
B^n
$
are predictable processes. Since
$
\sigma_1
$
and
$
\sigma_2
$
are measurable mappings by Hypothesis~\ref{HYp1}, it follows that
$
\sigma_i\bigl(t,u^n(t),B^n(t)\bigr),
$
$i=1,2$, are predictable
$
L_2(\mathfrak U_i;H^{1,1})
$-valued processes.
\medskip

\noindent In addition, corresponding to the filtration $\{\mathcal F\}_{t\geq 0}$ defined in \eqref{MF1}, for $t\geq0$, we define the  filtration generated by the
Galerkin processes by
\begin{align}
\mathcal F_t^n
:=
\sigma\left(
u^n(s),B^n(s):0\leq s\leq t
\right).
\label{FI}
\end{align}
Since $u^n$ and $B^n$ are
$\{\mathcal F_t\}_{t\geq0}$-adapted, it follows from \eqref{MF1} and \eqref{FI} that
\begin{align}
\mathcal F_t^n
\subseteq
\mathcal F_t,
\;
t\geq0.
\label{lka0}
\end{align}
\noindent
We first derive the following a priori estimates, which will also be used in the stochastic analysis.
\begin{lemma}
Let $(u_0,B_0)\in H^{0,1}\times H^{0,1}$. Then the solution
$(u,B)$ to \eqref{DMHDu}-\eqref{DMHD} satisfies
\begin{align}
&
\lVert \partial_2u(t)\rVert_{L^2}^2
+
\lVert \partial_2B(t)\rVert_{L^2}^2
+
\int_0^t
\Big(
\nu\lVert \partial_1\partial_2u(s)\rVert_{L^2}^2
+
\eta\lVert \nabla\partial_2B(s)\rVert_{L^2}^2
\Big)\,ds
\nonumber\\
&\qquad\leq
\exp\left\{
C_{\nu,\eta}
\int_0^t
\Big(
\lVert \partial_1u(s)\rVert_{L^2}
+
\lVert \partial_1u(s)\rVert_{L^2}^2
+
\lVert \nabla B(s)\rVert_{L^2}
+
\lVert \nabla B(s)\rVert_{L^2}^2
\Big)\,ds
\right\}
\nonumber\\
&\qquad\quad\times
\Big(
\lVert \partial_2u_0\rVert_{L^2}^2
+
\lVert \partial_2B_0\rVert_{L^2}^2
\Big).
\label{0.49A}
\end{align}
\end{lemma}
\begin{proof}
Taking the $L^2$ inner products of \eqref{DMHDu} with $u$ and
\eqref{DMHDB} with $B$, respectively, and adding the resulting identities
gives
\begin{align}
\frac{1}{2}\frac{d}{dt}
\left(
\lVert u\rVert_{L^2}^2
+
\lVert B\rVert_{L^2}^2
\right)
+
\nu\lVert\partial_1u\rVert_{L^2}^2
+
\eta\lVert\nabla B\rVert_{L^2}^2
=0.
\label{2.28H}
\end{align}
 We next apply $\partial_2$ to \eqref{DMHDu} and \eqref{DMHDB},
and then taking the $L^2$-inner products of the resulting equations
with $\partial_2u$ and $\partial_2B$, respectively, to obtain
\begin{align}
\frac{1}{2}\frac{d}{dt}\Big(
\lVert \partial_2u(t)\rVert_{L^2}^2
+
\lVert \partial_2B(t)\rVert_{L^2}^2
\Big)
&+
\nu\lVert \partial_1\partial_2u\rVert_{L^2}^2
+
\eta\lVert \partial_1\partial_2B\rVert_{L^2}^2=I1+I2+I3+I4
\end{align}
It is easy to observe that 
\begin{align}
I1=(\partial_2(u \cdot \nabla u)\rvert \partial_2 u)=(\partial_2(u \cdot \nabla u^1)\rvert \partial_2u^1)+(\partial_2(u \cdot \nabla u^2)\rvert \partial_2u^2),
\label{A1}
\end{align}
where the 1st term on the right hand side of  \eqref{A1} is simplified as follows
\begin{align}
(\partial_2(u \cdot \nabla u^1)\rvert \partial_2u^1)&= (\partial_2 (u^1 \partial_1 u^1+u^2 \partial_2 u^1)\rvert \partial_2u^1)
\nonumber\\
& = (\partial_2 u^1 \partial_1 u^1\rvert \partial_2 u^1)+( u^1 \partial_2 \partial_1 u^1\rvert \partial_2 u^1)+ (\partial_2  u^2 \partial_2 u^1\rvert \partial_2 u^1)+(  u^2 \partial_2 \partial_2 u^1\rvert \partial_2 u^1)
\nonumber\\
&=0,
\label{A2}
\end{align}
where we used $\nabla \cdot u=0$ and $
 (\partial_2  u^1 \partial_1 u^1\rvert \partial_2 u^1)+ (\partial_2  u^2 \partial_2 u^1\rvert \partial_2 u^1)=0$.
Applying H\"older's inequality,  anisotropic inequality
\cite[Lemma~3.4]{liang2021deterministic},
\begin{equation}
\begin{aligned}
&\lVert f\rVert_{L_v^2(L_h^\infty)}^2
\lesssim
\lVert f\rVert_{L^2}\lVert\partial_1f\rVert_{L^2}
+\lVert f\rVert_{L^2}^2,
\\
&\lVert f\rVert_{L_h^2(L_v^\infty)}^2
\lesssim
\lVert f\rVert_{L^2}\lVert\partial_2f\rVert_{L^2}
+\lVert f\rVert_{L^2}^2,
\label{Lia}
\end{aligned}
\end{equation} Integration by parts, and $\nabla \cdot u=0$, we move to work out the 2nd term on \eqref{A1} in the following 
\begin{align}
(\partial_2(u \cdot \nabla u^2)\rvert \partial_2u^2)&= (\partial_2 (u^1 \partial_1 u^2+u^2 \partial_2 u^2)\rvert \partial_2u^2)
\nonumber\\
& = (\partial_2 u^1 \partial_1 u^2\rvert \partial_2 u^2)+( u^1 \partial_2 \partial_1 u^2\rvert \partial_2 u^2)+ (\partial_2  u^2 \partial_2 u^2\rvert \partial_2 u^2)+(  u^2 \partial_2 \partial_2 u^2\rvert \partial_2 u^2)
\nonumber\\
&= (\partial_2 u^1 \partial_1 u^2\rvert \partial_2 u^2)+ (\partial_2  u^2 \partial_2 u^2\rvert \partial_2 u^2)-\frac{1}{2} \int_{D}(\nabla \cdot u)\lvert \partial_2 u^2\rvert^2
\nonumber\\
& \leq (\lVert \partial_2 u^1\rVert_{L_h^{\infty}{L_v^2}} \lVert \partial_1 u^2\rVert_{L_h^{2}L_v^{\infty}}+\lVert \partial_2 u^2\rVert_{L_h^{\infty}{L_v^2}} \lVert \partial_2 u^2\rVert_{L_h^{2}L_v^{\infty}}) \lVert \partial_2 u^2\rVert_{L^2}
\nonumber\\
& \leq (\lVert \partial_2 u^1\rVert_{L_h^{\infty}{L_v^2}} \lVert \partial_1 u^2\rVert_{L_h^{2}L_v^{\infty}}+\lVert \partial_2 u^2\rVert_{L_h^{\infty}{L_v^2}} \lVert \partial_1 u^1\rVert_{L_h^{2}L_v^{\infty}}) \lVert \partial_2 u^2\rVert_{L^2}
\label{A3}
\end{align} Combining \eqref{A1}, \eqref{A2} and \eqref{A3} implies
\begin{align}
I1 & \lesssim (\lVert \partial_2 u\rVert_{L^2}+\lVert \partial_2 u\rVert_{L^2}^{\frac{1}{2}}\lVert \partial_1 \partial_2 u\rVert_{L^2}^{\frac{1}{2}})(\lVert \partial_1 u\rVert_{L^2}+\lVert \partial_1 u\rVert_{L^2}^{\frac{1}{2}}\lVert \partial_1 \partial_2 u\rVert_{L^2}^{\frac{1}{2}})\lVert \partial_2 u^2\rVert_{L^2}
\nonumber\\
& =\lVert \partial_2 u\rVert_{L^2} \lVert \partial_1 u\rVert_{L^2} \lVert \partial_2 u^2\rVert_{L^2}+ \lVert \partial_2 u\rVert_{L^2}^{\frac{1}{2}}\lVert \partial_1 \partial_2 u\rVert_{L^2}^{\frac{1}{2}} \lVert \partial_1 u\rVert_{L^2}\lVert \partial_2 u^2\rVert_{L^2}
\nonumber\\
&+\lVert \partial_2 u\rVert_{L^2}\lVert \partial_1 u\rVert_{L^2}^{\frac{1}{2}}\lVert \partial_1 \partial_2 u\rVert_{L^2}^{\frac{1}{2}}\lVert \partial_2 u^2\rVert_{L^2}+\lVert \partial_2 u\rVert_{L^2}^{\frac{1}{2}}\lVert \partial_1 \partial_2 u\rVert_{L^2}^{\frac{1}{2}}\lVert \partial_1 u\rVert_{L^2}^{\frac{1}{2}}
\lVert \partial_1 \partial_2 u\rVert_{L^2}^{\frac{1}{2}}\lVert \partial_2 u^2\rVert_{L^2}
\nonumber\\
& \leq \lVert \partial_2 u\rVert_{L^2} \lVert \partial_1 u\rVert_{L^2} \lVert \partial_2 u^2\rVert_{L^2}+\lVert \partial_2 u\rVert_{L^2}^{\frac{1}{2}}\lVert \partial_1 \partial_2 u\rVert_{L^2}^{\frac{1}{2}} \lVert \partial_1 u\rVert_{L^2}\lVert \partial_2 u^2\rVert_{L^2}
\nonumber\\
&+\lVert \partial_2 u\rVert_{L^2}\lVert \partial_1 u\rVert_{L^2}^{\frac{1}{2}}\lVert \partial_1 \partial_2 u\rVert_{L^2}^{\frac{1}{2}}\lVert \partial_2 u^2\rVert_{L^2}+\lVert \partial_2 u\rVert_{L^2}\lVert \partial_1 \partial_2 u\rVert_{L^2} \lVert \partial_1 u\rVert_{L^2}
\label{A4}
\end{align}
Applying Young's inequality
and using 
$$
\lVert \partial_1 u\rVert_{L^2}^{\frac{4}{3}}
 \leq \lVert \partial_1 u\rVert_{L^2}+\lVert \partial_1 u\rVert_{L^2}^{2},$$ we deduce from 
\eqref{A4} that
\begin{align}
I1&\leq \lVert \partial_2 u\rVert_{L^2}^{2} \lVert \partial_1 u\rVert_{L^2}+ \frac{1}{8}\lVert \partial_1 \partial_2 u\rVert_{L^2}^{2}+C \lVert \partial_1 u\rVert_{L^2}^{\frac{4}{3}}\lVert \partial_2 u\rVert_{L^2}^{\frac{6}{3}}
\nonumber\\
&+\frac{1}{8}\lVert \partial_1 \partial_2 u\rVert_{L^2}^{2}+C \lVert \partial_2 u\rVert_{L^2}^{2}\lVert \partial_1 u\rVert_{L^2}^{\frac{4}{3}}+ \frac{1}{8}\lVert \partial_1 \partial_2 u\rVert_{L^2}^{2}+C \lVert \partial_1 u\rVert_{L^2}^{2} \lVert \partial_2 u\rVert_{L^2}^{2}
\nonumber\\
& \leq \frac{1}{2}\lVert \partial_1 \partial_2 u\rVert_{L^2}^{2}+ C(\lVert \partial_1 u\rVert_{L^2}+\lVert \partial_1 u\rVert_{L^2}^{2}+\lVert \partial_1 u\rVert_{L^2}^{\frac{4}{3}})\lVert \partial_2 u\rVert_{L^2}^{2}
 \nonumber\\
& \leq \frac{1}{2}\lVert \partial_1 \partial_2 u\rVert_{L^2}^{2}+ C(\lVert \partial_1 u\rVert_{L^2}+\lVert \partial_1 u\rVert_{L^2}^{2})\lVert \partial_2 u\rVert_{L^2}^{2}.
\label{A5}
\end{align}
Applying H\"older's  inequality, \eqref{Lia}, and Young's inequality, we obtain
\begin{align}
\lvert I_2\rvert
&\lesssim
\lVert \nabla B\rVert_{L^2}
\Big(
\lVert \partial_2u\rVert_{L^2}
+
\lVert \partial_2u\rVert_{L^2}^{\frac12}
\lVert \partial_1\partial_2u\rVert_{L^2}^{\frac12}
\Big)
\Big(
\lVert \partial_2B\rVert_{L^2}
+
\lVert \partial_2B\rVert_{L^2}^{\frac12}
\lVert \partial_2^2B\rVert_{L^2}^{\frac12}
\Big)
\nonumber\\
&\lesssim
\lVert \nabla B\rVert_{L^2}
\lVert \partial_2u\rVert_{L^2}
\lVert \partial_2B\rVert_{L^2}
+
\lVert \nabla B\rVert_{L^2}
\lVert \partial_2u\rVert_{L^2}
\lVert \partial_2B\rVert_{L^2}^{\frac12}
\lVert \partial_2^2B\rVert_{L^2}^{\frac12}
\nonumber\\
&\quad
+
\lVert \nabla B\rVert_{L^2}
\lVert \partial_2u\rVert_{L^2}^{\frac12}
\lVert \partial_1\partial_2u\rVert_{L^2}^{\frac12}
\lVert \partial_2B\rVert_{L^2}
\nonumber\\
&\quad
+
\lVert \nabla B\rVert_{L^2}
\lVert \partial_2u\rVert_{L^2}^{\frac12}
\lVert \partial_1\partial_2u\rVert_{L^2}^{\frac12}
\lVert \partial_2B\rVert_{L^2}^{\frac12}
\lVert \partial_2^2B\rVert_{L^2}^{\frac12}
\nonumber\\
&\leq
\frac{\nu}{4}
\lVert \partial_1\partial_2u\rVert_{L^2}^{2}
+
\frac{\eta}{4}
\lVert \nabla\partial_2B\rVert_{L^2}^{2}
\nonumber\\
&\quad
+
C
\Big(
\lVert \nabla B\rVert_{L^2}
+
\lVert \nabla B\rVert_{L^2}^{2}
\Big)
\Big(
\lVert \partial_2u\rVert_{L^2}^{2}
+
\lVert \partial_2B\rVert_{L^2}^{2}
\Big),
\label{A6}
\end{align}
Applying H\"older's  inequality, \eqref{Lia}, and Young's inequality, we obtain
\begin{align}
\lvert I_3\rvert
&\lesssim
\lVert \nabla B\rVert_{L^2}
\lVert \partial_2u\rVert_{L_h^\infty L_v^2}
\lVert \partial_2B\rVert_{L_h^2L_v^\infty}
\nonumber\\
&\lesssim
\lVert \nabla B\rVert_{L^2}
\Big(
\lVert \partial_2u\rVert_{L^2}
+
\lVert \partial_2u\rVert_{L^2}^{\frac12}
\lVert \partial_1\partial_2u\rVert_{L^2}^{\frac12}
\Big)
\Big(
\lVert \partial_2B\rVert_{L^2}
+
\lVert \partial_2B\rVert_{L^2}^{\frac12}
\lVert \partial_2^2B\rVert_{L^2}^{\frac12}
\Big)
\nonumber\\
&\leq
\frac{\nu}{4}
\lVert \partial_1\partial_2u\rVert_{L^2}^2
+
\frac{\eta}{4}
\lVert \partial_2^2B\rVert_{L^2}^2
\nonumber\\
&\quad
+
C\lVert \nabla B\rVert_{L^2}
\Big(
\lVert \partial_2u\rVert_{L^2}^2
+
\lVert \partial_2B\rVert_{L^2}^2
\Big)
\nonumber\\
&\quad
+
C\lVert \nabla B\rVert_{L^2}^{\frac43}
\left[
\frac23
\left(
\lVert \partial_2u\rVert_{L^2}^{\frac43}
\right)^{\frac32}
+
\frac13
\left(
\lVert \partial_2B\rVert_{L^2}^{\frac23}
\right)^3
\right]
\nonumber\\
&\quad
+
C\lVert \nabla B\rVert_{L^2}^{\frac43}
\left[
\frac13
\left(
\lVert \partial_2u\rVert_{L^2}^{\frac23}
\right)^3
+
\frac23
\left(
\lVert \partial_2B\rVert_{L^2}^{\frac43}
\right)^{\frac32}
\right]
\nonumber\\
&\quad
+
\frac{C}{2}
\lVert \nabla B\rVert_{L^2}^2
\Big(
\lVert \partial_2u\rVert_{L^2}^2
+
\lVert \partial_2B\rVert_{L^2}^2
\Big)
\nonumber\\
&\leq
\frac{\nu}{4}
\lVert \partial_1\partial_2u\rVert_{L^2}^2
+
\frac{\eta}{4}
\lVert \nabla\partial_2B\rVert_{L^2}^2
\nonumber\\
&\quad
+
C
\Big(
\lVert \nabla B\rVert_{L^2}
+
\lVert \nabla B\rVert_{L^2}^2
\Big)
\Big(
\lVert \partial_2u\rVert_{L^2}^2
+
\lVert \partial_2B\rVert_{L^2}^2
\Big).
\label{A7}
\end{align}
Applying H\"older's  inequality, \eqref{Lia}, and Young's inequality, we obtain
\begin{align}
\lvert I_4\rvert
&=
2\left\lvert
\big((\partial_2B\cdot\nabla)u,\partial_2B\big)
\right\rvert
\nonumber\\
&\lesssim
\Big(
\lVert \partial_2B\rVert_{L_h^\infty L_v^2}
\lVert \partial_1u\rVert_{L_h^2L_v^\infty}
+
\lVert \partial_2u\rVert_{L_h^\infty L_v^2}
\lVert \partial_1B\rVert_{L_h^2L_v^\infty}
\Big)
\lVert \partial_2B\rVert_{L^2}
\nonumber\\
&\lesssim
\Big(
\lVert \partial_2B\rVert_{L^2}
+
\lVert \partial_2B\rVert_{L^2}^{\frac12}
\lVert \partial_1\partial_2B\rVert_{L^2}^{\frac12}
\Big)
\Big(
\lVert \partial_1u\rVert_{L^2}
+
\lVert \partial_1u\rVert_{L^2}^{\frac12}
\lVert \partial_1\partial_2u\rVert_{L^2}^{\frac12}
\Big)
\lVert \partial_2B\rVert_{L^2}
\nonumber\\
&\quad
+
\Big(
\lVert \partial_2u\rVert_{L^2}
+
\lVert \partial_2u\rVert_{L^2}^{\frac12}
\lVert \partial_1\partial_2u\rVert_{L^2}^{\frac12}
\Big)
\Big(
\lVert \partial_1B\rVert_{L^2}
+
\lVert \partial_1B\rVert_{L^2}^{\frac12}
\lVert \partial_1\partial_2B\rVert_{L^2}^{\frac12}
\Big)
\lVert \partial_2B\rVert_{L^2}
\nonumber\\
&\lesssim
\lVert \partial_1u\rVert_{L^2}
\lVert \partial_2B\rVert_{L^2}^2
+
\lVert \partial_1u\rVert_{L^2}^{\frac12}
\lVert \partial_1\partial_2u\rVert_{L^2}^{\frac12}
\lVert \partial_2B\rVert_{L^2}^2
\nonumber\\
&\quad
+
\lVert \partial_1u\rVert_{L^2}
\lVert \partial_2B\rVert_{L^2}^{\frac32}
\lVert \partial_1\partial_2B\rVert_{L^2}^{\frac12}
+
\lVert \partial_1u\rVert_{L^2}^{\frac12}
\lVert \partial_1\partial_2u\rVert_{L^2}^{\frac12}
\lVert \partial_2B\rVert_{L^2}^{\frac32}
\lVert \partial_1\partial_2B\rVert_{L^2}^{\frac12}
\nonumber\\
&\quad
+
\lVert \partial_1B\rVert_{L^2}
\lVert \partial_2u\rVert_{L^2}
\lVert \partial_2B\rVert_{L^2}
+
\lVert \partial_1B\rVert_{L^2}
\lVert \partial_2u\rVert_{L^2}^{\frac12}
\lVert \partial_1\partial_2u\rVert_{L^2}^{\frac12}
\lVert \partial_2B\rVert_{L^2}
\nonumber\\
&\quad
+
\lVert \partial_1B\rVert_{L^2}^{\frac12}
\lVert \partial_1\partial_2B\rVert_{L^2}^{\frac12}
\lVert \partial_2u\rVert_{L^2}
\lVert \partial_2B\rVert_{L^2}
\nonumber\\
&\quad
+
\lVert \partial_1B\rVert_{L^2}^{\frac12}
\lVert \partial_1\partial_2B\rVert_{L^2}^{\frac12}
\lVert \partial_2u\rVert_{L^2}^{\frac12}
\lVert \partial_1\partial_2u\rVert_{L^2}^{\frac12}
\lVert \partial_2B\rVert_{L^2}
\label{A8}
\end{align}
Since the elementary  inequality
$x^{\frac23}\leq 1+x,
\, x\geq0,$ implies
$$
\Big(
\lVert \partial_1u\rVert_{L^2}
\lVert \partial_2B\rVert_{L^2}
\Big)^{\frac23}
\leq
1+
\lVert \partial_1u\rVert_{L^2}
\lVert \partial_2B\rVert_{L^2},
$$ therefore which yields together with  \eqref{A8} as follows
\begin{align}
\lvert I_4\rvert
&\leq
\frac{\nu}{4}
\lVert \partial_1\partial_2u\rVert_{L^2}^{2}
+
\frac{\eta}{4}
\lVert \partial_1\partial_2B\rVert_{L^2}^{2}
\nonumber\\
&\quad
+
C_{\nu,\eta}
\Big(
\lVert \partial_1u\rVert_{L^2}
+
\lVert \partial_1u\rVert_{L^2}^{2}
+
\lVert \partial_1B\rVert_{L^2}
+
\lVert \partial_1B\rVert_{L^2}^{2}+\lVert \partial_2B\rVert_{L^2}^{2}
\Big)
\Big(
\lVert \partial_2u\rVert_{L^2}^{2}
+
\lVert \partial_2B\rVert_{L^2}^{2}
\Big)
.
\label{A9}
\end{align}
Now it follows from \eqref{A5}-\eqref{A9} that
\begin{align}
\frac{d}{dt}\Big(
\lVert \partial_2u(t)\rVert_{L^2}^2
+
\lVert \partial_2B(t)\rVert_{L^2}^2
\Big)
&+
\nu\lVert \partial_1\partial_2u(t)\rVert_{L^2}^2
+
\eta\lVert \nabla\partial_2B(t)\rVert_{L^2}^2
\nonumber\\
&\leq
C_{\nu,\eta}
\Big(
\lVert \partial_1u(t)\rVert_{L^2}
+
\lVert \partial_1u(t)\rVert_{L^2}^2
+
\lVert \nabla B(t)\rVert_{L^2}
+
\lVert \nabla B(t)\rVert_{L^2}^2
\Big)
\nonumber\\
&\qquad\times
\Big(
\lVert \partial_2u(t)\rVert_{L^2}^2
+
\lVert \partial_2B(t)\rVert_{L^2}^2
\Big),
\label{s4vzq}
\end{align}
where the constant $C_{\nu,\eta}$.
Therefore, by Gr\"onwall's inequality, \eqref{s4vzq} gives
\begin{align}
&
\lVert \partial_2u(t)\rVert_{L^2}^2
+
\lVert \partial_2B(t)\rVert_{L^2}^2
+
\int_0^t
\Big(
\nu\lVert \partial_1\partial_2u(s)\rVert_{L^2}^2
+
\eta\lVert \nabla\partial_2B(s)\rVert_{L^2}^2
\Big)\,ds
\nonumber\\
&\qquad\leq
\exp\left\{
C_{\nu,\eta}
\int_0^t
\Big(
\lVert \partial_1u(s)\rVert_{L^2}
+
\lVert \partial_1u(s)\rVert_{L^2}^2
+
\lVert \nabla B(s)\rVert_{L^2}
+
\lVert \nabla B(s)\rVert_{L^2}^2
\Big)\,ds
\right\}
\nonumber\\
&\qquad\quad\times
\Big(
\lVert \partial_2u_0\rVert_{L^2}^2
+
\lVert \partial_2B_0\rVert_{L^2}^2
\Big).
\end{align}
\end{proof}
\noindent
We now turn to  state the main results.
\section{Main Results}
\begin{theorem}
\label{thm10}
Assume that
$
(u_0,B_0)\in H^{0,1}\times H^{0,1},
$
and suppose that
$\sigma_1$ and $\sigma_2$ satisfy
Hypothesis~\ref{HYp1} with
\begin{align}
0\le K_2<\frac15,
\qquad
0\le \widetilde K_2<\frac15.
\end{align}
Then \eqref{MHD} admits a global martingale solution on
$[0,\infty)$ in the sense of
Definition~\ref{DEF}.
Moreover, \eqref{MHD} admits a global probabilistically weak solution
on $[0,\infty)$ in the sense of Definition~\ref{d-w-s}.
\end{theorem}
\begin{theorem}
\label{thm1}
Assume that
$
(u_0, B_0)\in H^{0,1}\times H^{0,1},$
and suppose that 
$\sigma_1$ and $\sigma_2$ satisfy
Hypothesis~\ref{HYp1} with
\begin{align}
0\le K_2<\frac15,
\qquad
0\le \widetilde K_2<\frac15, \qquad 0\leq L_2 < \frac {1} {10}.
\end{align}
If
$
(u^{(1)},B^{(1)})
$
and
$
(u^{(2)},B^{(2)})
$
are two probabilistically weak solutions of \eqref{MHD}
on the same stochastic basis, then
$
(u^{(1)},B^{(1)})=(u^{(2)},B^{(2)})
$
$\mathbb P$-a.s.
\end{theorem}
\noindent
We begin the proof by deriving the necessary a priori estimates for the
Galerkin approximations.
\begin{lemma}
\label{ELY6}
Let $T>0$, and under the hypothesis of Theorem \ref{thm10}, we obtain
\begin{align}
\mathbb E
\sup_{0\le t\le T}
\left(
\lVert u^n(t)\rVert_{L^2}^2
+
\lVert B^n(t)\rVert_{L^2}^2
\right)
+
\mathbb E
\int_0^T
\left(
\nu\lVert \partial_1u^n(t)\rVert_{L^2}^2
+
\eta\lVert \nabla B^n(t)\rVert_{L^2}^2
\right)\,dt
\le
C_T.
\label{L2engy}
\end{align}
The constant $C_T$ depends only on
$T$, $K_0$, $K_1$, $K_2$, $\nu$, $\eta$, but not $n$.
\end{lemma}
\begin{proof}
Recall that the Galerkin approximations satisfy the finite-dimensional
stochastic differential system
\eqref{gVT}-\eqref{GVL}. Applying It\^o's formula to
$
\lVert u^n(t)\rVert_{L^2}^2+\lVert B^n(t)\rVert_{L^2}^2,
$
we obtain
\begin{align}
\lVert u^n(t)\rVert_{L^2}^2
+
\lVert B^n(t)\rVert_{L^2}^2
+&
2\nu
\int_0^t
\lVert \partial_1 u^n(s)\rVert_{L^2}^2\,ds
\nonumber\\
&+
2\eta
\int_0^t
\lVert \nabla B^n(s)\rVert_{L^2}^2\,ds
\nonumber\\
&=
\lVert P_nu_0\rVert_{L^2}^2
+
\lVert P_nB_0\rVert_{L^2}^2
-2\int_0^t
\bigl(
(u^n(s)\cdot\nabla)u^n(s),
u^n(s)
\bigr)\,ds
\nonumber\\
&
-
2\int_0^t
\bigl(
(u^n(s)\cdot\nabla)B^n(s),
B^n(s)
\bigr)\,ds
\nonumber\\
&
+
2\int_0^t
\bigl(
(B^n(s)\cdot\nabla)B^n(s),
u^n(s)
\bigr)\,ds
\nonumber\\
&
+
2\int_0^t
\bigl(
(B^n(s)\cdot\nabla)u^n(s),
B^n(s)
\bigr)\,ds
\nonumber\\
&
+
2\int_0^t
\bigl(
u^n(s),
P_n\sigma_1(s,u^n(s),B^n(s))\Pi_n^{(1)}\,dW_1(s)
\bigr)
\nonumber\\
&
+
2\int_0^t
\bigl(
B^n(s),
P_n\sigma_2(s,u^n(s),B^n(s))\Pi_n^{(2)}\,dW_2(s)
\bigr)
\nonumber\\
&
+
\int_0^t
\lVert P_n\sigma_1(s,u^n(s),B^n(s))\Pi_n^{(1)}\rVert_{L_2(\mathfrak U_1;L^2)}^2\,ds
\nonumber\\
&
+
\int_0^t
\lVert P_n\sigma_2(s,u^n(s),B^n(s))\Pi_n^{(2)}\rVert_{L_2(\mathfrak U_2;L^2)}^2\,ds.
\label{L2-Ito}
\end{align}
Using $\nabla\cdot B^n=0$ and
$$
\bigl(
(B^n\cdot\nabla)B^n,u^n
\bigr)=
-\bigl(
(B^n\cdot\nabla)u^n,B^n
\bigr),
$$
it follows from \eqref{L2-Ito} that
\begin{align}
&\lVert u^n(t)\rVert_{L^2}^2
+
\lVert B^n(t)\rVert_{L^2}^2
+
2\nu
\int_0^t
\lVert \partial_1u^n(s)\rVert_{L^2}^2\,ds
+
2\eta
\int_0^t
\lVert \nabla B^n(s)\rVert_{L^2}^2\,ds
\nonumber\\
&=
\lVert P_nu_0\rVert_{L^2}^2
+
\lVert P_nB_0\rVert_{L^2}^2
+
2\int_0^t
\bigl(
u^n(s),
P_n\sigma_1(s,u^n(s),B^n(s))\Pi_n^{(1)}\,dW_1(s)
\bigr)
\nonumber\\
&
+
2\int_0^t
\bigl(
B^n(s),
P_n\sigma_2(s,u^n(s),B^n(s))\Pi_n^{(2)}\,dW_2(s)
\bigr)
+
\int_0^t
\lVert P_n\sigma_1(s,u^n(s),B^n(s))\Pi_n^{(1)}\rVert_{L_2(\mathfrak U_1;L^2)}^2\,ds
\nonumber\\
&
+
\int_0^t
\lVert P_n\sigma_2(s,u^n(s),B^n(s))\Pi_n^{(2)}\rVert_{L_2(\mathfrak U_2;L^2)}^2\,ds.
\label{L2-id}
\end{align}
It follows from \eqref{nL2} that
\begin{align}
&\lVert P_n\sigma_1(t,u^n,B^n)\Pi_n^{(1)}\rVert_{L_2(\mathfrak U_1;L^2)}^2
+
\lVert P_n\sigma_2(t,u^n,B^n)\Pi_n^{(2)}\rVert_{L_2(\mathfrak U_2;L^2)}^2
\nonumber\\
&\qquad\le
K_0
+
K_1
\left(
\lVert u^n(t)\rVert_{L^2}^2
+
\lVert B^n(t)\rVert_{L^2}^2
\right)
\nonumber\\
&\qquad\quad
+
K_2
\left(
\nu\lVert \partial_1u^n(t)\rVert_{L^2}^2
+
\eta\lVert \nabla B^n(t)\rVert_{L^2}^2
\right).
\label{L2A1}
\end{align}
Applying Burkholder-Davis-Gundy inequality, substituting \eqref{L2A1} into \eqref{L2-id}, 
 applying Young's inequality, it follows from \eqref{L2-id} that
\begin{align}
&(1-\beta)
\mathbb E
\sup_{0\le r\le t}
\left(
\lVert u^n(r)\rVert_{L^2}^2
+
\lVert B^n(r)\rVert_{L^2}^2
\right)
\nonumber\\
&\quad
+
\left[
2-
\left(
1+\frac{9}{\beta}
\right)K_2
\right]
\mathbb E
\int_0^t
\left(
\nu\lVert \partial_1u^n(s)\rVert_{L^2}^2
+
\eta\lVert \nabla B^n(s)\rVert_{L^2}^2
\right)\,ds
\nonumber\\
&\le
\lVert P_nu_0\rVert_{L^2}^2
+
\lVert P_nB_0\rVert_{L^2}^2
+
\left(
1+\frac{9}{\beta}
\right)K_0t
\nonumber\\
&\quad
+
\left(
1+\frac{9}{\beta}
\right)K_1
\int_0^t
\mathbb E
\left(
\lVert u^n(s)\rVert_{L^2}^2
+
\lVert B^n(s)\rVert_{L^2}^2
\right)\,ds.
\label{L2A2}
\end{align}
Since the dissipation term in
\eqref{L2A2}
is nonnegative, we first drop it and in addition, 
 since
$
\lVert P_nu_0\rVert_{L^2}
\le
\lVert u_0\rVert_{L^2},$
and
$\lVert P_nB_0\rVert_{L^2}
\le
\lVert B_0\rVert_{L^2},
$
therefore,
\eqref{L2A2}
implies
\begin{align}
&(1-\beta)
\mathbb E
\sup_{0\le r\le t}
\left(
\lVert u^n(r)\rVert_{L^2}^2
+
\lVert B^n(r)\rVert_{L^2}^2
\right)
\nonumber\\
&\le
\lVert u_0\rVert_{L^2}^2
+
\lVert B_0\rVert_{L^2}^2
+
\left(
1+\frac{9}{\beta}
\right)K_0T
\nonumber\\
&\quad
+
\left(
1+\frac{9}{\beta}
\right)K_1
\int_0^t
\mathbb E
\left(
\lVert u^n(s)\rVert_{L^2}^2
+
\lVert B^n(s)\rVert_{L^2}^2
\right)\,ds.
\label{L2-G1}
\end{align}
Applying Gronwall's inequality,
we conclude that
\begin{align}
\mathbb E
\sup_{0\le t\le T}
\left(
\lVert u^n(t)\rVert_{L^2}^2
+
\lVert B^n(t)\rVert_{L^2}^2
\right)
\le
C_T
\left(
1+
\lVert u_0\rVert_{L^2}^2
+
\lVert B_0\rVert_{L^2}^2
\right),
\label{L2-sup}
\end{align}
where $C_T>0$ depends only on
$K_0,K_1,K_2,\nu,\eta,T$,
through the fixed choice of
$\beta$,
and is independent of $n$.
This  completes the proof.
\end{proof}
\noindent
\noindent
We next establish a uniform $L^4(\Omega)$-estimate for the Galerkin
approximations, corresponding to the $L^2$-energy estimate obtained in
Lemma~\ref{ELY6}.

\begin{lemma}
\label{lemma-fourth-moment-L2}
Let $T>0$, and under the hypothesis of 
\begin{align}
&\mathbb E
\sup_{0\le t\le T}
\left(
\lVert u^n(t)\rVert_{L^2}^2
+
\lVert B^n(t)\rVert_{L^2}^2
\right)^2
+
\mathbb E
\int_0^T
\left(
\lVert u^n(t)\rVert_{L^2}^2
+
\lVert B^n(t)\rVert_{L^2}^2
\right)
\nonumber\\
&\qquad\qquad\qquad\qquad\qquad\qquad\qquad\qquad \qquad\qquad\qquad\times
\left(
\nu\lVert \partial_1u^n(t)\rVert_{L^2}^2
+
\eta\lVert \nabla B^n(t)\rVert_{L^2}^2
\right)\,dt
\le
C_T.
\label{FL2}
\end{align}
The constant $C_T$ depends only on
$T$, $K_0$, $K_1$, $K_2$, $\nu$, $\eta$, and  is independent of $n$.
\end{lemma}

\begin{proof}
Applying It\^o's formula to
$(
\lVert u^n(t)\rVert_{L^2}^2
+
\lVert B^n(t)\rVert_{L^2}^2)^2,
$ and using
$$
d\Big[
\lVert u^n\rVert_{L^2}^2,\,
\lVert B^n\rVert_{L^2}^2
\Big](t)
=
0,$$ 
$$
d\Big[
\lVert u^n\rVert_{L^2}^2,\,
\lVert u^n\rVert_{L^2}^2
\Big](t)
=
4
\left\lVert 
\left(
P_n\sigma_1(t,u^n,B^n)\Pi_n^{(1)}
\right)^*u^n(t)
\right\rVert_{\mathfrak U_1}^2\,dt,
$$
$$
d\Big[
\lVert B^n\rVert_{L^2}^2,\,
\lVert B^n\rVert_{L^2}^2
\Big](t)
=
4
\left\lVert 
\left(
P_n\sigma_2(t,u^n,B^n)\Pi_n^{(2)}
\right)^*B^n(t)
\right\rVert_{\mathfrak U_2}^2\,dt,
$$
we obtain from \eqref{gVT}-\eqref{GVL} that
\begin{align}
&d
\left(
\lVert u^n(t)\rVert_{L^2}^2
+
\lVert B^n(t)\rVert_{L^2}^2
\right)^2
\nonumber\\
&\quad
+
4
\left(
\lVert u^n(t)\rVert_{L^2}^2
+
\lVert B^n(t)\rVert_{L^2}^2
\right)
\left(
\nu\lVert \partial_1u^n(t)\rVert_{L^2}^2
+
\eta\lVert \nabla B^n(t)\rVert_{L^2}^2
\right)\,dt
\nonumber\\
&=
2
\left(
\lVert u^n(t)\rVert_{L^2}^2
+
\lVert B^n(t)\rVert_{L^2}^2
\right)
\nonumber\\
&\qquad\times
\left[
\lVert P_n\sigma_1(t,u^n,B^n)\Pi_n^{(1)}\rVert_{L_2(\mathfrak U_1;L^2)}^2
+
\lVert P_n\sigma_2(t,u^n,B^n)\Pi_n^{(2)}\rVert_{L_2(\mathfrak U_2;L^2)}^2
\right]dt
\nonumber\\
&\quad
+
4
\left(
\lVert u^n(t)\rVert_{L^2}^2
+
\lVert B^n(t)\rVert_{L^2}^2
\right)
\left(
u^n(t),
P_n\sigma_1(t,u^n,B^n)\Pi_n^{(1)}\,dW_1(t)
\right)
\nonumber\\
&\quad
+
4
\left(
\lVert u^n(t)\rVert_{L^2}^2
+
\lVert B^n(t)\rVert_{L^2}^2
\right)
\left(
B^n(t),
P_n\sigma_2(t,u^n,B^n)\Pi_n^{(2)}\,dW_2(t)
\right)
\nonumber\\
&\quad
+
4
\left\lVert 
\left(
P_n\sigma_1(t,u^n,B^n)\Pi_n^{(1)}
\right)^*u^n(t)
\right\rVert_{\mathfrak U_1}^2\,dt
+
4
\left\lVert 
\left(
P_n\sigma_2(t,u^n,B^n)\Pi_n^{(2)}
\right)^*B^n(t)
\right\rVert_{\mathfrak U_2}^2\,dt.
\label{4df}
\end{align}
By the Hilbert-Schmidt inequality \eqref{LKH8}, we obtain
\begin{align}
&4
\lVert 
\left(
P_n\sigma_1(t,u^n,B^n)\Pi_n^{(1)}
\right)^*u^n(t)
\rVert_{\mathfrak U_1}^2
\nonumber\\
&\quad
+
4
\lVert
\left(
P_n\sigma_2(t,u^n,B^n)\Pi_n^{(2)}
\right)^*B^n(t)
\rVert_{\mathfrak U_2}^2
\nonumber\\
&\le
4\lVert u^n(t)\rVert_{L^2}^2
\lvert P_n\sigma_1(t,u^n,B^n)\rVert_{L_2(\mathfrak U_1;L^2)}^2
+
4\lVert B^n(t)\rVert_{L^2}^2
\lvert P_n\sigma_2(t,u^n,B^n)\rVert_{L_2(\mathfrak U_2;L^2)}^2
\nonumber\\
&\le
4
\left(
\lVert u^n(t)\rVert_{L^2}^2
+
\lVert B^n(t)\rVert_{L^2}^2
\right)
\nonumber\\
&\qquad\times
\left[
\lvert P_n\sigma_1(t,u^n,B^n)\rVert_{L_2(\mathfrak U_1;L^2)}^2
+
\lVert P_n\sigma_2(t,u^n,B^n)\rVert_{L_2(\mathfrak U_2;L^2)}^2
\right].
\label{zx12}
\end{align}
It follows from \eqref{4df} and
\eqref{zx12} that
\begin{align}
&d
\left(
\lVert u^n(t)\rVert_{L^2}^2
+
\lVert B^n(t)\rVert_{L^2}^2
\right)^2
\nonumber\\
&\quad
+
4
\left(
\lVert u^n(t)\rVert_{L^2}^2
+
\lVert B^n(t)\rVert_{L^2}^2
\right)
\left(
\nu\lVert \partial_1u^n(t)\rVert_{L^2}^2
+
\eta\lVert \nabla B^n(t)\rVert_{L^2}^2
\right)\,dt
\nonumber\\
&\le
6
\left(
\lVert u^n(t)\rVert_{L^2}^2
+
\lVert B^n(t)\rVert_{L^2}^2
\right)
\nonumber\\
&\qquad\times
\left[
\lVert P_n\sigma_1(t,u^n,B^n)\rVert_{L_2(\mathfrak U_1;L^2)}^2
+
\lVert P_n\sigma_2(t,u^n,B^n)\rVert_{L_2(\mathfrak U_2;L^2)}^2
\right]dt
\nonumber\\
&\quad
+
4
\left(
\lVert u^n(t)\rVert_{L^2}^2
+
\lVert B^n(t)\rVert_{L^2}^2
\right)
\times
\left(
u^n(t),
P_n\sigma_1(t,u^n,B^n)\Pi_n^{(1)}\,dW_1(t)
\right)
\nonumber\\
&\quad
+
4
\left(
\lVert u^n(t)\rVert_{L^2}^2
+
\lVert B^n(t)\rVert_{L^2}^2
\right)\times
\left(
B^n(t),
P_n\sigma_2(t,u^n,B^n)\Pi_n^{(2)}\,dW_2(t)
\right).
\label{OPLR}
\end{align}
Applying \eqref{nL2} into
\eqref{OPLR}, we obtain
\begin{align}
&d
\left(
\lVert u^n(t)\rVert_{L^2}^2
+
\lVert B^n(t)\rVert_{L^2}^2
\right)^2
\nonumber\\
&\quad
+
\left(
4-6 K_2
\right)
\left(
\lVert u^n(t)\rVert_{L^2}^2
+
\lVert B^n(t)\rVert_{L^2}^2
\right)
\nonumber\\
&\qquad\times
\left(
\nu\lVert \partial_1u^n(t)\rVert_{L^2}^2
+
\eta\lVert \nabla B^n(t)\rVert_{L^2}^2
\right)\,dt
\nonumber\\
&\le
6K_0
\left(
\lVert u^n(t)\rVert_{L^2}^2
+
\lVert B^n(t)\rVert_{L^2}^2
\right)\,dt
+
6K_1
\left(
\lVert u^n(t)\rVert_{L^2}^2
+
\lVert B^n(t)\rVert_{L^2}^2
\right)^2\,dt
\nonumber\\
&\quad
+
4
\left(
\lVert u^n(t)\rVert_{L^2}^2
+
\lVert B^n(t)\rVert_{L^2}^2
\right)
\times
\left(
u^n(t),
P_n\sigma_1(t,u^n,B^n)\Pi_n^{(1)}\,dW_1(t)
\right)
\nonumber\\
&\quad
+
4
\left(
\lVert u^n(t)\rVert_{L^2}^2
+
\lVert B^n(t)\rVert_{L^2}^2
\right)
\times
\left(
B^n(t),
P_n\sigma_2(t,u^n,B^n)\Pi_n^{(2)}\,dW_2(t)
\right).
\label{kl23x}
\end{align}
Integrating \eqref{kl23x} over
$[0,t]$, we obtain
\begin{align}
&
\left(
\lVert u^n(t)\rVert_{L^2}^2
+
\lVert B^n(t)\rVert_{L^2}^2
\right)^2
\nonumber\\
&\quad
+
\left(
4-6\widetilde K_2
\right)
\int_0^t
\left(
\lVert u^n(s)\rVert_{L^2}^2
+
\lVert B^n(s)\rVert_{L^2}^2
\right)
\nonumber\\
&\qquad\qquad\times
\left(
\nu\lVert \partial_1u^n(s)\rVert_{L^2}^2
+
\eta\lVert \nabla B^n(s)\rVert_{L^2}^2
\right)\,ds
\nonumber\\
&\le
\left(
\lVert P_nu_0\rVert_{L^2}^2
+
\lVert P_nB_0\rVert_{L^2}^2
\right)^2
\nonumber\\
&\quad
+
6K_0
\int_0^t
\left(
\lVert u^n(s)\rVert_{L^2}^2
+
\lVert B^n(s)\rVert_{L^2}^2
\right)\,ds
\nonumber\\
&\quad
+
6K_1
\int_0^t
\left(
\lVert u^n(s)\rVert_{L^2}^2
+
\lVert B^n(s)\rVert_{L^2}^2
\right)^2\,ds
\nonumber\\
&\quad
+
4
\int_0^t
\left(
\lVert u^n(s)\rVert_{L^2}^2
+
\lVert B^n(s)\rVert_{L^2}^2
\right)
\left(
u^n(s),
P_n\sigma_1(s,u^n,B^n)\Pi_n^{(1)}\,dW_1(s)
\right)
\nonumber\\
&\quad
+
4
\int_0^t
\left(
\lVert u^n(s)\rVert_{L^2}^2
+
\lVert B^n(s)\rVert_{L^2}^2
\right)
\left(
B^n(s),
P_n\sigma_2(s,u^n,B^n)\Pi_n^{(2)}\,dW_2(s)
\right).
\label{B23dg2}
\end{align}
Taking the supremum over $r\in[0,t]$, taking expectations, and applying
the Burkholder-Davis-Gundy inequality on \eqref{B23dg2},to obtain
\begin{align}
&\mathbb E
\sup_{0\le r\le t}
\left(
\lVert u^n(r)\rVert_{L^2}^2
+
\lVert B^n(r)\rVert_{L^2}^2
\right)^2
\nonumber\\
&\quad
+
\left(
4-6K_2
\right)
\mathbb E
\int_0^t
\left(
\lVert u^n(s)\rVert_{L^2}^2
+
\lVert B^n(s)\rVert_{L^2}^2
\right)
\nonumber\\
&\qquad\qquad\times
\left(
\nu\lVert \partial_1u^n(s)\rVert_{L^2}^2
+
\eta\lVert \nabla B^n(s)\rVert_{L^2}^2
\right)\,ds
\nonumber\\
&\le
\mathbb E
\left(
\lVert P_nu_0\rVert_{L^2}^2
+
\lVert P_nB_0\rVert_{L^2}^2
\right)^2
\nonumber\\
&\quad
+
6K_0
\mathbb E
\int_0^t
\left(
\lVert u^n(s)\rVert_{L^2}^2
+
\lVert B^n(s)\rVert_{L^2}^2
\right)\,ds
+
6K_1
\mathbb E
\int_0^t
\left(
\lVert u^n(s)\rVert_{L^2}^2
+
\lVert B^n(s)\rVert_{L^2}^2
\right)^2\,ds
\nonumber\\
&\quad
+
12\,
\mathbb E
\Bigg[
\int_0^t
\left(
\lVert u^n(s)\rVert_{L^2}^2
+
\lVert B^n(s)\rVert_{L^2}^2
\right)^2
\times
\Bigg(
\left\lVert 
\left(
P_n\sigma_1(s,u^n,B^n)\Pi_n^{(1)}
\right)^*u^n(s)
\right\rVert_{\mathfrak U_1}^2
\nonumber\\
&\hspace{5.5cm}
+
\left\lVert 
\left(
P_n\sigma_2(s,u^n,B^n)\Pi_n^{(2)}
\right)^*B^n(s)
\right\rVert_{\mathfrak U_2}^2
\Bigg)\,ds
\Bigg]^{1/2}.
\label{fm-BDG}
\end{align}
Using Hilbert-Schmidt inequality,   \eqref{nL2}, and Young's inequality, we obtain
\begin{align}
&12\,
\mathbb E
\Bigg[
\sup_{0\le r\le t}
\left(
\lVert u^n(r)\rVert_{L^2}^2
+
\lVert B^n(r)\rVert_{L^2}^2
\right)^2
\Bigg]^{1/2}
\nonumber\\
&\quad\times
\Bigg[
\int_0^t
\left(
\lVert u^n(s)\rVert_{L^2}^2
+
\lVert B^n(s)\rVert_{L^2}^2
\right)
\nonumber\\
&\qquad\qquad\times
\left[
\lVert P_n\sigma_1(s,u^n,B^n)\rVert_{L_2(\mathfrak U_1;L^2)}^2
+
\lVert P_n\sigma_2(s,u^n,B^n)\rVert_{L_2(\mathfrak U_2;L^2)}^2
\right]\,ds
\Bigg]^{1/2}
\nonumber\\
&\le
\gamma
\mathbb E
\sup_{0\le r\le t}
\left(
\lVert u^n(r)\rVert_{L^2}^2
+
\lVert B^n(r)\rVert_{L^2}^2
\right)^2
+
\frac{36K_0}{\gamma}
\mathbb E
\int_0^t
\left(
\lVert u^n(s)\rVert_{L^2}^2
+
\lVert B^n(s)\rVert_{L^2}^2
\right)\,ds
\nonumber\\
&\quad
+
\frac{36K_1}{\gamma}
\mathbb E
\int_0^t
\left(
\lVert u^n(s)\rVert_{L^2}^2
+
\lVert B^n(s)\rVert_{L^2}^2
\right)^2\,ds
\nonumber\\
& \quad+
\frac{36K_2}{\gamma}
\mathbb E
\int_0^t
\left(
\lVert u^n(s)\rVert_{L^2}^2
+
\lVert B^n(s)\rVert_{L^2}^2
\right)\times
\left(
\nu\lVert \partial_1u^n(s)\rVert_{L^2}^2
+
\eta\lVert \nabla B^n(s)\rVert_{L^2}^2
\right)\,ds.
\label{fm-Young}
\end{align}
Applying Tonelli's Theorem, and combining \eqref{fm-BDG} and \eqref{fm-Young}, we find
\begin{align}
&(1-\gamma)
\mathbb E
\sup_{0\le r\le t}
\left(
\lVert u^n(r)\rVert_{L^2}^2
+
\lVert B^n(r)\rVert_{L^2}^2
\right)^2
\nonumber\\
&\quad
+
\left[
4-
\left(
6+\frac{36}{\gamma}
\right)K_2
\right]
\mathbb E
\int_0^t
\left(
\lVert u^n(s)\rVert_{L^2}^2
+
\lVert B^n(s)\rVert_{L^2}^2
\right)
\nonumber\\
&\qquad\qquad\times
\left(
\nu\lVert \partial_1u^n(s)\rVert_{L^2}^2
+
\eta\lVert \nabla B^n(s)\rVert_{L^2}^2
\right)\,ds
\nonumber\\
&\le
\mathbb E
\left(
\lVert P_nu_0\rVert_{L^2}^2
+
\lVert P_nB_0\rVert_{L^2}^2
\right)^2
+
\left(
6+\frac{36}{\gamma}
\right)K_0
\mathbb E
\int_0^t
\left(
\lVert u^n(s)\rVert_{L^2}^2
+
\lVert B^n(s)\rVert_{L^2}^2
\right)\,ds
\nonumber\\
&\quad
+
\left(
6+\frac{36}{\gamma}
\right)K_1
\int_0^t
\mathbb E
\sup_{0\le r\le s}
\left(
\lVert u^n(r)\rVert_{L^2}^2
+
\lVert B^n(r)\rVert_{L^2}^2
\right)^2\,ds.
\label{fm-main}
\end{align}
Since
$
\lVert P_nu_0\rVert_{L^2}\le \lVert u_0\rVert_{L^2},
\lVert P_nB_0\rVert_{L^2}\le \lVert B_0\rVert_{L^2}$, using the previously established
$L^2$-estimate \eqref{L2-sup}, and dropping the nonnegative weighted
dissipation term from the left-hand side of \eqref{fm-main}, we obtain
\begin{align}
&\mathbb E
\sup_{0\le r\le t}
\left(
\lVert u^n(r)\rVert_{L^2}^2
+
\lVert B^n(r)\rVert_{L^2}^2
\right)^2
\le
C_T
+
C
\int_0^t
\mathbb E
\sup_{0\le r\le s}
\left(
\lVert u^n(r)\rVert_{L^2}^2
+
\lVert B^n(r)\rVert_{L^2}^2
\right)^2\,ds,
\label{fm-Gronwall}
\end{align}
where $C_T$ and $C$ are independent of $n$.
Therefore, Gronwall's inequality yields
\begin{align}
\sup_{n\in\mathbb N}
\mathbb E
\sup_{0\le t\le T}
\left(
\lVert u^n(t)\rVert_{L^2}^2
+
\lVert B^n(t)\rVert_{L^2}^2
\right)^2
\le
C_T.
\label{L^4-mo}
\end{align}
\end{proof}
\noindent We next derive uniform $H^{0,1}$-estimates for the Galerkin approximations.
\begin{lemma}
\label{lem-vertical-estimate}
Let $T>0$. 
Then there exists a constant $C_T>0$, independent of $n$, such that
\begin{align}
&\mathbb E
\sup_{0\le t\le T}
\exp\left\{
-\int_0^t
\left[
C_{\alpha,\nu,\eta}
\left(
1
+
\lVert \partial_1u^n(s)\rVert_{L^2}
+
\lVert \partial_1u^n(s)\rVert_{L^2}^2
\right.
\right.
\right.
\nonumber\\
&\hspace{5.4cm}\left.\left.\left.
+
\lVert \nabla B^n(s)\rVert_{L^2}
+
\lVert \nabla B^n(s)\rVert_{L^2}^2
\right)
+
\widetilde K_1
\right]ds
\right\}
\nonumber\\
&\qquad\times
\left(
\lVert \partial_2u^n(t)\rVert_{L^2}^2
+
\lVert \partial_2B^n(t)\rVert_{L^2}^2
\right)
\nonumber\\
&\quad
+
\mathbb E
\int_0^T
\exp\left\{
-\int_0^t
\left[
C_{\alpha,\nu,\eta}
\left(
1
+
\lVert \partial_1u^n(s)\rVert_{L^2}
+
\lVert \partial_1u^n(s)\rVert_{L^2}^2
\right.
\right.
\right.
\nonumber\\
&\hspace{5.4cm}\left.\left.\left.
+
\lVert \nabla B^n(s)\rVert_{L^2}
+
\lVert \nabla B^n(s)\rVert_{L^2}^2
\right)
+
\widetilde K_1
\right]ds
\right\}
\nonumber\\
&\qquad\times
\left(
\nu\lVert \partial_1\partial_2u^n(t)\rVert_{L^2}^2
+
\eta\lVert \nabla\partial_2B^n(t)\rVert_{L^2}^2
\right)\,dt
\le
C_T.
\label{vert-estimate}
\end{align}
The constant $C_T$ depends on
$
T,
\nu,
\eta,
\widetilde K_0,
\widetilde K_1,
\widetilde K_2,
\alpha,$
and is independent of $n$.
\end{lemma}

\begin{proof}
Applying $\partial_2$ to \eqref{gVT} and \eqref{GVL}, and then applying
It\^o's formula to $\lVert \partial_2u^n(t)\rVert_{L^2}^2$ and
$\lVert \partial_2B^n(t)\rVert_{L^2}^2$, respectively, and adding the
resulting identities, we obtain
\begin{align}
&d\left(
\lVert \partial_2u^n(t)\rVert_{L^2}^2
+
\lVert \partial_2B^n(t)\rVert_{L^2}^2
\right)
+
2\nu
\lVert \partial_1\partial_2u^n(t)\rVert_{L^2}^2\,dt
\nonumber\\
&\quad
+
2\eta
\lVert \nabla\partial_2B^n(t)\rVert_{L^2}^2\,dt
\nonumber\\
&=
-2\bigl(
\partial_2((u^n\cdot\nabla)u^n),
\partial_2u^n
\bigr)\,dt
+
2\bigl(
\partial_2((B^n\cdot\nabla)B^n),
\partial_2u^n
\bigr)\,dt
\nonumber\\
&\quad
-
2\bigl(
\partial_2((u^n\cdot\nabla)B^n),
\partial_2B^n
\bigr)\,dt
+
2\bigl(
\partial_2((B^n\cdot\nabla)u^n),
\partial_2B^n
\bigr)\,dt
\nonumber\\
&\quad
+
\lVert
\partial_2P_n\sigma_1(t,u^n,B^n)\Pi_n^{(1)}
\rVert_{L_2(\mathfrak U_1;L^2)}^2\,dt
\nonumber\\
&\quad
+
\lVert
\partial_2P_n\sigma_2(t,u^n,B^n)\Pi_n^{(2)}
\rVert_{L_2(\mathfrak U_2;L^2)}^2\,dt
+
dM_n^{(2)}(t),
\label{RTlq}
\end{align}
where
\begin{align}
dM_n^{(2)}(t)
=
2\bigl(
\partial_2P_n\sigma_1(t,u^n,B^n)\Pi_n^{(1)}\,dW_1(t),
\partial_2u^n(t)
\bigr)
+
2\bigl(
\partial_2P_n\sigma_2(t,u^n,B^n)\Pi_n^{(2)}\,dW_2(t),
\partial_2B^n(t)
\bigr).
\label{ty4}
\end{align}
Using
$
\nabla\cdot u^n=0,
\,
\nabla\cdot B^n=0,
$
and
$$
\bigl(
(u^n\cdot\nabla)\partial_2u^n,
\partial_2u^n
\bigr)
=
\bigl(
(u^n\cdot\nabla)\partial_2B^n,
\partial_2B^n
\bigr)
=0,
$$
$$
\bigl(
(B^n\cdot\nabla)\partial_2B^n,
\partial_2u^n
\bigr)
+
\bigl(
(B^n\cdot\nabla)\partial_2u^n,
\partial_2B^n
\bigr)
=0,
$$
and repeating the estimates in \eqref{A5}-\eqref{A9} with Young's
inequality for $\alpha>0$, we obtain
\begin{align}
&\biggl\lvert
-2\bigl(
(\partial_2u^n\cdot\nabla)u^n,
\partial_2u^n
\bigr)
+
2\bigl(
(\partial_2B^n\cdot\nabla)B^n,
\partial_2u^n
\bigr)
\nonumber\\
&\hspace{1.5cm}
-
2\bigl(
(\partial_2u^n\cdot\nabla)B^n,
\partial_2B^n
\bigr)
+
2\bigl(
(\partial_2B^n\cdot\nabla)u^n,
\partial_2B^n
\bigr)
\biggr\rvert
\nonumber\\
&\le
2\alpha
\bigl(
\nu\lVert\partial_1\partial_2u^n\rVert_{L^2}^2
+
\eta\lVert\nabla\partial_2B^n\rVert_{L^2}^2
\bigr)
\nonumber\\
&\quad
+
C_{\alpha,\nu,\eta}
\bigl(
1
+
\lVert\partial_1u^n\rVert_{L^2}
+
\lVert\partial_1u^n\rVert_{L^2}^2
+
\lVert\nabla B^n\rVert_{L^2}
+
\lVert\nabla B^n\rVert_{L^2}^2
\bigr)
\nonumber\\
&\qquad\times
\bigl(
\lVert\partial_2u^n\rVert_{L^2}^2
+
\lVert\partial_2B^n\rVert_{L^2}^2
\bigr).
\label{iqwe}
\end{align}
Substituting \eqref{iqwe} into
\eqref{RTlq}, we obtain
\begin{align}
&d\bigl(
\lVert\partial_2u^n(t)\rVert_{L^2}^2
+
\lVert\partial_2B^n(t)\rVert_{L^2}^2
\bigr)
\nonumber\\
&\quad
+
2(1-\alpha)
\bigl(
\nu\lVert\partial_1\partial_2u^n(t)\rVert_{L^2}^2
+
\eta\lVert\nabla\partial_2B^n(t)\rVert_{L^2}^2
\bigr)\,dt
\nonumber\\
&\le
C_{\alpha,\nu,\eta}
\bigl(
1
+
\lVert\partial_1u^n(t)\rVert_{L^2}
+
\lVert\partial_1u^n(t)\rVert_{L^2}^2
+
\lVert\nabla B^n(t)\rVert_{L^2}
+
\lVert\nabla B^n(t)\rVert_{L^2}^2
\bigr)
\nonumber\\
&\qquad\times
\bigl(
\lVert\partial_2u^n(t)\rVert_{L^2}^2
+
\lVert\partial_2B^n(t)\rVert_{L^2}^2
\bigr)\,dt
\nonumber\\
&\quad
+
\lVert
\partial_2P_n\sigma_1(t,u^n,B^n)\Pi_n^{(1)}
\rVert_{L_2(\mathfrak U_1;L^2)}^2\,dt
\nonumber\\
&\quad
+
\lVert
\partial_2P_n\sigma_2(t,u^n,B^n)\Pi_n^{(2)}
\rVert_{L_2(\mathfrak U_2;L^2)}^2\,dt
+
dM_n^{(2)}(t).
\label{78AW}
\end{align}
Since $P_n$ and $\Pi_n^{(i)}$ are orthogonal projections, and
$P_n$ commutes with $\partial_2$ on the Galerkin space,
\eqref{nH01} implies
\begin{align}
&\lVert \partial_2P_n\sigma_1(t,u^n,B^n)
\Pi_n^{(1)}\rVert_{L_2(\mathfrak U_1;L^2)}^2
+
\lvert \partial_2P_n\sigma_2(t,u^n,B^n)
\Pi_n^{(2)}\rVert_{L_2(\mathfrak U_2;L^2)}^2
\nonumber\\
&\le
\widetilde K_0
+
\widetilde K_1
\left(
\lVert \partial_2u^n(t)\rVert_{L^2}^2
+
\lvert \partial_2B^n(t)\rVert_{L^2}^2
\right)
\nonumber\\
&\quad
+
\widetilde K_2
(
\nu\lvert \partial_1\partial_2u^n(t)\rVert_{L^2}^2
+
\eta\lVert \nabla\partial_2B^n(t)\rVert_{L^2}^2
).
\label{tyreQ}
\end{align}
Substituting \eqref{tyreQ} into
\eqref{78AW}, we conclude that
\begin{align}
&d\bigl(
\lVert\partial_2u^n(t)\rVert_{L^2}^2
+
\lVert\partial_2B^n(t)\rVert_{L^2}^2
\bigr)
\nonumber\\
&\quad
+
\bigl(
2-2\alpha-\widetilde K_2
\bigr)
\bigl(
\nu\lVert\partial_1\partial_2u^n(t)\rVert_{L^2}^2
+
\eta\lVert\nabla\partial_2B^n(t)\rVert_{L^2}^2
\bigr)\,dt
\nonumber\\
&\le
\biggl[
C_{\alpha,\nu,\eta}
\bigl(
1
+
\lVert\partial_1u^n(t)\rVert_{L^2}
+
\lVert\partial_1u^n(t)\rVert_{L^2}^2
+
\lVert\nabla B^n(t)\rVert_{L^2}
+
\lVert\nabla B^n(t)\rVert_{L^2}^2
\bigr)
+
\widetilde K_1
\biggr]
\nonumber\\
&\qquad\times
\bigl(
\lVert\partial_2u^n(t)\rVert_{L^2}^2
+
\lVert\partial_2B^n(t)\rVert_{L^2}^2
\bigr)\,dt
+
\widetilde K_0\,dt
+
dM_n^{(2)}(t).
\label{67W}
\end{align}
Define
$$
h_n(t)
:=
\int_0^t
\Bigg[
C_{\alpha,\nu,\eta}
\Big(
1
+
\lVert \partial_1u^n(s)\rVert_{L^2}
+
\lVert \partial_1u^n(s)\rVert_{L^2}^2
+
\lVert \nabla B^n(s)\rVert_{L^2}
+
\lVert \nabla B^n(s)\rVert_{L^2}^2
\Big)
+
\widetilde K_1
\Bigg]\,ds,
$$ and 
$$
h_n'(t)=
C_{\alpha,\nu,\eta}
\Big(
1
+
\lVert \partial_1u^n(t)\rVert_{L^2}
+
\lVert \partial_1u^n(t)\rVert_{L^2}^2
+
\lVert \nabla B^n(t)\rVert_{L^2}
+
\lVert \nabla B^n(t)\rVert_{L^2}^2
\Big)
+
\widetilde K_1,$$
we have 
$$
h_n(t)=\int_0^t h_n'(s)\,ds
\textrm{
with} \; h_n'\in L^1(0,T).$$ Since
$$
d(e^{-h_n(t)})=-h_n'(t)e^{-h_n(t)}\,dt
\qquad \textrm{and} \qquad
[e^{-h_n},
\lVert\partial_2u^n\rVert_{L^2}^2+
\lVert\partial_2B^n\rVert_{L^2}^2]_t=0,
$$
therefore, the product formula gives
\begin{align}
d\Bigg[
e^{-h_n(t)}
\bigl(
\lVert\partial_2u^n(t)\rVert_{L^2}^2
+
\lVert\partial_2B^n(t)\rVert_{L^2}^2
\bigr)
\Bigg]
&=
e^{-h_n(t)}
d\bigl(
\lVert\partial_2u^n(t)\rVert_{L^2}^2
+
\lVert\partial_2B^n(t)\rVert_{L^2}^2
\bigr)
\nonumber\\
&\quad
-
h_n'(t)e^{-h_n(t)}
\bigl( 
\lVert\partial_2u^n(t)\rVert_{L^2}^2
+
\lVert\partial_2B^n(t)\rVert_{L^2}^2
\bigr)\,dt.
\label{yu4e} 
\end{align}
Multiplying \eqref{67W} by $e^{-h_n(t)}$ using \eqref{yu4e} and canceling$$
h_n'(t)e^{-h_n(t)}
\left(
\lVert \partial_2u^n(t)\rVert_{L^2}^2
+
\lVert \partial_2B^n(t)\rVert_{L^2}^2
\right)\,dt
$$ from both sides, we obtain
\begin{align}
&d\Bigg[
e^{-h_n(t)}
\left(
\lVert \partial_2u^n(t)\rVert_{L^2}^2
+
\lVert \partial_2B^n(t)\rVert_{L^2}^2
\right)
\Bigg]
\nonumber\\
&\quad
+
\left(
2-2\alpha-\widetilde K_2
\right)
e^{-h_n(t)}
\left(
\nu\lVert \partial_1\partial_2u^n(t)\rVert_{L^2}^2
+
\eta\lVert \nabla\partial_2B^n(t)\rVert_{L^2}^2
\right)\,dt
\nonumber\\
&\le
\widetilde K_0e^{-h_n(t)}\,dt
+
e^{-h_n(t)}\,dM_n^{(2)}(t).
\label{87gb}
\end{align}
Integrating \eqref{87gb} over $[0,t]$, we obtain
\begin{align}
&e^{-h_n(t)}
\left(
\lVert \partial_2u^n(t)\rVert_{L^2}^2
+
\lVert \partial_2B^n(t)\rVert_{L^2}^2
\right)
\nonumber\\
&\quad
+
\left(
2-2\alpha-\widetilde K_2
\right)
\int_0^t
e^{-h_n(s)}
\left(
\nu\lVert \partial_1\partial_2u^n(s)\rVert_{L^2}^2
+
\eta\lVert \nabla\partial_2B^n(s)\rVert_{L^2}^2
\right)\,ds
\nonumber\\
&\le
\lVert \partial_2P_nu_0\rVert_{L^2}^2
+
\lVert \partial_2P_nB_0\rVert_{L^2}^2
+
\widetilde K_0
\int_0^t e^{-h_n(s)}\,ds
+
\int_0^t e^{-h_n(s)}\,dM_n^{(2)}(s).
\label{6tyv}
\end{align}
Since,
$
0<e^{-h_n(t)}\le1 \, \textrm{and}\,
\int_0^t e^{-h_n(s)}\,ds
\le
t
\le
T,
$ it follows from \eqref{6tyv} that
\begin{align}
&e^{-h_n(t)}
\left(
\lVert \partial_2u^n(t)\rVert_{L^2}^2
+
\lVert \partial_2B^n(t)\rVert_{L^2}^2
\right)
\nonumber\\
&\quad
+
\left(
2-2\alpha-\widetilde K_2
\right)
\int_0^t
e^{-h_n(s)}
\left(
\nu\lVert \partial_1\partial_2u^n(s)\rVert_{L^2}^2
+
\eta\lVert \nabla\partial_2B^n(s)\rVert_{L^2}^2
\right)\,ds
\nonumber\\
&\le
\lVert \partial_2P_nu_0\rVert_{L^2}^2
+
\lVert \partial_2P_nB_0\rVert_{L^2}^2
+
\widetilde K_0T
+
\int_0^t e^{-h_n(s)}\,dM_n^{(2)}(s).
\label{rt45h}
\end{align}
Taking the supremum over $t\in[0,T]$ in
\eqref{rt45h}, and taking
expectations, we obtain
\begin{align}
&\mathbb E
\sup_{0\le t\le T}
\Bigg[
e^{-h_n(t)}
\left(
\lVert \partial_2u^n(t)\rVert_{L^2}^2
+
\lVert \partial_2B^n(t)\rVert_{L^2}^2
\right)
\Bigg]
\nonumber\\
&\quad
+
\left(
2-2\alpha-\widetilde K_2
\right)
\mathbb E
\int_0^T
e^{-h_n(s)}
\left(
\nu\lVert \partial_1\partial_2u^n(s)\rVert_{L^2}^2
+
\eta\lVert \nabla\partial_2B^n(s)\rVert_{L^2}^2
\right)\,ds
\nonumber\\
&\le
\mathbb E
\left(
\lVert \partial_2P_nu_0\rVert_{L^2}^2
+
\lVert \partial_2P_nB_0\rVert_{L^2}^2
\right)
+
\widetilde K_0T
\nonumber\\
&\quad
+
\mathbb E
\sup_{0\le t\le T}
\left\lvert
\int_0^t
e^{-h_n(s)}\,dM_n^{(2)}(s)
\right\rvert.
\label{89bfn}
\end{align}
Using the independence of $W_1$ and $W_2$, so that
$
[\beta_k^{(1)},\beta_\ell^{(2)}]_t=0
$
for all $k,\ell\ge1$, the Burkholder-Davis-Gundy inequality, the
Hilbert-Schmidt inequality \eqref{LKH8}, Young's inequality
$
6ab\le\beta a^2+\frac{9}{\beta}b^2
$
for $a,b\ge0$ and $0<\beta<1$, and  \eqref{tyreQ},
we obtain
\begin{align}
&\mathbb E
\sup_{0\le t\le T}
\biggl\lvert
\int_0^t
e^{-h_n(s)}\,dM_n^{(2)}(s)
\biggr\rvert
\nonumber\\
&\le
\beta
\mathbb E
\sup_{0\le r\le T}
\biggl[
e^{-h_n(r)}
\bigl(
\lVert\partial_2u^n(r)\rVert_{L^2}^2
+
\lVert\partial_2B^n(r)\rVert_{L^2}^2
\bigr)
\biggr]
+
\frac{9\widetilde K_0T}{\beta}
\nonumber\\
&\quad
+
\frac{9\widetilde K_1}{\beta}
\mathbb E
\int_0^T
e^{-h_n(s)}
\bigl(
\lVert\partial_2u^n(s)\rVert_{L^2}^2
+
\lVert\partial_2B^n(s)\rVert_{L^2}^2
\bigr)\,ds
\nonumber\\
&\quad
+
\frac{9\widetilde K_2}{\beta}
\mathbb E
\int_0^T
e^{-h_n(s)}
\bigl(
\nu\lVert\partial_1\partial_2u^n(s)\rVert_{L^2}^2
+
\eta\lVert\nabla\partial_2B^n(s)\rVert_{L^2}^2
\bigr)\,ds.
\label{BDAS2}
\end{align}
Substituting \eqref{BDAS2} into
\eqref{89bfn}, and moving the corresponding terms to the left-hand side, we obtain
\begin{align}
&(1-\beta)
\mathbb E
\sup_{0\le t\le T}
\Bigg[
e^{-h_n(t)}
\left(
\lVert \partial_2u^n(t)\rVert_{L^2}^2
+
\lVert \partial_2B^n(t)\rVert_{L^2}^2
\right)
\Bigg]
\nonumber\\
&\quad
+
\left[
2
-
2\alpha
-
\left(
1+\frac{9}{\beta}
\right)\widetilde K_2
\right]
\mathbb E
\int_0^T
e^{-h_n(s)}
\left(
\nu\lVert \partial_1\partial_2u^n(s)\rVert_{L^2}^2
+
\eta\lVert \nabla\partial_2B^n(s)\rVert_{L^2}^2
\right)\,ds
\nonumber\\
&\le
\mathbb E
\left(
\lVert \partial_2P_nu_0\rVert_{L^2}^2
+
\lVert \partial_2P_nB_0\rVert_{L^2}^2
\right)
+
\left(
1+\frac{9}{\beta}
\right)\widetilde K_0T
\nonumber\\
&\quad
+
\frac{9\widetilde K_1}{\beta}
\mathbb E
\int_0^T
e^{-h_n(s)}
\left(
\lVert \partial_2u^n(s)\rVert_{L^2}^2
+
\lVert \partial_2B^n(s)\rVert_{L^2}^2
\right)\,ds.
\label{78x.k}
\end{align}
Since
$$
\int_0^T
e^{-h_n(s)}
\lVert \partial_2u^n(s)\rVert_{L^2}^2
+
\lVert \partial_2B^n(s)\rVert_{L^2}^2\,ds
\le
\int_0^T
\sup_{0\le r\le s}
\Bigg[
e^{-h_n(r)}
\left(
\lVert \partial_2u^n(r)\rVert_{L^2}^2
+
\lVert \partial_2B^n(r)\rVert_{L^2}^2
\right)
\Bigg]\,ds
$$ and dropping the second nonnegative term on the left-hand side of \eqref{78x.k} and using Tonelli's theorem, we obtain
\begin{align}
&\mathbb E
\sup_{0\le t\le T}
\Bigg[
e^{-h_n(t)}
\left(
\lVert \partial_2u^n(t)\rVert_{L^2}^2
+
\lVert \partial_2B^n(t)\rVert_{L^2}^2
\right)
\Bigg]
\nonumber\\
&\le
C_0
+
C_1
\int_0^T
\mathbb E
\sup_{0\le r\le s}
\Bigg[
e^{-h_n(r)}
\left(
\lVert \partial_2u^n(r)\rVert_{L^2}^2
+
\lVert \partial_2B^n(r)\rVert_{L^2}^2
\right)
\Bigg]\,ds.
\label{redw}
\end{align}
where $C_0>0$ depends only on
$T,\widetilde K_0,\beta$ and the initial data, while $C_1>0$
depends only on $\widetilde K_1,\beta$; both are independent of $n$.
Applying the Gronwall's inequality to
\eqref{redw},
we conclude that
\begin{align}
\mathbb E
\sup_{0\le t\le T}
\Bigg[
e^{-h_n(t)}
\left(
\lVert \partial_2u^n(t)\rVert_{L^2}^2
+
\lVert \partial_2B^n(t)\rVert_{L^2}^2
\right)
\Bigg]
\le
C_0e^{C_1T},
\label{7hczz}
\end{align}
Combining \eqref{89bfn} and \eqref{BDAS2} again with
\eqref{7hczz},
we obtain
\begin{align}
&\mathbb E
\sup_{0\le t\le T}
\biggl[
e^{-h_n(t)}
\bigl(
\lVert\partial_2u^n(t)\rVert_{L^2}^2
+
\lVert\partial_2B^n(t)\rVert_{L^2}^2
\bigr)
\biggr]
\nonumber\\
&\quad+
\mathbb E
\int_0^T
e^{-h_n(s)}
\bigl(
\nu\lVert\partial_1\partial_2u^n(s)\rVert_{L^2}^2
+
\eta\lVert\nabla\partial_2B^n(s)\rVert_{L^2}^2
\bigr)\,ds
\le C,
\label{vwre}
\end{align}
where the constant depends only on
$
T,\widetilde K_0,\widetilde K_1,\widetilde K_2,\alpha,\beta
$
and the initial data, and is independent of $n$.
This completes the proof.
\end{proof}
\noindent
We now use the preceding a priori estimates to establish the tightness
of the laws of the Galerkin approximations.
\begin{lemma}
\label{tight}
Under the hypotheses of Theorem~\ref{thm10}, the family of laws
$
\bigl\{
\mathcal L(u^n,B^n),\ n\in\mathbb N
\bigr\}
$
is tight in the space
$
\mathcal Z
=
\mathcal X_u\times\mathcal X_B,
$
where
\begin{align}
\mathcal X_u
:=
C(\mathbb R_+;H^{-1})
\cap
L_{\mathrm{loc}}^2(\mathbb R_+;H)
\cap
L_{\mathrm{loc},w}^2(\mathbb R_+;H^{1,0})\cap
L_{\mathrm{loc},w}^2(\mathbb R_+;H^{1,1})
\cap
L_{\mathrm{loc},w^\ast}^{\infty}
(\mathbb R_+;H^{0,1}),
\label{Xu_Tig}
\end{align}
and
\begin{align}
\mathcal X_B
:=
&C(\mathbb R_+;H^{-1})
\cap
L_{\mathrm{loc}}^2(\mathbb R_+;H)
\cap
L_{\mathrm{loc},w}^2(\mathbb R_+;H^1)
\nonumber\\
&\cap
L_{\mathrm{loc},w}^2(\mathbb R_+;H^{1,1})
\cap
L_{\mathrm{loc},w}^2(\mathbb R_+;H^{0,2})
\cap
L_{\mathrm{loc},w^\ast}^{\infty}
(\mathbb R_+;H^{0,1}).
\label{XB_Tig}
\end{align}
\end{lemma}
\begin{proof}
For every $T>0$, define
\begin{align}
A_1(u,B;T)
:=
&\sup_{0\le t\le T}
\left(
\lVert u(t)\rVert_{L^2}^2
+
\lVert B(t)\rVert_{L^2}^2
\right)
+
\int_0^T
\left(
\nu\lVert \partial_1u(t)\rVert_{L^2}^2
+
\eta\lVert \nabla B(t)\rVert_{L^2}^2
\right)\,dt
\nonumber\\
&+
\lVert u\rVert_{C^{1/8}([0,T];H^{-1})}
+
\lVert B\rVert_{C^{1/8}([0,T];H^{-1})},
\label{TA1}
\end{align}
and
\begin{align}
A_2(u,B;T)
:=
&\sup_{0\le t\le T}
e^{-h_{u,B}(t)}
\left(
\lVert \partial_2u(t)\rVert_{L^2}^2
+
\lVert \partial_2B(t)\rVert_{L^2}^2
\right)
\nonumber\\
&+
\int_0^T
e^{-h_{u,B}(t)}
\left(
\nu\lVert \partial_1\partial_2u(t)\rVert_{L^2}^2
+
\eta\lVert \nabla\partial_2B(t)\rVert_{L^2}^2
\right)\,dt.
\label{TA2a}
\end{align}
Set
\begin{align}
A(u,B;T)
:=
A_1(u,B;T)+A_2(u,B;T).
\label{J5r}
\end{align}
For $R>0$ and $T>0$, define
\begin{align}
K_R^T
:=
\Bigl\{
(u,B)\in\mathcal Z:
&A(u,B;T)\le R,
\quad
\nabla\cdot u=0,
\quad
\nabla\cdot B=0
\Bigr\}.
\label{KRT}
\end{align}
We now prove the following two assertions:
\begin{align}
\text{\rm(i)}\quad
&\text{for any }R(T)>0,\qquad
\bigcap_{T=1}^{\infty}K_{R(T)}^T
\text{ is relatively compact in }\mathcal Z,
\label{GTRW}
\\
\text{\rm(ii)}\quad
&\text{for any }\varepsilon>0\text{ and any }T>0,
\text{ there exists }R_0(T)>0
\text{ such that}
\nonumber\\
&\hspace{2.7cm}
\mu_n\!\left(K_{R_0(T)}^T\right)
>
1-\frac{\varepsilon}{2^T},
\qquad n\in\mathbb N.
\label{TGYG}
\end{align}

\noindent
\textit{Proof of \textup{(i)}.}
Since convergence in $\mathcal Z=\mathcal X_u\times\mathcal X_B$
is equivalent to convergence in the spaces appearing in
\eqref{Xu_Tig}-\eqref{XB_Tig} on $[0,T]$ for every $T>0$, it suffices
to prove that, for any $R>0$, $T>0$, and sequence
$\{(u_m,B_m)\}_{m\geq1}\subset K_R^T$, there exists a subsequence
converging in all these spaces on $[0,T]$.

\medskip
\noindent Fix $T>0$ and $R>0$, and consider an arbitrary sequence
$
\{(u_m,B_m)\}_{m\ge1}
\subset
K_R^T.
$
By the definition of $K_R^T$,  \eqref{KRT} implies
\begin{align}
\sup_{0\le t\le T}
\left(
\lVert u_m(t)\rVert_{L^2}^2
+
\lVert B_m(t)\rVert_{L^2}^2
\right)
&\le
R,
\label{YUR1}
\\
\int_0^T
\left(
\nu\lVert \partial_1u_m(t)\rVert_{L^2}^2
+
\eta\lVert \nabla B_m(t)\rVert_{L^2}^2
\right)\,dt
&\le
R,
\label{6twe}
\\
\lVert u_m\rVert_{C^{1/8}([0,T];H^{-1})}
+
\lVert B_m\rVert_{C^{1/8}([0,T];H^{-1})}
&\le
R,
\label{ugew}
\end{align}
and
\begin{align}
&\sup_{0\le t\le T}
e^{-h_{u_m,B_m}(t)}
\left(
\lVert \partial_2u_m(t)\rVert_{L^2}^2
+
\lVert \partial_2B_m(t)\rVert_{L^2}^2
\right)
\nonumber\\
&\quad
+
\int_0^T
e^{-h_{u_m,B_m}(t)}
\left(
\nu\lVert \partial_1\partial_2u_m(t)\rVert_{L^2}^2
+
\eta\lVert \nabla\partial_2B_m(t)\rVert_{L^2}^2
\right)\,dt
\le
R.
\label{qwacl}
\end{align}
Since
$
e^{-C_{R,T,\nu,\eta}}
\le
e^{-h_{u_m,B_m}(t)}
\le
1,
\, t\in[0,T],$
using \eqref{qwacl}, we obtain
\begin{align}
\sup_{0\le t\le T}
\bigl(
\lVert\partial_2u_m(t)\rVert_{L^2}^2
+
\lVert\partial_2B_m(t)\rVert_{L^2}^2
\bigr)
&\le
R e^{C_{R,T,\nu,\eta}},
\label{lKGE}
\\
\int_0^T
\bigl(
\nu\lVert\partial_1\partial_2u_m(t)\rVert_{L^2}^2
+
\eta\lVert\nabla\partial_2B_m(t)\rVert_{L^2}^2
\bigr)\,dt
&\le
R e^{C_{R,T,\nu,\eta}}.
\label{ikgq}
\end{align}
\underline{\textbf{ Uniform bounds for the velocity sequence.}}

By \eqref{YUR1} and \eqref{6twe},
\begin{align}
\int_0^T
\lVert u_m(t)\rVert_{H^{1,0}}^2\,dt
&=
\int_0^T
\left(
\lVert u_m(t)\rVert_{L^2}^2
+
\lVert \partial_1u_m(t)\rVert_{L^2}^2
\right)\,dt
\le
TR+\frac{R}{\nu}
\le
C_{R,T,\nu,\eta}.
\label{u-H10}
\end{align}
By \eqref{YUR1}, \eqref{6twe},
\eqref{lKGE}, and \eqref{ikgq}, we obtain
\begin{align}
\int_0^T
\lVert u_m(t)\rVert_{H^{1,1}}^2\,dt
&=
\int_0^T
\left(
\lVert u_m\rVert_{L^2}^2
+
\lVert \partial_1u_m\rVert_{L^2}^2
+
\lVert \partial_2u_m\rVert_{L^2}^2
+
\lVert \partial_1\partial_2u_m\rVert_{L^2}^2
\right)\,dt
\nonumber\\
&\le
TR+\frac{R}{\nu}
+TC_R+\frac{C_R}{\nu}
\le
C_{R,T,\nu,\eta}.
\label{u-H11}
\end{align}
Similarly, \eqref{YUR1} and \eqref{lKGE} give
\begin{align}
\sup_{0\le t\le T}
\lVert u_m(t)\rVert_{H^{0,1}}^2
&=
\sup_{0\le t\le T}
\left(
\lVert u_m(t)\rVert_{L^2}^2
+
\lVert \partial_2u_m(t)\rVert_{L^2}^2
\right)
\le
R+C_R.
\label{u-H01}
\end{align}
Therefore, \eqref{u-H10}, \eqref{u-H11}, and \eqref{u-H01} imply that
\begin{align}
\{u_m\}_{m\ge1}
\ \text{is bounded in}\
L^2(0,T;H^{1,0})
\cap
L^2(0,T;H^{1,1})
\cap
L^\infty(0,T;H^{0,1}).
\label{uRE}
\end{align}

\medskip
\noindent
\underline{\textbf{ Uniform bounds for the magnetic sequence.}}
Using \eqref{YUR1} and \eqref{6twe}, we obtain
\begin{align}
\int_0^T
\lvert B_m(t)\rVert_{H^1}^2\,dt
&=
\int_0^T
(
\lVert B_m(t)\rVert_{L^2}^2
+
\lvert \nabla B_m(t)\rVert_{L^2}^2
)\,dt
\le
TR+\frac{R}{\eta}
\le
C_{R,T,\nu,\eta}.
\label{B-H1}
\end{align}
Moreover, since
$$
\lVert \partial_1B_m\rVert_{L^2}^2
\le
\lVert \nabla B_m\rVert_{L^2}^2,
\qquad
\lVert \partial_1\partial_2B_m\rVert_{L^2}^2
\le
\lVert \nabla\partial_2B_m\rVert_{L^2}^2,
$$
\eqref{YUR1}, \eqref{6twe},
\eqref{lKGE}, and \eqref{ikgq} yield
\begin{align}
\int_0^T
\lVert B_m(t)\rVert_{H^{1,1}}^2\,dt
&\le
TR+\frac{R}{\eta}
+TC_R+\frac{C_R}{\eta}
\le
C_{R,T,\nu,\eta}.
\label{B-H11}
\end{align}
Similarly, using
$
\lVert \partial_2^2B_m\rVert_{L^2}^2
\le
\lVert \nabla\partial_2B_m\rVert_{L^2}^2,
$
we have
\begin{align}
\int_0^T
\lVert B_m(t)\rVert_{H^{0,2}}^2\,dt
&\le
TR+TC_R+\frac{C_R}{\eta}
\le
C_{R,T,\nu,\eta}.
\label{B-H02}
\end{align}
Finally, \eqref{YUR1} and \eqref{lKGE} give
\begin{align}
\sup_{0\le t\le T}
\lVert B_m(t)\rVert_{H^{0,1}}^2
&=
\sup_{0\le t\le T}
\left(
\lVert B_m(t)\rVert_{L^2}^2
+
\lVert \partial_2B_m(t)\rVert_{L^2}^2
\right)
\le
R+C_R.
\label{B-H01}
\end{align}
Therefore, \eqref{B-H1}, \eqref{B-H11}, \eqref{B-H02}, and \eqref{B-H01} imply that
\begin{align}
\{B_m\}_{m\ge1}
\ \text{is bounded in}\
L^2(0,T;H^1)
\cap
L^2(0,T;H^{1,1})
\cap
L^2(0,T;H^{0,2})
\cap
L^\infty(0,T;H^{0,1}).
\label{rtAA}
\end{align}
\medskip
\noindent
\underline{\textbf{ Compactness in $C([0,T];H^{-1})$.}}

\noindent
By \eqref{YUR1}, for every $t\in[0,T]$,
\begin{align}
\sup_{m\ge1}
\left(
\lVert u_m(t)\rVert_{L^2}^2
+
\lvert B_m(t)\rVert_{L^2}^2
\right)
\le R.
\end{align}
Since
$
L^2(\mathbb T^2)
\hookrightarrow\hookrightarrow
H^{-1}(\mathbb T^2),
$
the sets
$
\{u_m(t):m\ge1\}
$
and
$
\{B_m(t):m\ge1\}$
are relatively compact in $H^{-1}$ for every $t\in[0,T].$

\medskip

\noindent Moreover, \eqref{ugew} gives, for all $s,t\in[0,T]$,
\begin{align}
\lVert u_m(t)-u_m(s)\rVert_{H^{-1}}
&\le R\lvert t-s\rvert^{1/8},
\label{u-C}
\\
\lVert B_m(t)-B_m(s)\rVert_{H^{-1}}
&\le R\lvert t-s\rvert^{1/8}.
\label{B-C}
\end{align}
Thus both sequences are equicontinuous in $H^{-1}$. Hence, by the
Arzel\`a-Ascoli theorem, after passing to a further subsequence,
\begin{align}
u_m
&\longrightarrow u
&&\text{strongly in }C([0,T];H^{-1}),
\label{er2}
\\
B_m
&\longrightarrow B
&&\text{strongly in }C([0,T];H^{-1}).
\label{6ty4av}
\end{align}

\medskip
\noindent
\underline{\textbf{ Strong compactness in $L^2(0,T;H)$.}}

\noindent
We use the interpolation inequality
\begin{align}
\lVert f\rVert_{L^2}^2
\le
\lVert f\rVert_{H^{-1}}\lVert f\rVert_{H^1},
\qquad
f\in H^1(\mathbb T^2).
\label{uikl}
\end{align}
Since
$
H^{1,1}(\mathbb T^2)\hookrightarrow H^1(\mathbb T^2),
$
\eqref{u-H11} and \eqref{B-H1} imply
\begin{align}
\sup_{m\ge1}
\int_0^T
\left(
\lVert u_m(t)\rVert_{H^1}^2
+
\lVert B_m(t)\rVert_{H^1}^2
\right)\,dt
\le
C_{R,T,\nu,\eta}.
\label{ew2A1}
\end{align}
Hence, for $m,\ell\ge1$, by
\eqref{uikl} and the Cauchy-Schwarz inequality,
\begin{align}
\int_0^T
\lVert u_m(t)-u_\ell(t)\rVert_{L^2}^2\,dt
&\le
\sup_{0\le t\le T}
\lVert u_m(t)-u_\ell(t)\rVert_{H^{-1}}
\int_0^T
\lVert u_m(t)-u_\ell(t)\rVert_{H^1}\,dt
\nonumber\\
&\le
C_{R,T,\nu,\eta}
\sup_{0\le t\le T}
\lVert u_m(t)-u_\ell(t)\rVert_{H^{-1}},
\label{6thu}
\\
\int_0^T
\lVert B_m(t)-B_\ell(t)\rVert_{L^2}^2\,dt
&\le
\sup_{0\le t\le T}
\lVert B_m(t)-B_\ell(t)\rVert_{H^{-1}}
\int_0^T
\lVert B_m(t)-B_\ell(t)\rVert_{H^1}\,dt
\nonumber\\
&\le
C_{R,T,\nu,\eta}
\sup_{0\le t\le T}
\lVert B_m(t)-B_\ell(t)\rVert_{H^{-1}}.
\label{5tgh}
\end{align}
By \eqref{er2} and
\eqref{6ty4av}, the right-hand sides of
\eqref{6thu}-\eqref{5tgh} tend to zero as
$m,\ell\to\infty$. Therefore, $\{u_m\}_{m\ge1}$ and
$\{B_m\}_{m\ge1}$ are Cauchy in $L^2(0,T;H)$.
Their limits are $u$ and $B$, respectively, by
\eqref{er2} and \eqref{6ty4av}.
Consequently,
\begin{align}
u_m
&\longrightarrow u
&&\text{strongly in }L^2(0,T;H),
\label{hg4xq}
\\
B_m
&\longrightarrow B
&&\text{strongly in }L^2(0,T;H).
\label{TW1z}
\end{align}
\noindent \underline{\textbf{ Convergence in the full product path space.}}

\noindent By \eqref{uRE}, the sequence $\{u_m\}_{m\ge1}$ is bounded in
$$
L^2(0,T;H^{1,0}),
\qquad
L^2(0,T;H^{1,1}),
\qquad
L^\infty(0,T;H^{0,1});
$$
hence, by reflexivity of the $L^2$-spaces and the Banach-Alaoglu
theorem, there exists a subsequence, still denoted by $\{u_m\}_{m\ge1}$,
such that
\begin{align}
u_m
&\rightharpoonup u
&&\text{weakly in }L^2(0,T;H^{1,0}),
\label{ty4eh}
\\
u_m
&\rightharpoonup u
&&\text{weakly in }L^2(0,T;H^{1,1}),
\label{gh4}
\\
u_m
&\overset{\ast}{\rightharpoonup} u
&&\text{ in }
L^\infty(0,T;H^{0,1}).
\label{87xs}
\end{align}
Similarly, by \eqref{rtAA}, the sequence
$\{B_m\}_{m\ge1}$ is bounded in
$$
L^2(0,T;H^1),
\qquad
L^2(0,T;H^{1,1}),
\qquad
L^2(0,T;H^{0,2}),
\qquad
L^\infty(0,T;H^{0,1}),
$$
hence, by reflexivity and the Banach-Alaoglu theorem, 
\begin{align}
B_m
&\rightharpoonup B
&&\text{weakly in }L^2(0,T;H^1),
\label{j6d}
\\
B_m
&\rightharpoonup B
&&\text{weakly in }L^2(0,T;H^{1,1}),
\label{hu6D4}
\\
B_m
&\rightharpoonup B
&&\text{weakly in }L^2(0,T;H^{0,2}),
\label{97b}
\\
B_m
&\overset{\ast}{\rightharpoonup} B
&&\text{in } L^\infty(0,T;H^{0,1}).
\label{se2}
\end{align}
Thus, for every $R,T>0$, each sequence
$\{(u_m,B_m)\}_{m\ge1}\subset K_R^T$
admits a subsequence satisfying all the above convergences. Consequently, for any $R(T)>0$,
\begin{align}
\bigcap_{T=1}^{\infty}K_{R(T)}^T
\end{align}
is relatively compact in $\mathcal Z$, which proves \eqref{GTRW}.
\medskip

\noindent
\textit{Proof of \textup{(ii)}.}
Fix $\varepsilon>0$ and $T>0$. By \eqref{L2engy},
\begin{align}
\sup_{n\in\mathbb N}
\mathbb E
\Bigg[
&\sup_{0\le t\le T}
\left(
\lVert u^n(t)\rVert_{L^2}^2
+
\lVert B^n(t)\rVert_{L^2}^2
\right)
+
\int_0^T
\left(
\nu\lVert \partial_1u^n(t)\rVert_{L^2}^2
+
\eta\lVert \nabla B^n(t)\rVert_{L^2}^2
\right)\,dt
\Bigg]
\le
C_T,
\label{TBE}
\end{align}
hence, by Chebyshev's inequality, for every $\varepsilon>0$, there
exists $R_1>0$, independent of $n$, such that
\begin{align}
\mathbb P
\Bigg(\sup_{0\le t\le T}
\left(
\lVert u^n(t)\rVert_{L^2}^2
+
\lVert B^n(t)\rVert_{L^2}^2
\right)
+
\int_0^T
\left(
\nu\lVert \partial_1u^n(t)\rVert_{L^2}^2
+
\eta\lVert \nabla B^n(t)\rVert_{L^2}^2
\right)\,dt
>
R_1
\Bigg)
<
\frac{\varepsilon}{3\cdot 2^T}
\label{ncxzJ1}
\end{align}
for every $n\in\mathbb N$.

\noindent Similarly, \eqref{vwre} gives
\begin{align}
\sup_{n\in\mathbb N}
\mathbb E
\Bigg[
&\sup_{0\le t\le T}
e^{-h_n(t)}
\left(
\lVert \partial_2u^n(t)\rVert_{L^2}^2
+
\lVert \partial_2B^n(t)\rVert_{L^2}^2
\right)
\nonumber\\
&\quad
+
\int_0^T
e^{-h_n(t)}
\left(
\nu\lVert \partial_1\partial_2u^n(t)\rVert_{L^2}^2
+
\eta\lVert \nabla\partial_2B^n(t)\rVert_{L^2}^2
\right)\,dt
\Bigg]
\le
C_T.
\label{bhge}
\end{align}
Hence, by Chebyshev's inequality, for every $\varepsilon>0$, there
exists $R_1>0$, independent of $n$, such that
\begin{align}
\mathbb P
\Bigg(
&\sup_{0\le t\le T}
e^{-h_n(t)}
\left(
\lVert \partial_2u^n(t)\rVert_{L^2}^2
+
\lVert \partial_2B^n(t)\rVert_{L^2}^2
\right)
\nonumber\\
&\quad
+
\int_0^T
e^{-h_n(t)}
\left(
\nu\lVert \partial_1\partial_2u^n(t)\rVert_{L^2}^2
+
\eta\lVert \nabla\partial_2B^n(t)\rVert_{L^2}^2
\right)\,dt
>
R_1
\Bigg)
<
\frac{\varepsilon}{3\cdot 2^T}.
\label{76cm}
\end{align} 
\medskip
\noindent It remains to control the time-H\"older norms in $H^{-1}$.
For $0\le s<t\le T$, the Galerkin equations give
\begin{align}
u^n(t)-u^n(s)
&=
\nu\int_s^t
\partial_1^2u^n(r)\,dr
-
\int_s^t
P_n\operatorname{div}
\left(
u^n(r)\otimes u^n(r)
\right)\,dr
+
\int_s^t
P_n\operatorname{div}
\left(
B^n(r)\otimes B^n(r)
\right)\,dr
\nonumber\\
&\quad
+
\int_s^t
P_n\sigma_1(r,u^n(r),B^n(r))\Pi_n^{(1)}\,dW_1(r),
\label{Cdf2}
\\
B^n(t)-B^n(s)
&=
\eta\int_s^t
\Delta B^n(r)\,dr
-
\int_s^t
P_n\operatorname{div}
\left(
u^n(r)\otimes B^n(r)
\right)\,dr
+
\int_s^t
P_n\operatorname{div}
\left(
B^n(r)\otimes u^n(r)
\right)\,dr
\nonumber\\
&\quad
+
\int_s^t
P_n\sigma_2(r,u^n(r),B^n(r))\Pi_n^{(2)}\,dW_2(r).
\label{D1i}
\end{align}
We first estimate the deterministic terms in $H^{-1}$. By integration by parts,
\begin{align}
\lVert \partial_1^2u^n\rVert_{H^{-1}}
&\le
\lVert \partial_1u^n\rVert_{L^2},
\label{Er2}
\\
\lVert \Delta B^n\rVert_{H^{-1}}
&\le
\lVert \nabla B^n\rVert_{L^2}.
\label{W1a}
\end{align}
Moreover, for $a,b\in L^4(\mathbb T^2)$, we obtain
\begin{align}
\lVert \operatorname{div}(a\otimes b)\rVert_{H^{-1}}
\le
C\lVert a\rVert_{L^4}\lVert b\rVert_{L^4}.
\label{ujA1v}
\end{align}
By \eqref{ncxzJ1} and \eqref{76cm}, with probability at least
$1-\frac{2\varepsilon}{3\cdot2^T}$, the same estimates
\eqref{lKGE}, \eqref{ikgq}, \eqref{u-H11}, \eqref{u-H01},
\eqref{B-H11}, and \eqref{B-H01}, with $R$ replaced by $R_1$, yield
\begin{align}
\sup_{0\le t\le T}
\Bigl(
\lVert u^n(t)\rVert_{H^{0,1}}^2
+
\lVert B^n(t)\rVert_{H^{0,1}}^2
\Bigr)
&\le
C_{R_1,T,\nu,\eta},
\label{7gE3}
\\
\int_0^T
\Bigl(
\lVert u^n(t)\rVert_{H^{1,1}}^2
+
\lVert B^n(t)\rVert_{H^{1,1}}^2
\Bigr)\,dt
&\le
C_{R_1,T,\nu,\eta}.
\label{78hER}
\end{align}
By interpolation between $H^{0,1}$ and $H^{1,1}$,
along with
$
H^{1/2,1}(\mathbb T^2)
\hookrightarrow
H^{1/2}(\mathbb T^2)
\hookrightarrow
L^4(\mathbb T^2),
$
we have
\begin{align}
\lVert v\rVert_{L^4}^4
\le
C
\lVert v\rVert_{H^{0,1}}^2
\lVert v\rVert_{H^{1,1}}^2.
\label{56fa}
\end{align}
Hence, by \eqref{7gE3} and \eqref{78hER},
\begin{align}
\int_0^T
\lVert u^n(t)\rVert_{L^4}^4\,dt
&\le
C
\sup_{0\le t\le T}
\lVert u^n(t)\rVert_{H^{0,1}}^2
\int_0^T
\lVert u^n(t)\rVert_{H^{1,1}}^2\,dt
\le
C_{R_1,T,\nu,\eta},
\label{56sza}
\\
\int_0^T
\lVert B^n(t)\rVert_{L^4}^4\,dt
&\le
C
\sup_{0\le t\le T}
\lVert B^n(t)\rVert_{H^{0,1}}^2
\int_0^T
\lVert B^n(t)\rVert_{H^{1,1}}^2\,dt
\le
C_{R_1,T,\nu,\eta}.
\label{ty6m}
\end{align}
It follows from \eqref{ujA1v} and
\eqref{56sza}-\eqref{ty6m} that
\begin{align}
\int_0^T
\lVert
\operatorname{div}(u^n\otimes u^n)
\rVert_{H^{-1}}^2\,dt
&\le
C\int_0^T
\lVert u^n\rVert_{L^4}^4\,dt
\le
C_{R_1,T,\nu,\eta},
\label{54fv}
\\
\int_0^T
\lVert
\operatorname{div}(B^n\otimes B^n)
\rVert_{H^{-1}}^2\,dt
&\le
C\int_0^T
\lVert B^n\rVert_{L^4}^4\,dt
\le
C_{R_1,T,\nu,\eta},
\label{rf458}
\\
\int_0^T
\lVert
\operatorname{div}(u^n\otimes B^n)
\rVert_{H^{-1}}^2\,dt
&\le
C\int_0^T
\lVert u^n\rVert_{L^4}^2
\lVert B^n\rVert_{L^4}^2\,dt
\nonumber\\
&\le
C
\left(
\int_0^T
\lVert u^n\rVert_{L^4}^4\,dt
\right)^{1/2}
\left(
\int_0^T
\lVert B^n\rVert_{L^4}^4\,dt
\right)^{1/2}
\nonumber\\
&\le
C_{R_1,T,\nu,\eta},
\label{d45b}
\\
\int_0^T
\lVert
\operatorname{div}(B^n\otimes u^n)
\rVert_{H^{-1}}^2\,dt
&\le
C\int_0^T
\lVert B^n\rVert_{L^4}^2
\lVert u^n\rVert_{L^4}^2\,dt
\nonumber\\
&\le
C
\left(
\int_0^T
\lVert B^n\rVert_{L^4}^4\,dt
\right)^{1/2}
\left(
\int_0^T
\lVert u^n\rVert_{L^4}^4\,dt
\right)^{1/2}
\nonumber\\
&\le
C_{R_1,T,\nu,\eta}.
\label{r32a}
\end{align}
Using \eqref{Er2}, \eqref{W1a},
\eqref{YUR1}, \eqref{6twe}, and
\eqref{54fv}-\eqref{r32a}, together with
$
\lVert P_n f\rVert_{H^{-1}}
\le
\lVert f\rVert_{H^{-1}},
$
we obtain
\begin{align}
&\int_0^T
\Bigl\lVert
\nu\partial_1^2u^n
-
P_n\operatorname{div}(u^n\otimes u^n)
+
P_n\operatorname{div}(B^n\otimes B^n)
\Bigr\rVert_{H^{-1}}^2\,dt
\nonumber\\
&\quad+
\int_0^T
\Bigl\lVert
\eta\Delta B^n
-
P_n\operatorname{div}(u^n\otimes B^n)
+
P_n\operatorname{div}(B^n\otimes u^n)
\Bigr\rVert_{H^{-1}}^2\,dt
\le
C_{R_1,T,\nu,\eta}.
\label{56FG}
\end{align}
Consequently, for $0\le s<t\le T$,  \eqref{56FG} gives
\begin{align}
&\int_s^t
\Bigl\lVert
\nu\partial_1^2u^n(r)
-
P_n\operatorname{div}(u^n(r)\otimes u^n(r))
+
P_n\operatorname{div}(B^n(r)\otimes B^n(r))
\Bigr\rVert_{H^{-1}}^2\,dr
\le
C_{R_1,T,\nu,\eta},
\label{6thm}
\\
&\int_s^t
\Bigl\lVert
\eta\Delta B^n(r)
-
P_n\operatorname{div}(u^n(r)\otimes B^n(r))
+
P_n\operatorname{div}(B^n(r)\otimes u^n(r))
\Bigr\rVert_{H^{-1}}^2\,dr
\le
C_{R_1,T,\nu,\eta}.
\label{564}
\end{align}
Therefore, by H\"older's inequality in time, \eqref{6thm} and \eqref{564} yield
\begin{align}
&\Biggl\lVert
\int_s^t
\Bigl[
\nu\partial_1^2u^n(r)
-
P_n\operatorname{div}(u^n(r)\otimes u^n(r))
+
P_n\operatorname{div}(B^n(r)\otimes B^n(r))
\Bigr]\,dr
\Biggr\rVert_{H^{-1}}
\nonumber\\
&\quad\le
\lvert t-s\rvert^{1/2}
\left(
\int_s^t
\Bigl\lVert
\nu\partial_1^2u^n(r)
-
P_n\operatorname{div}(u^n(r)\otimes u^n(r))
+
P_n\operatorname{div}(B^n(r)\otimes B^n(r))
\Bigr\rVert_{H^{-1}}^2\,dr
\right)^{1/2}
\nonumber\\
&\quad\le
C_{R_1,T,\nu,\eta}\lvert t-s\rvert^{1/2},
\label{hlub3}
\\
&\Biggl\lVert
\int_s^t
\Bigl[
\eta\Delta B^n(r)
-
P_n\operatorname{div}(u^n(r)\otimes B^n(r))
+
P_n\operatorname{div}(B^n(r)\otimes u^n(r))
\Bigr]\,dr
\Biggr\rVert_{H^{-1}}
\nonumber\\
&\quad\le
\lvert t-s\rvert^{1/2}
\left(
\int_s^t
\Bigl\lVert
\eta\Delta B^n(r)
-
P_n\operatorname{div}(u^n(r)\otimes B^n(r))
+
P_n\operatorname{div}(B^n(r)\otimes u^n(r))
\Bigr\rVert_{H^{-1}}^2\,dr
\right)^{1/2}
\nonumber\\
&\quad\le
C_{R_1,T,\nu,\eta}\lvert t-s\rvert^{1/2}.
\label{Bhhy5}
\end{align}
We next estimate the stochastic increments. By the Burkholder-Davis-Gundy
inequality, the uniform boundedness of $P_n$ on $H^{-1}$, the
$H^{-1}$-growth assumption \eqref{NH1}, H\"older's inequality, and
\eqref{FL2}, for $0\le s<t\le T$,
\begin{align}
&\mathbb E
\left\lVert
\int_s^t
P_n\sigma_1(r,u^n(r),B^n(r))\Pi_n^{(1)}\,dW_1(r)
\right\rVert_{H^{-1}}^4
\nonumber\\
&\quad\le
C\mathbb E
\left(
\int_s^t
\left\lVert
\sigma_1(r,u^n(r),B^n(r))
\right\rVert_{L_2(\mathfrak U_1;H^{-1})}^2\,dr
\right)^2
\nonumber\\
&\quad\le
C\lvert t-s\rvert
\int_s^t
\mathbb E
\left[
K_0'
+
K_1'
\bigl(
\lVert u^n(r)\rVert_{L^2}^2
+
\lVert B^n(r)\rVert_{L^2}^2
\bigr)
\right]^2\,dr
\nonumber\\
&\quad\le
C\lvert t-s\rvert
\int_s^t
\left[
1+
\mathbb E
\sup_{0\le\tau\le T}
\bigl(
\lVert u^n(\tau)\rVert_{L^2}^2
+
\lVert B^n(\tau)\rVert_{L^2}^2
\bigr)^2
\right]\,dr
\nonumber\\
&\quad\le
C_T\lvert t-s\rvert^2.
\label{9ngfe}
\end{align}
The same argument like \eqref{9ngfe} gives
\begin{align}
\mathbb E
\left\lVert
\int_s^t
P_n\sigma_2(r,u^n(r),B^n(r))\Pi_n^{(2)}\,dW_2(r)
\right\rVert_{H^{-1}}^4
\le
C_T\lvert t-s\rvert^2.
\label{yba2q}
\end{align}
By \eqref{9ngfe} and \eqref{yba2q}, for every $0<\alpha<1/4$, the Kolmogorov continuity theorem yields, 
\begin{equation}
\begin{aligned}
&\sup_{n\in\mathbb N}
\mathbb E
\Bigg[
\sup_{0\le s<t\le T}
\frac{
\left\lVert
\displaystyle
\int_s^t
P_n\sigma_1(r,u^n(r),B^n(r))\Pi_n^{(1)}\,dW_1(r)
\right\rVert_{H^{-1}}^4
}{
\lvert t-s\rvert^{4\alpha}
}
\Bigg]
\le C_{T,\alpha},
\\
&\sup_{n\in\mathbb N}
\mathbb E
\Bigg[
\sup_{0\le s<t\le T}
\frac{
\left\lVert
\displaystyle
\int_s^t
P_n\sigma_2(r,u^n(r),B^n(r))\Pi_n^{(2)}\,dW_2(r)
\right\rVert_{H^{-1}}^4
}{
\lvert t-s\rvert^{4\alpha}
}
\Bigg]
\le C_{T,\alpha}.
\label{ASK2}
\end{aligned}
\end{equation}
Taking $\alpha=1/8$ in \eqref{ASK2}, we obtain
\begin{align}
&\sup_{n\in\mathbb N}
\mathbb E
\Bigg[
\sup_{0\le s<t\le T}
\frac{
\left\lVert
\displaystyle
\int_s^t
P_n\sigma_1(r,u^n(r),B^n(r))\Pi_n^{(1)}\,dW_1(r)
\right\rVert_{H^{-1}}^4
}{
\lvert t-s\rvert^{1/2}
}
\Bigg]
\le C_T,
\label{7hv4V4}
\\
&\sup_{n\in\mathbb N}
\mathbb E
\Bigg[
\sup_{0\le s<t\le T}
\frac{
\left\lVert
\displaystyle
\int_s^t
P_n\sigma_2(r,u^n(r),B^n(r))\Pi_n^{(2)}\,dW_2(r)
\right\rVert_{H^{-1}}^4
}{
\lvert t-s\rvert^{1/2}
}
\Bigg]
\le C_T.
\label{yuRw}
\end{align}
Let $R_1>0$ be chosen as in \eqref{ncxzJ1} and \eqref{76cm}, and set
\begin{align}
K_{R_1}^n
:=
\Bigg\{
&
\sup_{0\le t\le T}
\bigl(
\lVert u^n(t)\rVert_{L^2}^2
+
\lVert B^n(t)\rVert_{L^2}^2
\bigr)
+
\int_0^T
\bigl(
\nu\lVert\partial_1u^n(t)\rVert_{L^2}^2
+
\eta\lVert\nabla B^n(t)\rVert_{L^2}^2
\bigr)\,dt
\le R_1,
\nonumber\\
&
\sup_{0\le t\le T}
e^{-h_n(t)}
\bigl(
\lVert\partial_2u^n(t)\rVert_{L^2}^2
+
\lVert\partial_2B^n(t)\rVert_{L^2}^2
\bigr)
\nonumber\\
&\qquad+
\int_0^T
e^{-h_n(t)}
\bigl(
\nu\lVert\partial_1\partial_2u^n(t)\rVert_{L^2}^2
+
\eta\lVert\nabla\partial_2B^n(t)\rVert_{L^2}^2
\bigr)\,dt
\le R_1
\Bigg\}.
\end{align}
Since
$
\lvert t-s\rvert^{1/2}
\le
T^{3/8}\lvert t-s\rvert^{1/8},
\quad 0\le s<t\le T,
$
\eqref{hlub3} and \eqref{Bhhy5} yield,
\begin{align}
&\sup_{0\le s<t\le T}
\frac{
\left\lVert
\displaystyle
\int_s^t
\Bigl[
\nu\partial_1^2u^n(r)
-
P_n\operatorname{div}\bigl(u^n(r)\otimes u^n(r)\bigr)
+
P_n\operatorname{div}\bigl(B^n(r)\otimes B^n(r)\bigr)
\Bigr]\,dr
\right\rVert_{H^{-1}}
}{
\lvert t-s\rvert^{1/8}
}
\le
C_{R_1,T,\nu,\eta},
\label{drv2}
\\
&\sup_{0\le s<t\le T}
\frac{
\left\lVert
\displaystyle
\int_s^t
\Bigl[
\eta\Delta B^n(r)
-
P_n\operatorname{div}\bigl(u^n(r)\otimes B^n(r)\bigr)
+
P_n\operatorname{div}\bigl(B^n(r)\otimes u^n(r)\bigr)
\Bigr]\,dr
\right\rVert_{H^{-1}}
}{
\lvert t-s\rvert^{1/8}
}
\le
C_{R_1,T,\nu,\eta}.
\label{Df2}
\end{align}
Combining \eqref{Cdf2} with \eqref{drv2}, and
\eqref{D1i} with \eqref{Df2}, gives
\begin{align}
&\lVert u^n\rVert_{C^{1/8}([0,T];H^{-1})}
+
\lVert B^n\rVert_{C^{1/8}([0,T];H^{-1})}
\nonumber\\
&\quad\le
C_{R_1,T,\nu,\eta}
+
\sup_{0\le s<t\le T}
\frac{
\left\lVert
\displaystyle
\int_s^t
P_n\sigma_1(r,u^n(r),B^n(r))\Pi_n^{(1)}\,dW_1(r)
\right\rVert_{H^{-1}}
}{
\lvert t-s\rvert^{1/8}
}
\nonumber\\
&\qquad+
\sup_{0\le s<t\le T}
\frac{
\left\lVert
\displaystyle
\int_s^t
P_n\sigma_2(r,u^n(r),B^n(r))\Pi_n^{(2)}\,dW_2(r)
\right\rVert_{H^{-1}}
}{
\lvert t-s\rvert^{1/8}
}.
\label{H1F1E}
\end{align}
Multiplying \eqref{H1F1E} by
$\mathbf 1_{K_{R_1}^n}$, taking expectations, and using
H\"older's inequality together with \eqref{7hv4V4} and
\eqref{yuRw}, we obtain
\begin{align}
\sup_{n\in\mathbb N}
\mathbb E
\Bigg[
\left(
\lVert u^n\rVert_{C^{1/8}([0,T];H^{-1})}
+
\lVert B^n\rVert_{C^{1/8}([0,T];H^{-1})}
\right)
\mathbf 1_{K_{R_1}^n}
\Bigg]
\le
C_{R_1,T,\nu,\eta}.
\label{MHDH1Bs}
\end{align}
By \eqref{ncxzJ1}, \eqref{76cm}, and Boole's inequality, we obtain,
\begin{align}
\mathbb P\bigl((K_{R_1}^n)^c\bigr)
<
\frac{2\varepsilon}{3\cdot2^T},
\qquad n\in\mathbb N.
\label{T1G1H9}
\end{align}
By Chebyshev's inequality and \eqref{MHDH1Bs}, we may
choose $R_2>0$, independent of $n$, such that
\begin{align}
&\mathbb P
\Bigg(
\Bigl\{
\lVert u^n\rVert_{C^{1/8}([0,T];H^{-1})}
+
\lVert B^n\rVert_{C^{1/8}([0,T];H^{-1})}
>R_2
\Bigr\}
\cap
K_{R_1}^n
\Bigg)
\nonumber\\
&\quad\le
\frac{1}{R_2}
\mathbb E
\Bigg[
\left(
\lVert u^n\rVert_{C^{1/8}([0,T];H^{-1})}
+
\lVert B^n\rVert_{C^{1/8}([0,T];H^{-1})}
\right)
\mathbf 1_{K_{R_1}^n}
\Bigg]
\nonumber\\
&\quad\le
\frac{C_{R_1,T,\nu,\eta}}{R_2}
<
\frac{\varepsilon}{3\cdot2^T},
\qquad n\in\mathbb N.
\label{TyPy4}
\end{align}
Set $R_0(T):=2R_1+R_2$. By the definition of
$K_{R_0(T)}^T$, together with Boole's inequality,
\eqref{T1G1H9}, and \eqref{TyPy4}, we obtain
\begin{align}
\mu_n(K_{R_0(T)}^T)
>
1-\frac{2\varepsilon}{3\cdot2^T}
-\frac{\varepsilon}{3\cdot2^T}
=
1-\frac{\varepsilon}{2^T},
\qquad n\in\mathbb N.
\label{TLRE1}
\end{align}
Thus \eqref{TGYG} is proved. Combining \eqref{GTRW} and
\eqref{TGYG}, the family
$\{\mu_n\}_{n\in\mathbb N}$ is tight on the global path space
$\mathcal Z$.
\end{proof}
\subsection*{Skorokhod-Jakubowski representation}
\noindent After proving the tightness of the laws of the Galerkin approximations,
we apply the Skorokhod-Jakubowski representation theorem to obtain
almost sure convergence on a new probability space.
By the tightness result established above, ${\mu_n}_{n\geq1}$ is tight on $\mathcal Z$. Since $\mathcal Z$ admits a countable family of continuous functions separating points, the Skorokhod-Jakubowski representation theorem yields a probability space $(\widetilde\Omega,\widetilde{\mathcal F},\widetilde{\mathbb P})$, a subsequence, still indexed by $n$, and $\mathcal Z$-valued random variables $(\widetilde u^n,\widetilde B^n)$ and $(\widetilde u,\widetilde B)$ such that
\begin{enumerate}
\item[(i)]
\begin{align}
\mathcal L_{\widetilde{\mathbb P}}
(\widetilde u^n,\widetilde B^n)
&=\mu_n,
\qquad n\ge1,
\label{yuh5}
\end{align}

\item[(ii)]
$\mu_n\rightharpoonup\widetilde\mu$ for some probability measure
$\widetilde\mu$;

\item[(iii)]
\begin{align}
(\widetilde u^n,\widetilde B^n)
&\longrightarrow
(\widetilde u,\widetilde B)
\qquad
\widetilde{\mathbb P}\text{-a.s. in }\mathcal Z,
\label{AST}
\end{align}
and
$\mathcal L_{\widetilde{\mathbb P}}
(\widetilde u,\widetilde B)=\widetilde\mu$.
\end{enumerate}
From \eqref{AST}, we have,
\begin{align}
\widetilde u^n
&\longrightarrow\widetilde u
&&\text{strongly in }C([0,T];H^{-1}),
\label{L1tildn}
\\
\widetilde B^n
&\longrightarrow\widetilde B
&&\text{strongly in }C([0,T];H^{-1}),
\label{L2bsH-1}
\\
\widetilde u^n
&\longrightarrow\widetilde u
&&\text{strongly in }L^2(0,T;H),
\label{TUL2H}
\\
\widetilde B^n
&\longrightarrow\widetilde B
&&\text{strongly in }L^2(0,T;H),
\label{TBL2H2}
\\
\widetilde u^n
&\rightharpoonup\widetilde u
&&\text{weakly in }L^2(0,T;H^{1,0}),
\label{G4}
\\
\widetilde u^n
&\rightharpoonup\widetilde u
&&\text{weakly in }L^2(0,T;H^{1,1}),
\label{we34}
\\
\widetilde u^n
&\overset{\ast}{\rightharpoonup}\widetilde u
&&\text{in }L^\infty(0,T;H^{0,1}),
\label{8b23z}
\\
\widetilde B^n
&\rightharpoonup\widetilde B
&&\text{weakly in }L^2(0,T;H^1),
\label{H56d}
\\
\widetilde B^n
&\rightharpoonup\widetilde B
&&\text{weakly in }L^2(0,T;H^{1,1}),
\label{H6yu}
\\
\widetilde B^n
&\rightharpoonup\widetilde B
&&\text{weakly in }L^2(0,T;H^{0,2}),
\label{H123}
\\
\widetilde B^n
&\overset{\ast}{\rightharpoonup}\widetilde B
&&\text{in }L^\infty(0,T;H^{0,1}),
\label{y6Fr}
\end{align}
$\widetilde{\mathbb P}$-a.s.
By \eqref{yuh5}, the a priori estimates for
$(u^n,B^n)$ transfer to
$(\widetilde u^n,\widetilde B^n)$. Hence
\eqref{L2engy}, \eqref{L^4-mo}, and \eqref{vwre} yield
\begin{align}
\sup_{n\ge1}\widetilde{\mathbb E}
\Bigg[
&\sup_{0\le t\le T}
\bigl(
\lVert\widetilde u^n(t)\rVert_{L^2}^2
+
\lVert\widetilde B^n(t)\rVert_{L^2}^2
\bigr)+
\int_0^T
\bigl(
\nu\lVert\partial_1\widetilde u^n\rVert_{L^2}^2
+
\eta\lVert\nabla\widetilde B^n\rVert_{L^2}^2
\bigr)\,dt
\Bigg]
\le C_T,
\label{olTilde}
\\
\sup_{n\ge1}\widetilde{\mathbb E}
\Bigg[
&\sup_{0\le t\le T}
\bigl(
\lVert\widetilde u^n(t)\rVert_{L^2}^2
+
\lVert\widetilde B^n(t)\rVert_{L^2}^2
\bigr)^2
\nonumber\\
&+
\int_0^T
\bigl(
\lVert\widetilde u^n\rVert_{L^2}^2
+
\lVert\widetilde B^n\rVert_{L^2}^2
\bigr)
\bigl(
\nu\lVert\partial_1\widetilde u^n\rVert_{L^2}^2
+
\eta\lVert\nabla\widetilde B^n\rVert_{L^2}^2
\bigr)\,dt
\Bigg]
\le C_T,
\label{Tilvg}
\\
\sup_{n\ge1}\widetilde{\mathbb E}
\Bigg[
&\sup_{0\le t\le T}
e^{-\widetilde h_n(t)}
\bigl(
\lVert\partial_2\widetilde u^n(t)\rVert_{L^2}^2
+
\lVert\partial_2\widetilde B^n(t)\rVert_{L^2}^2
\bigr)
\nonumber\\
&+
\int_0^T e^{-\widetilde h_n(t)}
\bigl(
\nu\lVert\partial_1\partial_2\widetilde u^n\rVert_{L^2}^2
+
\eta\lVert\nabla\partial_2\widetilde B^n\rVert_{L^2}^2
\bigr)\,dt
\Bigg]
\le C_T.
\label{ujTILF}
\end{align}

By \eqref{G4}-\eqref{y6Fr},
\eqref{olTilde}-\eqref{ujTILF}, weak
and $\textrm{weak}^{\ast}$ lower semicontinuity,
\begin{align}
\widetilde{\mathbb E}
\int_0^T
\lVert\widetilde u(t)\rVert_{H^{1,0}}^2\,dt
&\le C_T,
\label{L1}
\\
\widetilde{\mathbb E}
\int_0^T
\lVert\widetilde u(t)\rVert_{H^{1,1}}^2\,dt
&\le C_T,
\label{L2h11}
\\
\widetilde{\mathbb E}
\lVert\widetilde u\rVert_{L^\infty(0,T;H^{0,1})}^2
&\le C_T,
\label{L3H01}
\\
\widetilde{\mathbb E}
\int_0^T
\lVert\widetilde B(t)\rVert_{H^1}^2\,dt
&\le C_T,
\label{L7H1}
\\
\widetilde{\mathbb E}
\int_0^T
\lVert\widetilde B(t)\rVert_{H^{1,1}}^2\,dt
&\le C_T,
\label{L4Bh11}
\\
\widetilde{\mathbb E}
\int_0^T
\lVert\widetilde B(t)\rVert_{H^{0,2}}^2\,dt
&\le C_T,
\label{L8bH02}
\\
\widetilde{\mathbb E}
\lVert\widetilde B\rVert_{L^\infty(0,T;H^{0,1})}^2
&\le C_T,
\label{L5H01}
\end{align}
consequently,
\begin{align}
\widetilde u
&\in
C([0,T];H^{-1})
\cap L^2(0,T;H)
\cap L^2(0,T;H^{1,0})
\cap L^2(0,T;H^{1,1})
\cap L^\infty(0,T;H^{0,1}),
\quad
\widetilde{\mathbb P}\text{-a.s.},
\label{L1C2}
\\
\widetilde B
&\in
C([0,T];H^{-1})
\cap L^2(0,T;H)
\cap L^2(0,T;H^1)
\cap L^2(0,T;H^{1,1})
\cap L^2(0,T;H^{0,2})
\nonumber\\
&
\cap L^\infty(0,T;H^{0,1}),
\quad
\widetilde{\mathbb P}\text{-a.s.}
\label{K1M3}
\end{align}
From \eqref{L1C2},
\eqref{K1M3}, and \eqref{Xu_Tig}-\eqref{XB_Tig}, we obtain
\begin{align}
\widetilde u\in\mathcal X_u,\qquad
\widetilde B\in\mathcal X_B,\qquad
(\widetilde u,\widetilde B)\in\mathcal Z,
\qquad
\widetilde{\mathbb P}\text{-a.s.}
\end{align}
\paragraph{Strong convergence in $L^2(0,T;L^4)$.}
Since weakly convergent sequences are bounded, \eqref{we34} and
\eqref{H6yu}, together with
$H^{1,1}(\mathbb T^2)\hookrightarrow H^1(\mathbb T^2)$, yield
$\sup_{n\ge1}\lVert\widetilde u^n-\widetilde u\rVert_{L^2(0,T;H^1)}<\infty$
and
$\sup_{n\ge1}\lVert\widetilde B^n-\widetilde B\rVert_{L^2(0,T;H^1)}<\infty$.
By the two-dimensional Gagliardo-Nirenberg inequality and
Cauchy-Schwarz in time,
\begin{align}
\lVert\widetilde u^n-\widetilde u\rVert_{L^2(0,T;L^4)}^2
&\le
C
\lVert\widetilde u^n-\widetilde u\rVert_{L^2(0,T;L^2)}
\lVert\widetilde u^n-\widetilde u\rVert_{L^2(0,T;H^1)}
\longrightarrow0,
\\
\lVert\widetilde B^n-\widetilde B\rVert_{L^2(0,T;L^4)}^2
&\le
C
\lVert\widetilde B^n-\widetilde B\rVert_{L^2(0,T;L^2)}
\lVert\widetilde B^n-\widetilde B\rVert_{L^2(0,T;H^1)}
\longrightarrow 0,
\end{align}
from where we get more precisely
\begin{align}
\widetilde u^n
&\longrightarrow
\widetilde u
&&\text{strongly in }L^2(0,T;L^4),
\label{KLNG6}
\\
\widetilde B^n
&\longrightarrow
\widetilde B
&&\text{strongly in }L^2(0,T;L^4),
\label{km56C}
\end{align}
$\widetilde{\mathbb P}$-a.s. 

\noindent Applying H\"older's inequality, using $
\sup_{n\ge1}
\lVert\widetilde u^n\rVert_{L^2(0,T;L^4)}
<\infty$ and $
\sup_{n\ge1}
\lVert\widetilde B^n\rVert_{L^2(0,T;L^4)}
<\infty,$  \eqref{KLNG6}, \eqref{km56C}, 
 we obtain
\begin{align}
&
\lVert
\widetilde u^n\otimes\widetilde u^n
-
\widetilde u\otimes\widetilde u
\rVert_{L^1(0,T;L^2)}
\nonumber\\
&\qquad\le
\lVert\widetilde u^n-\widetilde u\rVert_{L^2(0,T;L^4)}
\bigl(
\lVert\widetilde u^n\rVert_{L^2(0,T;L^4)}
+
\lVert\widetilde u\rVert_{L^2(0,T;L^4)}
\bigr)
\longrightarrow0,
\\
&
\lVert
\widetilde B^n\otimes\widetilde B^n
-
\widetilde B\otimes\widetilde B
\rVert_{L^1(0,T;L^2)}
\nonumber\\
&\qquad\le
\lVert\widetilde B^n-\widetilde B\rVert_{L^2(0,T;L^4)}
\bigl(
\lVert\widetilde B^n\rVert_{L^2(0,T;L^4)}
+
\lVert\widetilde B\rVert_{L^2(0,T;L^4)}
\bigr)
\longrightarrow0,
\\
&
\lVert
\widetilde u^n\otimes\widetilde B^n
-
\widetilde u\otimes\widetilde B
\rVert_{L^1(0,T;L^2)}
\nonumber\\
&\qquad\le
\lVert\widetilde u^n-\widetilde u\rVert_{L^2(0,T;L^4)}
\lVert\widetilde B^n\rVert_{L^2(0,T;L^4)}
+
\lVert\widetilde u\rVert_{L^2(0,T;L^4)}
\lVert\widetilde B^n-\widetilde B\rVert_{L^2(0,T;L^4)}
\longrightarrow0,
\\
&
\lVert
\widetilde B^n\otimes\widetilde u^n
-
\widetilde B\otimes\widetilde u
\rVert_{L^1(0,T;L^2)}
\nonumber\\
&\qquad\le
\lVert\widetilde B^n-\widetilde B\rVert_{L^2(0,T;L^4)}
\lVert\widetilde u^n\rVert_{L^2(0,T;L^4)}
+
\lVert\widetilde B\rVert_{L^2(0,T;L^4)}
\lVert\widetilde u^n-\widetilde u\rVert_{L^2(0,T;L^4)}
\longrightarrow0,
\end{align}
from where  we obtain precisely
\begin{align}
\widetilde u^n\otimes\widetilde u^n
&\longrightarrow
\widetilde u\otimes\widetilde u
&&\text{strongly in }L^1(0,T;L^2),
\label{UUREW}
\\
\widetilde B^n\otimes\widetilde B^n
&\longrightarrow
\widetilde B\otimes\widetilde B
&&\text{strongly in }L^1(0,T;L^2),
\label{BBCRT}
\\
\widetilde u^n\otimes\widetilde B^n
&\longrightarrow
\widetilde u\otimes\widetilde B
&&\text{strongly in }L^1(0,T;L^2),
\label{TTERDS}
\\
\widetilde B^n\otimes\widetilde u^n
&\longrightarrow
\widetilde B\otimes\widetilde u
&&\text{strongly in }L^1(0,T;L^2).
\label{BYTW}
\end{align}
$\widetilde{\mathbb P}$-a.s.
In addition, using $\lVert \operatorname{div} A\rVert_{H^{-1}}\leq \lVert A\rVert_{L^2}$, \eqref{UUREW}, \eqref{BBCRT}, \eqref{TTERDS}, and \eqref{BYTW},
we obtain
\begin{align}
&
\lVert
\operatorname{div}
(\widetilde u^n\otimes\widetilde u^n)
-
\operatorname{div}
(\widetilde u\otimes\widetilde u)
\rVert_{L^1(0,T;H^{-1})}
\le
\lVert
\widetilde u^n\otimes\widetilde u^n
-
\widetilde u\otimes\widetilde u
\rVert_{L^1(0,T;L^2)}
\longrightarrow0,
\\
&
\lVert
\operatorname{div}
(\widetilde B^n\otimes\widetilde B^n)
-
\operatorname{div}
(\widetilde B\otimes\widetilde B)
\rVert_{L^1(0,T;H^{-1})}\le
\lVert
\widetilde B^n\otimes\widetilde B^n
-
\widetilde B\otimes\widetilde B
\rVert_{L^1(0,T;L^2)}
\longrightarrow0,
\\
&
\lVert
\operatorname{div}
(\widetilde u^n\otimes\widetilde B^n)
-
\operatorname{div}
(\widetilde u\otimes\widetilde B)
\rVert_{L^1(0,T;H^{-1})}
\le
\lVert
\widetilde u^n\otimes\widetilde B^n
-
\widetilde u\otimes\widetilde B
\rVert_{L^1(0,T;L^2)}
\longrightarrow0,
\\
&
\lVert
\operatorname{div}
(\widetilde B^n\otimes\widetilde u^n)
-
\operatorname{div}
(\widetilde B\otimes\widetilde u)
\rVert_{L^1(0,T;H^{-1})}
\le
\lVert
\widetilde B^n\otimes\widetilde u^n
-
\widetilde B\otimes\widetilde u
\rVert_{L^1(0,T;L^2)}
\longrightarrow0,
\end{align}
therefore, $\widetilde{\mathbb P}$-a.s., we obtain
\begin{align}
\operatorname{div}
(\widetilde u^n\otimes\widetilde u^n)
&\longrightarrow
\operatorname{div}
(\widetilde u\otimes\widetilde u)
&&\text{strongly in }L^1(0,T;H^{-1}),
\label{URE76}
\\
\operatorname{div}
(\widetilde B^n\otimes\widetilde B^n)
&\longrightarrow
\operatorname{div}
(\widetilde B\otimes\widetilde B)
&&\text{strongly in }L^1(0,T;H^{-1}),
\label{BBrt3}
\\
\operatorname{div}
(\widetilde u^n\otimes\widetilde B^n)
&\longrightarrow
\operatorname{div}
(\widetilde u\otimes\widetilde B)
&&\text{strongly in }L^1(0,T;H^{-1}),
\label{uy56}
\\
\operatorname{div}
(\widetilde B^n\otimes\widetilde u^n)
&\longrightarrow
\operatorname{div}
(\widetilde B\otimes\widetilde u)
&&\text{strongly in }L^1(0,T;H^{-1}).
\label{BUR3w}
\end{align}
We first consider the first term on the right-hand side of \eqref{KL4}, 
\begin{align*}
&
\int_0^t
\left\langle
\partial_1^2\widetilde u^n(s),
P_n\phi
\right\rangle_{H^{-1},H^1}\,ds
-
\int_0^t
\left\langle
\partial_1^2\widetilde u(s),
\phi
\right\rangle_{H^{-1},H^1}\,ds
\nonumber\\
&=
\int_0^t
\left\langle
\partial_1^2\widetilde u^n(s)
-
\partial_1^2\widetilde u(s),
\phi
\right\rangle_{H^{-1},H^1}\,ds
+
\int_0^t
\left\langle
\partial_1^2\widetilde u^n(s),
P_n\phi-\phi
\right\rangle_{H^{-1},H^1}\,ds,
\end{align*}
where we used
$\partial_1^2\widetilde u^n\rightharpoonup
\partial_1^2\widetilde u$ weakly in $L^2(0,T;H^{-1})$ for the first
term on the right-hand side; and $P_n\phi\to\phi$ strongly in $H^1$,
together with the boundedness of
$\{\partial_1^2\widetilde u^n\}_{n\geq1}$ in
$L^2(0,T;H^{-1})$, for the second term on the right-hand side, which
imply that
\begin{align}
\int_0^t
\left\langle
\partial_1^2\widetilde u^n(s),
P_n\phi
\right\rangle_{H^{-1},H^1}\,ds
\longrightarrow
\int_0^t
\left\langle
\partial_1^2\widetilde u(s),
\phi
\right\rangle_{H^{-1},H^1}\,ds.
\label{GHTdoF}
\end{align}
For the 2nd term on the right hand side of \eqref{KL4}, 
\begin{align*}
&
\int_0^t
\left\langle
\operatorname{div}
(\widetilde u^n\otimes\widetilde u^n)(s),
P_n\phi
\right\rangle_{H^{-1},H^1}\,ds
-
\int_0^t
\left\langle
\operatorname{div}
(\widetilde u\otimes\widetilde u)(s),
\phi
\right\rangle_{H^{-1},H^1}\,ds
\nonumber\\
&=
\int_0^t
\left\langle
\operatorname{div}
(\widetilde u^n\otimes\widetilde u^n)(s)
-
\operatorname{div}
(\widetilde u\otimes\widetilde u)(s),
P_n\phi
\right\rangle_{H^{-1},H^1}\,ds
\nonumber\\
&\quad
+
\int_0^t
\left\langle
\operatorname{div}
(\widetilde u\otimes\widetilde u)(s),
P_n\phi-\phi
\right\rangle_{H^{-1},H^1}\,ds,
\end{align*}
where we used \eqref{URE76} and $P_n\phi\to\phi$ strongly in $H^1$, which implies that
\begin{align}
\int_0^t
\left\langle
\operatorname{div}
(\widetilde u^n\otimes\widetilde u^n)(s),
P_n\phi
\right\rangle_{H^{-1},H^1}\,ds
\longrightarrow
\int_0^t
\left\langle
\operatorname{div}
(\widetilde u\otimes\widetilde u)(s),
\phi
\right\rangle_{H^{-1},H^1}\,ds.
\label{UUrbt}
\end{align}
Exactly the same argument, using
\eqref{BBrt3}, and $P_n\phi\to\phi$ strongly in $H^1$, gives
\begin{align}
\int_0^t
\left\langle
\operatorname{div}
(\widetilde B^n\otimes\widetilde B^n)(s),
P_n\phi
\right\rangle_{H^{-1},H^1}\,ds
\longrightarrow
\int_0^t
\left\langle
\operatorname{div}
(\widetilde B\otimes\widetilde B)(s),
\phi
\right\rangle_{H^{-1},H^1}\,ds.
\label{UGeBB}
\end{align}
Combining
\eqref{GHTdoF},
\eqref{UUrbt}, and
\eqref{UGeBB}, we obtain
\begin{align}
\int_0^t
\left\langle
F_1(\widetilde u^n(s),\widetilde B^n(s)),
P_n\phi
\right\rangle_{H^{-1},H^1}\,ds
\longrightarrow
\int_0^t
\left\langle
F_1(\widetilde u(s),\widetilde B(s)),
\phi
\right\rangle_{H^{-1},H^1}\,ds.
\label{FR1Con}
\end{align}
For the 1st term on the right hand side of \eqref{P2K}, we have
\begin{align}
&
\int_0^t
\left\langle
\Delta\widetilde B^n(s),
P_n\psi
\right\rangle_{H^{-1},H^1}\,ds
-
\int_0^t
\left\langle
\Delta\widetilde B(s),
\psi
\right\rangle_{H^{-1},H^1}\,ds
\nonumber\\
&=
\int_0^t
\left\langle
\Delta\widetilde B^n(s)
-
\Delta\widetilde B(s),
\psi
\right\rangle_{H^{-1},H^1}\,ds
+
\int_0^t
\left\langle
\Delta\widetilde B^n(s),
P_n\psi-\psi
\right\rangle_{H^{-1},H^1}\,ds.
\label{NBTUl},
\end{align}
where the first term converges to zero by
$\Delta \tilde{B}^n\rightharpoonup \Delta \tilde{B}^n\, \textrm{weakly in}\, L^2(0,T;H^{-1})$ and the second term converges to zero by $P_n\psi\to\psi$, 
therefore, we have,
\begin{align}
\int_0^t
\left\langle
\Delta\widetilde B^n(s),
P_n\psi
\right\rangle_{H^{-1},H^1}\,ds
\longrightarrow
\int_0^t
\left\langle
\Delta\widetilde B(s),
\psi
\right\rangle_{H^{-1},H^1}\,ds.
\label{D1W12s}
\end{align}
Similarly, since $P_n \psi \to \psi\, \textrm{strongly in}\,  H^1$, using
\eqref{uy56}, and 
\eqref{BUR3w},
 we obtain
\begin{align}
\int_0^t
\left\langle
\operatorname{div}
(\widetilde u^n\otimes\widetilde B^n)(s),
P_n\psi
\right\rangle_{H^{-1},H^1}\,ds
&\longrightarrow
\int_0^t
\left\langle
\operatorname{div}
(\widetilde u\otimes\widetilde B)(s),
\psi
\right\rangle_{H^{-1},H^1}\,ds,
\label{we23a}
\\
\int_0^t
\left\langle
\operatorname{div}
(\widetilde B^n\otimes\widetilde u^n)(s),
P_n\psi
\right\rangle_{H^{-1},H^1}\,ds
&\longrightarrow
\int_0^t
\left\langle
\operatorname{div}
(\widetilde B\otimes\widetilde u)(s),
\psi
\right\rangle_{H^{-1},H^1}\,ds.
\label{gtry}
\end{align}
Combining
\eqref{D1W12s},
\eqref{we23a}, and
\eqref{gtry}, we conclude that
\begin{align}
\int_0^t
\left\langle
F_2(\widetilde u^n(s),\widetilde B^n(s)),
P_n\psi
\right\rangle_{H^{-1},H^1}\,ds
\longrightarrow
\int_0^t
\left\langle
F_2(\widetilde u(s),\widetilde B(s)),
\psi
\right\rangle_{H^{-1},H^1}\,ds.
\label{F2CO}
\end{align}
Following the uniform $L^2$- and $L^4$-estimates, the tightness of
$\{\mu_n\}_{n\geq1}$ on $\mathcal Z$, and the
Skorokhod-Jakubowski representation, we now prove
Theorem~\ref{thm10}.

\medskip

\subsection*{Proof of Theorem~\ref{thm10} (Part I: Existence of a global martingale solution)}
\paragraph{\textbf{Verification of (M1).}} Define
\begin{align}
\widehat{\mathbb P}
:=
\mathcal L_{\widetilde{\mathbb P}}(\widetilde u,\widetilde B)
\in\mathcal P(\bar\Omega),
\qquad
\bar\Omega=\mathcal Z.
\label{LLaW}
\end{align}
 By \eqref{GIL} and the equality of laws $\mathcal L_{\widetilde{\mathbb P}}
(\widetilde u^n,\widetilde B^n)
=
\mathcal L_{\mathbb P}
(u^n,B^n)$,
$$
\widetilde u^n(0)=P_nu_0,
\qquad
\widetilde B^n(0)=P_nB_0,
\qquad
\widetilde{\mathbb P}\text{-a.s.}
$$
Since $P_nu_0\to u_0$ and $P_nB_0\to B_0$ strongly in $H$, hence in
$H^{-1}$, it follows from \eqref{L1tildn} and \eqref{L2bsH-1} that
$$
\widetilde u^n(0)\to\widetilde u(0),
\qquad
\widetilde B^n(0)\to\widetilde B(0)
\quad\text{in }H^{-1},
\qquad
\widetilde{\mathbb P}\text{-a.s.},
$$
uniqueness of limits in $H^{-1}$ yields
\begin{align}
\widetilde u(0)=u_0,
\qquad
\widetilde B(0)=B_0,
\qquad
\widetilde{\mathbb P}\text{-a.s.}
\label{hj21x}
\end{align}
Hence, by \eqref{LLaW}, \eqref{hj21x}, we obtain
\begin{align}
\widehat{\mathbb P}
\left(
(u(0),B(0))=(u_0,B_0)
\right)
=1.
\label{M1-D1}
\end{align}
By \eqref{L1C2}, \eqref{K1M3}, and \eqref{LLaW},
\begin{align}
\widehat{\mathbb P}
\left(
\sup_{0\le t\le T}
\lVert u(t)\rVert_{H^{0,1}}<\infty,\;
\int_0^T
\lVert u(t)\rVert_{H^{1,1}}^2\,dt<\infty
\right)
&=1,
\label{M1-R1}
\\
\widehat{\mathbb P}
\left(
\sup_{0\le t\le T}
\lVert B(t)\rVert_{H^{0,1}}<\infty,\;
\int_0^T
\left(
\lVert B(t)\rVert_{H^{0,2}}^2
+
\lVert B(t)\rVert_{H^{1,1}}^2
\right)\,dt<\infty
\right)
&=1.
\label{Bg12r}
\end{align}
Moreover, by \eqref{L2h11}, \eqref{L4Bh11},
\eqref{L3H01}, \eqref{L5H01}, \eqref{L7H1}, and \eqref{L8bH02},
\begin{align}
\widetilde{\mathbb E}
\int_0^T
\left(
\lVert\widetilde u(t)\rVert_{H^{1,1}}^2
+
\lVert\widetilde B(t)\rVert_{H^{1,1}}^2
\right)\,dt
&\le C_T,
\label{674F3}
\\
\widetilde{\mathbb E}
\left[
\lVert\widetilde u\rVert_{L^\infty(0,T;H^{0,1})}^2
+
\lVert\widetilde B\rVert_{L^\infty(0,T;H^{0,1})}^2
\right]
&\le C_T,
\label{t532}
\\
\widetilde{\mathbb E}
\int_0^T
\left(
\lVert\widetilde B(t)\rVert_{H^1}^2
+
\lVert\widetilde B(t)\rVert_{H^{0,2}}^2
\right)\,dt
&\le C_T.
\label{y6de}
\end{align}
By \eqref{LLaW}, the same estimates \eqref{674F3}-\eqref{y6de} hold under
$\widehat{\mathbb P}$ for $(u,B)$. In particular, \eqref{674F3} yields
\begin{align}
\widehat{\mathbb E}
\int_0^T
\left(
\lVert u(t)\rVert_{H^{1,1}}^2
+
\lVert B(t)\rVert_{H^{1,1}}^2
\right)\,dt
&=
\widetilde{\mathbb E}
\int_0^T
\left(
\lVert \widetilde u(t)\rVert_{H^{1,1}}^2
+
\lVert \widetilde B(t)\rVert_{H^{1,1}}^2
\right)\,dt
\nonumber\\
&\leq
C_T,
\label{kMMwrk}
\end{align}
while \eqref{t532} gives
\begin{align}
\widehat{\mathbb E}
\left[
\lVert u\rVert_{L^\infty(0,T;H^{0,1})}^2
+
\lVert B\rVert_{L^\infty(0,T;H^{0,1})}^2
\right]
&=
\widetilde{\mathbb E}
\left[
\lVert \widetilde u\rVert_{L^\infty(0,T;H^{0,1})}^2
+
\lVert \widetilde B\rVert_{L^\infty(0,T;H^{0,1})}^2
\right]
\nonumber\\
&\leq
C_T.
\end{align}
Finally, by \eqref{y6de}, we obtain
$$
\widehat{\mathbb E}
\int_0^T
\left(
\lVert B(t)\rVert_{H^1}^2
+
\lVert B(t)\rVert_{H^{0,2}}^2
\right)\,dt
=
\widetilde{\mathbb E}
\int_0^T
\left(
\lVert \widetilde B(t)\rVert_{H^1}^2
+
\lVert \widetilde B(t)\rVert_{H^{0,2}}^2
\right)\,dt
\leq
C_T.
$$
Moreover, by 
\eqref{L2h11},
\eqref{KLNG6}, and
\eqref{km56C}, we have
\begin{align}
&
\int_0^T
\left\lVert 
F_1\bigl(\widetilde u(s),\widetilde B(s)\bigr)
\right\rVert_{H^{-1}}\,ds
\nonumber\\
&\leq
\nu
\int_0^T
\left\lVert 
\partial_1^2\widetilde u(s)
\right\rVert_{H^{-1}}\,ds
+
\int_0^T
\left\lVert 
\operatorname{div}
\bigl(\widetilde u(s)\otimes\widetilde u(s)\bigr)
\right\rVert_{H^{-1}}\,ds
+
\int_0^T
\left\lVert 
\operatorname{div}
\bigl(\widetilde B(s)\otimes\widetilde B(s)\bigr)
\right\rVert_{H^{-1}}\,ds
\nonumber\\
&\leq
\nu T^{1/2}
\lVert \widetilde u\rVert_{L^2(0,T;H^1)}
+
\int_0^T
\lVert \widetilde u(s)\rVert_{L^4}^2\,ds
+
\int_0^T
\lVert \widetilde B(s)\rVert_{L^4}^2\,ds
\nonumber\\
&=
\nu T^{1/2}
\lVert \widetilde u\rVert_{L^2(0,T;H^1)}
+
\lVert \widetilde u\rVert_{L^2(0,T;L^4)}^2
+
\lVert \widetilde B\rVert_{L^2(0,T;L^4)}^2
<\infty,
\qquad
\widetilde{\mathbb P}\text{-a.s.}
\label{F1IC}
\end{align}
Moreover, by
\eqref{K1M3},
\eqref{KLNG6}, and
\eqref{km56C}, we have
\begin{align}
\int_0^T
\left\lVert 
F_2\bigl(\widetilde u(s),\widetilde B(s)\bigr)
\right\rVert_{H^{-1}}\,ds
&\leq
\eta
\int_0^T
\left\lVert 
\Delta\widetilde B(s)
\right\rVert_{H^{-1}}\,ds
+
\int_0^T
\left\lVert 
\operatorname{div}
\bigl(\widetilde u(s)\otimes\widetilde B(s)\bigr)
\right\rVert_{H^{-1}}\,ds
\nonumber\\
&\quad+
\int_0^T
\left\lVert 
\operatorname{div}
\bigl(\widetilde B(s)\otimes\widetilde u(s)\bigr)
\right\rVert_{H^{-1}}\,ds
\nonumber\\
&\leq
\eta T^{1/2}
\lVert \widetilde B\rVert_{L^2(0,T;H^1)}
+
\int_0^T
\lVert \widetilde u(s)\rVert_{L^4}
\lVert \widetilde B(s)\rVert_{L^4}\,ds
\nonumber\\
&\quad+
\int_0^T
\lVert \widetilde B(s)\rVert_{L^4}
\lVert \widetilde u(s)\rVert_{L^4}\,ds
\nonumber\\
&\leq
\eta T^{1/2}
\lVert \widetilde B\rVert_{L^2(0,T;H^1)}
+
2
\lVert \widetilde u\rVert_{L^2(0,T;L^4)}
\lVert \widetilde B\rVert_{L^2(0,T;L^4)}
<\infty,
\nonumber\\[-1mm]
&\hspace{45mm}
\widetilde{\mathbb P}\text{-a.s.}
\label{F2IC}
\end{align}
Therefore, from \eqref{F1IC} and  \eqref{F2IC}, we obtain
\begin{align}
F_1(\widetilde u,\widetilde B),\quad
F_2(\widetilde u,\widetilde B)
\in L^1(0,T;H^{-1}),
\qquad
\widetilde{\mathbb P}\text{-a.s.}
\label{w34B}
\end{align}
and consequently, by \eqref{LLaW}, we obtain
\begin{align}
\widehat{\mathbb P}
\left(
\int_0^T
\left(
\lVert F_1(u,B)\rVert_{H^{-1}}
+
\lVert F_2(u,B)\rVert_{H^{-1}}
\right)\,dt<\infty
\right)
=1.
\label{M1-I}
\end{align}
It remains to verify the integrability of the $\sigma_1$ and $\sigma_2$.
Using \eqref{nL2} together with
\eqref{L3H01}, \eqref{L2h11}, \eqref{L5H01}, and \eqref{L7H1},
we obtain
\begin{align}
&
\widetilde{\mathbb E}
\int_0^T
\Bigl(
\lVert
\sigma_1(t,\widetilde u(t),\widetilde B(t))
\rVert_{L_2(\mathfrak U_1;H)}^2
+
\lVert
\sigma_2(t,\widetilde u(t),\widetilde B(t))
\rVert_{L_2(\mathfrak U_2;H)}^2
\Bigr)\,dt
\nonumber\\
&\le
K_0T
+
K_1\widetilde{\mathbb E}
\int_0^T
\left(
\lVert\widetilde u(t)\rVert_{L^2}^2
+
\lVert\widetilde B(t)\rVert_{L^2}^2
\right)\,dt
\nonumber\\
&\quad+
K_2\widetilde{\mathbb E}
\int_0^T
\left(
\nu\lVert\partial_1\widetilde u(t)\rVert_{L^2}^2
+
\eta\lVert\nabla\widetilde B(t)\rVert_{L^2}^2
\right)\,dt
\le C_T.
\label{LIM1}
\end{align}
By \eqref{yuh5}, from \eqref{LIM1}, we obtain
\begin{align}
&\widehat{\mathbb E}\int_0^T
\left(
\lVert \sigma_1(t,u(t),B(t))\rVert_{L_2(\mathfrak U_1;H)}^2
+
\lVert \sigma_2(t,u(t),B(t))\rVert_{L_2(\mathfrak U_2;H)}^2
\right)\,dt
\nonumber\\
&\qquad=
\widetilde{\mathbb E}\int_0^T
\left(
\lVert \sigma_1(t,\widetilde u(t),\widetilde B(t))\rVert_{L_2(\mathfrak U_1;H)}^2
+
\lVert \sigma_2(t,\widetilde u(t),\widetilde B(t))\rVert_{L_2(\mathfrak U_2;H)}^2
\right)\,dt
\le C_T.
\label{HGRT1}
\end{align}
Since the integrand is nonnegative,
\eqref{LIM1} implies its time integral is finite
$\widetilde{\mathbb P}$-a.s. Hence, by \eqref{LLaW},
\begin{align}
\widehat{\mathbb P}
\Bigg(
\int_0^T
\Bigl(
\lVert\sigma_1(t,u(t),B(t))
\rVert_{L_2(\mathfrak U_1;H)}^2
+
\lVert\sigma_2(t,u(t),B(t))
\rVert_{L_2(\mathfrak U_2;H)}^2
\Bigr)\,dt<\infty
\Bigg)
=1.
\label{M1-N1}
\end{align}
Combining \eqref{M1-D1}, \eqref{M1-R1},
\eqref{Bg12r}, \eqref{M1-I}, and \eqref{M1-N1},
we conclude that $\widehat{\mathbb P}$ satisfies
Definition~\ref{DEF}(M1).
\medskip

\noindent
We next verify Definition~\ref{DEF}{\rm (M2)}. We decompose the
verification into four parts:
\begin{itemize}
\item[{\rm (1)}] We prove that
$M_{u,n}^{\phi}(t)\to M_u^\phi(t)$ and
$M_{B,n}^{\psi}(t)\to M_B^\psi(t)$.

\item[{\rm (2)}] We prove that
$M_u^\phi$ and $M_B^\psi$ are continuous martingales.

\item[{\rm (3)}] We prove that
$M_u^\phi(t),M_B^\psi(t)\in L^2(\widetilde\Omega)$.

\item[{\rm (4)}] We prove that
$$
\left\langle\!\left\langle M_u^\phi\right\rangle\!\right\rangle_t
=
\int_0^t
\left\lVert
\sigma_1^*(s,\widetilde u,\widetilde B)\phi
\right\rVert_{\mathfrak U_1}^2\,ds,
\qquad
\left\langle\!\left\langle M_B^\psi\right\rangle\!\right\rangle_t
=
\int_0^t
\left\lVert
\sigma_2^*(s,\widetilde u,\widetilde B)\psi
\right\rVert_{\mathfrak U_2}^2\,ds,
$$
and
$$
\left\langle\!\left\langle
M_u^\phi,M_B^\psi
\right\rangle\!\right\rangle_t=0.
$$
\end{itemize}
\noindent
\textbf{Part 1. Convergence of $M_{u,n}^{\phi}(t)$ and
$M_{B,n}^{\psi}(t)$.}

\noindent  Recalling
\eqref{def-F1}-\eqref{def-F2}, define
\begin{align}
M_{u,n}^{\phi}(t)
:={}&
\bigl(\widetilde u^n(t),\phi\bigr)_H
-
\bigl(P_nu_0,\phi\bigr)_H
-
\int_0^t
\left\langle
F_1\bigl(\widetilde u^n(s),\widetilde B^n(s)\bigr),
P_n\phi
\right\rangle_{H^{-1},H^1}\,ds,
\label{JKLT23}
\\
M_{B,n}^{\psi}(t)
:={}&
\bigl(\widetilde B^n(t),\psi\bigr)_H
-
\bigl(P_nB_0,\psi\bigr)_H
-
\int_0^t
\left\langle
F_2\bigl(\widetilde u^n(s),\widetilde B^n(s)\bigr),
P_n\psi
\right\rangle_{H^{-1},H^1}\,ds,
\label{123HKY}
\end{align}
correspondingly, we define as follows
\begin{align}
M_u^\phi(t)
:={}&
\bigl(\widetilde u(t),\phi\bigr)_H
-
\bigl(u_0,\phi\bigr)_H
-
\int_0^t
\left\langle
F_1\bigl(\widetilde u(s),\widetilde B(s)\bigr),
\phi
\right\rangle_{H^{-1},H^1}\,ds,
\label{HJ451}
\\
M_B^\psi(t)
:={}&
\bigl(\widetilde B(t),\psi\bigr)_H
-
\bigl(B_0,\psi\bigr)_H
-
\int_0^t
\left\langle
F_2\bigl(\widetilde u(s),\widetilde B(s)\bigr),
\psi
\right\rangle_{H^{-1},H^1}\,ds.
\label{B11R1}
\end{align}
We first consider the first two terms in \eqref{HJ451} and \eqref{B11R1}. By
\eqref{L1tildn} and \eqref{L2bsH-1}, 
\begin{align}
\bigl(\widetilde u^n(t),\phi\bigr)_H
&\longrightarrow
\bigl(\widetilde u(t),\phi\bigr)_H,
\qquad
\widetilde{\mathbb P}\text{-a.s.},
\label{haz3q}
\\
\bigl(\widetilde B^n(t),\psi\bigr)_H
&\longrightarrow
\bigl(\widetilde B(t),\psi\bigr)_H,
\qquad
\widetilde{\mathbb P}\text{-a.s.}
\label{Gcrz2}
\end{align}
Moreover,  \eqref{L^4-mo} gives
\begin{align}
\sup_{n\geq1}
\widetilde{\mathbb E}
\left\lvert
\bigl(\widetilde u^n(t),\phi\bigr)_H
\right\rvert^4
&\leq
C_T\lVert\phi\rVert_H^4,
\label{yhx87}
\\
\sup_{n\geq1}
\widetilde{\mathbb E}
\left\lvert
\bigl(\widetilde B^n(t),\psi\bigr)_H
\right\rvert^4
&\leq
C_T\lVert\psi\rVert_H^4,
\label{uw12v}
\end{align}
where
\eqref{yhx87}-\eqref{uw12v}
give uniform integrability, and combining this with
\eqref{haz3q}-\eqref{Gcrz2},
Vitali's convergence theorem yields
\begin{align}
\widetilde{\mathbb E}
\left\lvert
\bigl(\widetilde u^n(t),\phi\bigr)_H
-
\bigl(\widetilde u(t),\phi\bigr)_H
\right\rvert
&\longrightarrow0,
\label{A1R1U}
\\
\widetilde{\mathbb E}
\left\lvert
\bigl(\widetilde B^n(t),\psi\bigr)_H
-
\bigl(\widetilde B(t),\psi\bigr)_H
\right\rvert
&\longrightarrow0.
\label{s3ezq}
\end{align}
For the second two terms in \eqref{HJ451} and \eqref{B11R1}, since
$P_nu_0\to u_0$ and $P_nB_0\to B_0$ strongly in $H$, we have
\begin{align}
\left\lvert
\bigl(P_nu_0-u_0,\phi\bigr)_H
\right\rvert
&\leq
\lVert P_nu_0-u_0\rVert_H\lVert\phi\rVert_H
\longrightarrow0,
\label{df12zc}
\\
\left\lvert
\bigl(P_nB_0-B_0,\psi\bigr)_H
\right\rvert
&\leq
\lVert P_nB_0-B_0\rVert_H\lVert\psi\rVert_H
\longrightarrow0.
\label{ub45z.}
\end{align}
It remains to upgrade the almost sure convergence of the $F_1$- and
$F_2$-integrals already established in \eqref{FR1Con} and \eqref{F2CO}
to convergence in $L^1(\widetilde\Omega)$. To this end, we establish
uniform integrability. Since
$\phi,\psi\in C^\infty(\mathbb T^2;\mathbb R^2)$,
\begin{align}
\sup_{n\geq1}
\left(
\lVert\partial_1P_n\phi\rVert_H
+
\lVert\nabla P_n\phi\rVert_{L^\infty}
\right)
&\leq C_\phi,
\label{Jk100}
\\
\sup_{n\geq1}
\left(
\lVert\nabla P_n\psi\rVert_{L^2}
+
\lVert\nabla P_n\psi\rVert_{L^\infty}
\right)
&\leq C_\psi.
\label{Jk200}
\end{align}
For the third two terms in \eqref{HJ451} and \eqref{B11R1}, using \eqref{def-F1}-\eqref{def-F2}, integration by parts,
Hölder's inequality, and
\eqref{Jk100}-\eqref{Jk200},
we obtain
\begin{align}
&
\int_0^T
\left\lvert
\left\langle
F_1(\widetilde u^n,\widetilde B^n),
P_n\phi
\right\rangle_{H^{-1},H^1}
\right\rvert\,ds
\nonumber\\
&\leq
C_\phi
\left[
T^{1/2}
\left(
\int_0^T
\lVert\partial_1\widetilde u^n\rVert_H^2\,ds
\right)^{1/2}
+
T\sup_{0\leq s\leq T}
\lVert\widetilde u^n(s)\rVert_H^2
+
T\sup_{0\leq s\leq T}
\lVert\widetilde B^n(s)\rVert_H^2
\right],
\label{K100}
\\
\textrm{and},
\nonumber
\\
&
\int_0^T
\left\lvert
\left\langle
F_2(\widetilde u^n,\widetilde B^n),
P_n\psi
\right\rangle_{H^{-1},H^1}
\right\rvert\,ds
\nonumber\\
&\leq
C_\psi
\left[
T^{1/2}
\left(
\int_0^T
\lVert\nabla\widetilde B^n\rVert_{L^2}^2\,ds
\right)^{1/2}
+
2T
\sup_{0\leq s\leq T}
\lVert\widetilde u^n(s)\rVert_H
\sup_{0\leq s\leq T}
\lVert\widetilde B^n(s)\rVert_H
\right].
\label{JJK300}
\end{align}
Squaring \eqref{K100}-\eqref{JJK300},
taking expectations, and using the uniform energy estimate
\eqref{L2-sup} together with the fourth-moment estimate
\eqref{L^4-mo}, we obtain
\begin{align}
\sup_{n\geq1}
\widetilde{\mathbb E}
\left[
\left(
\int_0^T
\left\lvert
\left\langle
F_1(\widetilde u^n,\widetilde B^n),
P_n\phi
\right\rangle_{H^{-1},H^1}
\right\rvert\,ds
\right)^2
\right]
&<\infty,
\label{h6f48z}
\\
\sup_{n\geq1}
\widetilde{\mathbb E}
\left[
\left(
\int_0^T
\left\lvert
\left\langle
F_2(\widetilde u^n,\widetilde B^n),
P_n\psi
\right\rangle_{H^{-1},H^1}
\right\rvert\,ds
\right)^2
\right]
&<\infty.
\label{D2walz}
\end{align}
Hence
\eqref{h6f48z}-\eqref{D2walz}
imply uniform integrability of the two families of drift integrals.
Combining this with the almost sure convergences
\eqref{FR1Con} and \eqref{F2CO}, Vitali's convergence theorem yields
\begin{align}
&
\widetilde{\mathbb E}
\left\lvert
\int_0^t
\left\langle
F_1(\widetilde u^n,\widetilde B^n),P_n\phi
\right\rangle_{H^{-1},H^1}\,ds
-
\int_0^t
\left\langle
F_1(\widetilde u,\widetilde B),\phi
\right\rangle_{H^{-1},H^1}\,ds
\right\rvert
\longrightarrow0,
\label{yt3xz}
\\
&
\widetilde{\mathbb E}
\left\lvert
\int_0^t
\left\langle
F_2(\widetilde u^n,\widetilde B^n),P_n\psi
\right\rangle_{H^{-1},H^1}\,ds
-
\int_0^t
\left\langle
F_2(\widetilde u,\widetilde B),\psi
\right\rangle_{H^{-1},H^1}\,ds
\right\rvert
\longrightarrow0.
\label{1LD1}
\end{align}
Finally, combining
\eqref{A1R1U}, \eqref{df12zc}, and
\eqref{yt3xz} with
\eqref{JKLT23}-\eqref{HJ451}, and combining
\eqref{s3ezq}, \eqref{ub45z.}, and
\eqref{1LD1} with
\eqref{123HKY}-\eqref{B11R1}, we obtain, for every
$t\in[0,T]$,
\begin{align}
M_{u,n}^{\phi}(t)
&\longrightarrow
M_u^\phi(t)
&&\text{in }L^1(\widetilde\Omega),
\label{qaz}
\\
M_{B,n}^{\psi}(t)
&\longrightarrow
M_B^\psi(t)
&&\text{in }L^1(\widetilde\Omega).
\label{HJTLU}
\end{align}
\medskip
\noindent
\textbf{Part 2. $M_u^\phi$ and $M_B^\psi$ are continuous martingales.}
Let
$g:\mathcal X_u\times\mathcal X_B\to\mathbb R$
be bounded and continuous. By
\eqref{L1tildn}-\eqref{y6Fr}, it follows that
$
(\widetilde u^n,\widetilde B^n)
\longrightarrow
(\widetilde u,\widetilde B)
\;
\widetilde{\mathbb P}\text{-a.s.}
$
in $\mathcal X_u\times\mathcal X_B$; hence, by the continuity and
boundedness of $g$,
\begin{align}
g(\widetilde u^n,\widetilde B^n)
&\longrightarrow
g(\widetilde u,\widetilde B)
\qquad
\widetilde{\mathbb P}\text{-a.s.},
\label{M2C}
\\
\lvert
g(\widetilde u^n,\widetilde B^n)
-
g(\widetilde u,\widetilde B)
\rvert
&\leq
2\lVert g\rVert_\infty.
\label{M2Lo}
\end{align}
Applying \eqref{qaz} at the two times $r$ and $t$, we obtain
\begin{align}
&
\widetilde{\mathbb E}
\left\lvert
\left(
M_{u,n}^{\phi}(t)-M_{u,n}^{\phi}(r)
\right)
-
\left(
M_u^\phi(t)-M_u^\phi(r)
\right)
\right\rvert
\nonumber\\
&\leq
\widetilde{\mathbb E}
\left\lvert
M_{u,n}^{\phi}(t)-M_u^\phi(t)
\right\rvert
+
\widetilde{\mathbb E}
\left\lvert
M_{u,n}^{\phi}(r)-M_u^\phi(r)
\right\rvert
\longrightarrow0.
\label{KL2}
\end{align}
Moreover, since \eqref{qaz} gives convergence in
$L^1(\widetilde\Omega)$, we have
\begin{align}
M_u^\phi(r),\,M_u^\phi(t)
\in L^1(\widetilde\Omega).
\label{6df2}
\end{align}
Using \eqref{KL2}, we write
\begin{align}
&
\left(
M_{u,n}^{\phi}(t)-M_{u,n}^{\phi}(r)
\right)
g(\widetilde u^n,\widetilde B^n)
-
\left(
M_u^\phi(t)-M_u^\phi(r)
\right)
g(\widetilde u,\widetilde B)
\nonumber\\
&=
\left[
\left(
M_{u,n}^{\phi}(t)-M_{u,n}^{\phi}(r)
\right)
-
\left(
M_u^\phi(t)-M_u^\phi(r)
\right)
\right]
g(\widetilde u^n,\widetilde B^n)
\nonumber\\
&\quad+
\left(
M_u^\phi(t)-M_u^\phi(r)
\right)
\left[
g(\widetilde u^n,\widetilde B^n)
-
g(\widetilde u,\widetilde B)
\right].
\label{hj298}
\end{align}
By \eqref{KL2} and the boundedness of $g$, the first term on the
right-hand side of \eqref{hj298} converges to zero in
$L^1(\widetilde\Omega)$; by \eqref{M2C}, \eqref{M2Lo}, and
\eqref{6df2}, the second term converges to zero in
$L^1(\widetilde\Omega)$ by the dominated convergence theorem. Therefore, \eqref{hj298} implies that
\begin{align}
\widetilde{\mathbb E}
\left[
\left(
M_{u,n}^{\phi}(t)-M_{u,n}^{\phi}(r)
\right)
g(\widetilde u^n,\widetilde B^n)
\right]
\longrightarrow
\widetilde{\mathbb E}
\left[
\left(
M_u^\phi(t)-M_u^\phi(r)
\right)
g(\widetilde u,\widetilde B)
\right].
\label{K12f4}
\end{align}
Now, from \eqref{HGS} we have
\begin{align}
M_{u,n}^{\phi}(t)-M_{u,n}^{\phi}(r)
=
\int_r^t
\left(
P_n\sigma_1(s,u^n(s),B^n(s))
\Pi_n^{(1)}\,dW_1(s),
P_n\phi
\right)_H.
\label{T1L1R0}
\end{align}
The integrand in \eqref{T1L1R0} is predictable,
and  together
with \eqref{nL2}, give
\begin{align}
&
\left\lVert
\left(
P_n\sigma_1(s,u^n,B^n)\Pi_n^{(1)},P_n\phi
\right)_H
\right\rVert_{\mathfrak U_1}^{2}
\nonumber\\
&\leq
\lVert\phi\rVert_H^2
\lVert\sigma_1(s,u^n,B^n)\rVert_{L_2(\mathfrak U_1;H)}^2
\nonumber\\
&\leq
\lVert\phi\rVert_H^2
\left[
K_0
+
K_1
\left(
\lVert u^n\rVert_H^2+\lVert B^n\rVert_H^2
\right)
+
K_2
\left(
\nu\lVert\partial_1u^n\rVert_H^2
+
\eta\lVert\nabla B^n\rVert_H^2
\right)
\right].
\label{AS2b6}
\end{align}
Integrating \eqref{AS2b6} over $[0,T]$, taking expectation, and applying
\eqref{L2engy} yield
\begin{align}
\mathbb E\int_0^T
\left\lVert
\left(P_n\sigma_1(s,u^n(s),B^n(s))\Pi_n^{(1)},P_n\phi\right)_H
\right\rVert_{\mathfrak U_1}^{2}\,ds
<\infty,
\label{ghf11a}
\end{align}
so that \eqref{T1L1R0} and
\eqref{ghf11a} give
\begin{align}
\mathbb E\!\left[
M_{u,n}^{\phi}(t)-M_{u,n}^{\phi}(r)
\,\middle\rvert\,\mathcal F_r
\right]
=0,
\qquad 0\le r<t\le T.
\label{s1z2gk}
\end{align} Using
\eqref{lka0}
and
 and the tower property,it follows from \eqref{s1z2gk} that
\begin{align}
\mathbb E
\left[
M_{u,n}^{\phi}(t)-M_{u,n}^{\phi}(r)
\,\middle\rvert\,
\mathcal F_r^n
\right]
&=
\mathbb E
\left[
\mathbb E
\left[
M_{u,n}^{\phi}(t)-M_{u,n}^{\phi}(r)
\,\middle\rvert\,
\mathcal F_r
\right]
\,\middle\rvert\,
\mathcal F_r^n
\right]
=0.
\label{Cx3zk4}
\end{align}
Since $g(u^n,B^n)$ is $\mathcal F_r^n$-measurable, hence
\eqref{Cx3zk4} gives
\begin{align}
\mathbb E
\left[
\left(
M_{u,n}^{\phi}(t)-M_{u,n}^{\phi}(r)
\right)
g(u^n,B^n)
\right]
=
\mathbb E
\left[
g(u^n,B^n)
\mathbb E
\left[
M_{u,n}^{\phi}(t)-M_{u,n}^{\phi}(r)
\,\middle\rvert\,
\mathcal F_r^n
\right]
\right]
=0.
\label{L1zf4c"}
\end{align}
The equality of law, 
$
\mathcal L_{\widetilde{\mathbb P}}
(\widetilde u^n,\widetilde B^n)
=
\mathcal L_{\mathbb P}(u^n,B^n),
$ transfers
\eqref{L1zf4c"} to
\begin{align}
\widetilde{\mathbb E}
\left[
\left(
M_{u,n}^{\phi}(t)-M_{u,n}^{\phi}(r)
\right)
g(\widetilde u^n,\widetilde B^n)
\right]
=0.
\label{ag12zs}
\end{align}
Combining \eqref{ag12zs} with the
convergence \eqref{K12f4} gives
\begin{align}
\widetilde{\mathbb E}
\left[
\left(
M_u^\phi(t)-M_u^\phi(r)
\right)
g(\widetilde u,\widetilde B)
\right]
=0.
\label{MCV}
\end{align}
Following the filtration defined in \eqref{MF1}, on the Skorokhod probability space, we set
\begin{align}
\widetilde{\mathcal F}_r
:=
\sigma\left(
\widetilde u(s),\widetilde B(s):0\leq s\leq r
\right),
\qquad
0\leq r\leq T.
\label{65xd}
\end{align}
By the definition of $M_u^\phi$ from \eqref{HJ451} and the definition of the
filtration in \eqref{65xd},
\begin{align}
M_u^\phi(r)
\quad\text{is }
\widetilde{\mathcal F}_r\text{-measurable}.
\label{KLCx}
\end{align}
The monotone-class argument used in
\cite[Theorem~4.5, Step~3, pp.~1737-1738]{goldys2009martingale}
extends \eqref{MCV} to
\begin{align}
\widetilde{\mathbb E}
\left[
\left(
M_u^\phi(t)-M_u^\phi(r)
\right)G
\right]
=0
\qquad
\text{for every }
G\in L^\infty(\widetilde{\mathcal F}_r).
\label{KHJ54c}
\end{align}
The integrability in \eqref{6df2}, the
$\widetilde{\mathcal F}_r$-measurability in \eqref{KLCx}, and the
identity \eqref{KHJ54c} give
\begin{align}
\widetilde{\mathbb E}
\left[
M_u^\phi(t)
\,\middle\rvert\,
\widetilde{\mathcal F}_r
\right]
=
M_u^\phi(r),
\qquad
0\leq r<t\leq T.
\label{115gha}
\end{align}
Consequently, \eqref{6df2}, \eqref{KLCx}, and
\eqref{115gha} yield
\begin{align}
M_u^\phi
\quad\text{is an }
\{\widetilde{\mathcal F}_t\}_{t\in[0,T]}
\text{-martingale}.
\label{GHTL2}
\end{align}
Applying the same mechanism from \eqref{M2C}-\eqref{GHTL2},and  by the definition of $M_B^\psi$ from \eqref{B11R1}, we obtain
\eqref{65xd},
\begin{align}
M_B^\psi(r)
\quad\text{is }
\widetilde{\mathcal F}_r\text{-measurable}.
\label{0oL3}
\end{align}
Applying the same monotone-class argument used in
\eqref{MCV}-\eqref{KHJ54c}, 
by \eqref{qaz}-\eqref{HJTLU}, $M_B^\psi(r),M_B^\psi(t)\in
L^1(\widetilde\Omega)$; together with the
$\widetilde{\mathcal F}_r$-measurability in \eqref{0oL3},  we obtain
\begin{align}
M_B^\psi
\quad\text{is an }
\{\widetilde{\mathcal F}_t\}_{t\in[0,T]}
\text{-martingale}.
\label{lr2q}
\end{align}
Since $T>0$ was arbitrary, \eqref{CM2L1} and \eqref{L1M1A1} show that
$M_u^\phi$ and $M_B^\psi$ have continuous paths on $\mathbb R_+$.

\medskip
\noindent
Finally, we prove the continuity of $M_u^\phi$ and $M_B^\psi$. By the
definition of $M_u^\phi$ from \eqref{HJ451}, together with \eqref{L1C2} and
\eqref{F1IC},
\begin{align}
\lvert M_u^\phi(t)-M_u^\phi(r)\rvert
&\leq
\lVert\widetilde u(t)-\widetilde u(r)\rVert_{H^{-1}}
\lVert\phi\rVert_{H^1}
\nonumber\\
&\quad+
\lVert\phi\rVert_{H^1}
\int_r^t
\lVert
F_1(\widetilde u(s),\widetilde B(s))
\rVert_{H^{-1}}\,ds
\longrightarrow0
\qquad\text{as }t\to r.
\label{c1xz23}
\end{align}
Thus \eqref{c1xz23} gives
\begin{align}
M_u^\phi
\in C([0,T];\mathbb R),
\qquad
\widetilde{\mathbb P}\text{-a.s.}
\label{CM2L1}
\end{align}
Likewise, the definition of $M_B^\psi$ from \eqref{B11R1}, together with
\eqref{K1M3} and \eqref{F2IC}, gives
\begin{align}
\lvert M_B^\psi(t)-M_B^\psi(r)\rvert
&\leq
\lVert\widetilde B(t)-\widetilde B(r)\rVert_{H^{-1}}
\lVert\psi\rVert_{H^1}
\nonumber\\
&\quad+
\lVert\psi\rVert_{H^1}
\int_r^t
\lVert
F_2(\widetilde u(s),\widetilde B(s))
\rVert_{H^{-1}}\,ds
\longrightarrow0
\qquad\text{as }t\to r.
\label{KJU1}
\end{align}
It follows from \eqref{KJU1} that
\begin{align}
M_B^\psi
\in C([0,T];\mathbb R),
\qquad
\widetilde{\mathbb P}\text{-a.s.}
\label{L1M1A1}
\end{align}
Since $T>0$ was arbitrary,
\eqref{CM2L1} and \eqref{L1M1A1} show that
$M_u^\phi$ and $M_B^\psi$ have continuous paths on $\mathbb R_+$.

\medskip
\noindent
\textbf{Part 3. $M_u^\phi(t),M_B^\psi(t)\in L^2(\widetilde\Omega)$.}
Applying \eqref{JKLT23}
together with \eqref{HGS} gives
\begin{align}
M_{u,n}^{\phi}(t)
&=
\int_0^t
\left(
P_n\sigma_1
\bigl(s,u^n(s),B^n(s)\bigr)
\Pi_n^{(1)}\,dW_1(s),
P_n\phi
\right)_H .
\label{klc}
\end{align}
Since $u^n$ and $B^n$ are continuous adapted processes, the integrand
in \eqref{klc} is predictable, while
\eqref{ghf11a} gives its square integrability. Since $P_n$ and
$\Pi_n^{(1)}$ are orthogonal projections, the adjoint relation
rewrites \eqref{klc} as
\begin{align}
M_{u,n}^{\phi}(t)
&=
\int_0^t
\left(
\Pi_n^{(1)}
\sigma_1^\ast
\bigl(s,u^n(s),B^n(s)\bigr)
P_n\phi,
dW_1(s)
\right)_{\mathfrak U_1}.
\label{lkc}
\end{align}
The quadratic variation associated with \eqref{lkc} is given
\begin{align}
\left\langle\!\left\langle M_{u,n}^{\phi}
\right\rangle\!\right\rangle_t
&=
\int_0^t
\left\lVert
\Pi_n^{(1)}
\sigma_1^\ast
\bigl(s,u^n(s),B^n(s)\bigr)
P_n\phi
\right\rVert_{\mathfrak U_1}^2
\,ds.
\label{kas}
\end{align}
Applying the Burkholder-Davis-Gundy inequality to
\eqref{lkc}-\eqref{kas} gives
\begin{align}
\mathbb E
\lvert M_{u,n}^{\phi}(t)\rvert^4
&\leq
C\,
\mathbb E
\left(
\int_0^t
\left\lVert
\Pi_n^{(1)}
\sigma_1^\ast
\bigl(s,u^n(s),B^n(s)\bigr)
P_n\phi
\right\rVert_{\mathfrak U_1}^2\,ds
\right)^2.
\label{21Ma2}
\end{align}
Using
$\lVert P_n\phi\rVert_{H^1}\leq\lVert\phi\rVert_{H^1}$, and $\lVert G^\ast v\rVert_{\mathfrak U_1}^2 \leq \lVert G\rVert_{L_2(\mathfrak U_1;H^{-1})}^2 \lVert v\rVert_{H^1}^2$ give
\begin{align}
&
\left\lVert
\Pi_n^{(1)}
\sigma_1^\ast
\bigl(s,u^n(s),B^n(s)\bigr)
P_n\phi
\right\rVert_{\mathfrak U_1}^2
\leq
\left\lVert
\sigma_1
\bigl(s,u^n(s),B^n(s)\bigr)
\right\rVert_{L_2(\mathfrak U_1;H^{-1})}^2
\lVert\phi\rVert_{H^1}^2.
\label{h840}
\end{align}
From  \eqref{NH1}, we obtain
\begin{align}
&
\left\lVert
\sigma_1
\bigl(s,u^n(s),B^n(s)\bigr)
\right\rVert_{L_2(\mathfrak U_1;H^{-1})}^2
\leq
K_0'
+
K_1'
\left(
\lVert u^n(s)\rVert_H^2
+
\lVert B^n(s)\rVert_H^2
\right).
\label{940h}
\end{align}
Substituting
\eqref{h840}-\eqref{940h}
into \eqref{21Ma2} yields
\begin{align}
\mathbb E
\lvert M_{u,n}^{\phi}(t)\rvert^4
&\leq
C\lVert\phi\rVert_{H^1}^4
\mathbb E
\left(
\int_0^t
\left[
K_0'
+
K_1'
\left(
\lVert u^n(s)\rVert_H^2
+
\lVert B^n(s)\rVert_H^2
\right)
\right]ds
\right)^2
\nonumber\\
&\leq
C_{T,\phi}
\mathbb E
\left[
1+
\sup_{0\leq s\leq T}\lVert u^n(s)\rVert_H^4
+
\sup_{0\leq s\leq T}\lVert B^n(s)\rVert_H^4
\right].
\label{N1O2za}
\end{align}
Combining \eqref{N1O2za} with the fourth-moment
estimate \eqref{L^4-mo} gives
\begin{align}
\sup_{n\geq1}
\mathbb E
\lvert M_{u,n}^{\phi}(t)\rvert^4
<\infty.
\label{h640}
\end{align}
Since
$
\mathcal L_{\widetilde{\mathbb P}}
(\widetilde u^n,\widetilde B^n)
=
\mathcal L_{\mathbb P}(u^n,B^n),
$
and $M_{u,n}^{\phi}(t)$ is defined in \eqref{JKLT23}, we have
\begin{align}
\widetilde{\mathbb E}
\lvert M_{u,n}^{\phi}(t)\rvert^4
=
\mathbb E
\lvert M_{u,n}^{\phi}(t)\rvert^4.
\label{Rat15x}
\end{align}
Thus \eqref{h640} and
\eqref{Rat15x} imply
\begin{align}
\sup_{n\geq1}
\widetilde{\mathbb E}
\lvert M_{u,n}^{\phi}(t)\rvert^4
<\infty.
\label{aqw}
\end{align}
The $L^1(\widetilde\Omega)$-convergence \eqref{qaz} implies
\begin{align}
M_{u,n}^{\phi}(t)
\longrightarrow
M_u^\phi(t)
\qquad
\text{in }\widetilde{\mathbb P}\text{-probability}.
\label{kv2}
\end{align}
Choose a subsequence $\{n_k\}$ along which the convergence in
\eqref{kv2} holds $\widetilde{\mathbb P}$-a.s.  Fatou's lemma and
\eqref{aqw} then give
\begin{align}
\widetilde{\mathbb E}
\lvert M_u^\phi(t)\rvert^4
&\leq
\liminf_{k\to\infty}
\widetilde{\mathbb E}
\lvert M_{u,n_k}^{\phi}(t)\rvert^4
\nonumber\\
&\leq
\sup_{n\geq1}
\widetilde{\mathbb E}
\lvert M_{u,n}^{\phi}(t)\rvert^4
<\infty.
\label{F1O1R1}
\end{align}
The elementary inequality
$\lvert a-b\rvert^4\leq8(\lvert a\rvert^4+\lvert b\rvert^4)$,
together with \eqref{aqw} and \eqref{F1O1R1}, therefore
gives
\begin{align}
\sup_{n\geq1}
\widetilde{\mathbb E}
\lvert
M_{u,n}^{\phi}(t)-M_u^\phi(t)
\rvert^4
<\infty.
\label{U1f1x93}
\end{align}
For $R>0$, \eqref{U1f1x93} gives
\begin{align}
&
\sup_{n\geq1}
\widetilde{\mathbb E}
\left[
\lvert M_{u,n}^{\phi}(t)-M_u^\phi(t)\rvert^2
\mathbf 1_{
\{
\lvert M_{u,n}^{\phi}(t)-M_u^\phi(t)\rvert^2>R
\}}
\right]
\nonumber\\
&\qquad\leq
\frac{1}{R}
\sup_{n\geq1}
\widetilde{\mathbb E}
\lvert M_{u,n}^{\phi}(t)-M_u^\phi(t)\rvert^4
\longrightarrow0
\qquad\text{as }R\to\infty.
\label{hg30}
\end{align}
Thus \eqref{hg30} gives uniform integrability of
$\{\lvert M_{u,n}^{\phi}(t)-M_u^\phi(t)\rvert^2\}_{n\geq1}$,
while \eqref{kv2} gives
$\lvert M_{u,n}^{\phi}(t)-M_u^\phi(t)\rvert^2\to0$ in
$\widetilde{\mathbb P}$-probability. Vitali's theorem therefore
yields
\begin{align}
\widetilde{\mathbb E}
\lvert
M_{u,n}^{\phi}(t)-M_u^\phi(t)
\rvert^2
\longrightarrow0.
\label{6gvx}
\end{align}
For $M_{B,n}^{\psi}$, using
$\lVert P_n\psi\rVert_{H^1}\leq\lVert\psi\rVert_{H^1}$, and $\lVert G^\ast v\rVert_{\mathfrak U_1}^2 \leq \lVert G\rVert_{L_2(\mathfrak U_1;H^{-1})}^2 \lVert v\rVert_{H^1}^2$ and applying the same argument from \eqref{lkc}  to \eqref{6gvx} , we obtain

\begin{align}
\widetilde{\mathbb E}
\lvert
M_{B,n}^{\psi}(t)-M_B^\psi(t)
\rvert^2
\longrightarrow0.
\label{B11R12l}
\end{align}
Combining \eqref{6gvx} and \eqref{B11R12l}, for every
fixed $t\in[0,T]$, we obtain
\begin{align}
M_{u,n}^{\phi}(t)
&\longrightarrow
M_u^\phi(t)
\qquad
\text{in }L^2(\widetilde\Omega),
\label{4saq}
\\
M_{B,n}^{\psi}(t)
&\longrightarrow
M_B^\psi(t)
\qquad
\text{in }L^2(\widetilde\Omega).
\label{GHr12}
\end{align}
Since $M_{u,n}^{\phi}(t)$ is square-integrable by
\eqref{ghf11a}, and similarly $M_{B,n}^{\psi}(t)$, it follows 
 \eqref{4saq} and
\eqref{GHr12} that
\begin{align}
M_u^\phi(t)
&\in
L^2(\widetilde\Omega),
\qquad
t\in[0,T],
\label{ISLAMR1}
\\
M_B^\psi(t)
&\in
L^2(\widetilde\Omega),
\qquad
t\in[0,T].
\label{HJKT12}
\end{align}
Since $T>0$ was arbitrary,
\eqref{ISLAMR1} and
\eqref{HJKT12} hold for every $t\geq0$.
\medskip

\noindent
\textbf{Part 4(i). Proof of $\left\langle\!\left\langle M_u^\phi\right\rangle\!\right\rangle_t=\int_0^t\left\lVert\sigma_1^\ast(s,\widetilde u(s),\widetilde B(s))\phi\right\rVert_{\mathfrak U_1}^2\,ds$.}

\medskip
\noindent
 By \eqref{lkc} and \eqref{kas},
It\^o's formula applied to $x\mapsto x^2$ gives
\begin{align}
&
\left(M_{u,n}^{\phi}(t)\right)^2
-
\left(M_{u,n}^{\phi}(r)\right)^2
-
\int_r^t
\left\lVert
\Pi_n^{(1)}
\sigma_1^\ast
\bigl(s,u^n(s),B^n(s)\bigr)P_n\phi
\right\rVert_{\mathfrak U_1}^2\,ds
\nonumber\\
&\qquad=
2\int_r^t
M_{u,n}^{\phi}(s)
\left(
\Pi_n^{(1)}
\sigma_1^\ast
\bigl(s,u^n(s),B^n(s)\bigr)P_n\phi,
dW_1(s)
\right)_{\mathfrak U_1}.
\label{8bc}
\end{align}
The integrand on the right-hand side of \eqref{8bc} is predictable.
Moreover, by \eqref{h640},
\eqref{h840},
\eqref{940h}, and \eqref{L^4-mo},
\begin{align}
&
\mathbb E\int_0^T
\lvert M_{u,n}^{\phi}(s)\rvert^2
\left\lVert
\Pi_n^{(1)}
\sigma_1^\ast
\bigl(s,u^n(s),B^n(s)\bigr)P_n\phi
\right\rVert_{\mathfrak U_1}^2\,ds
<\infty,
\end{align}
which shows that the stochastic integral in \eqref{8bc} is a square-integrable
martingale. Since $g(u^n,B^n)$ is $\mathcal F_r^n$-measurable,
taking conditional expectation in \eqref{8bc} and using the tower
property yields
\begin{align}
&
\mathbb E
\left[
\left(
\left(M_{u,n}^{\phi}(t)\right)^2
-
\left(M_{u,n}^{\phi}(r)\right)^2
-
\int_r^t
\left\lVert
\Pi_n^{(1)}
\sigma_1^\ast
\bigl(s,u^n(s),B^n(s)\bigr)P_n\phi
\right\rVert_{\mathfrak U_1}^2\,ds
\right)
g(u^n,B^n)
\right]
=0.
\label{hjy5Lx}
\end{align}
By $\mathcal L_{\widetilde{\mathbb P}}
(\widetilde u^n,\widetilde B^n)
=
\mathcal L_{\mathbb P}
(u^n,B^n)$, and from  \eqref{hjy5Lx}, we obtain
\begin{align}
&
\widetilde{\mathbb E}
\left[
\left(
\left(M_{u,n}^{\phi}(t)\right)^2
-
\left(M_{u,n}^{\phi}(r)\right)^2
-
\int_r^t
\left\lVert
\Pi_n^{(1)}
\sigma_1^\ast
\bigl(s,\widetilde u^n(s),\widetilde B^n(s)\bigr)P_n\phi
\right\rVert_{\mathfrak U_1}^2\,ds
\right)
g(\widetilde u^n,\widetilde B^n)
\right]
=0.
\label{9HE}
\end{align}
We next pass to the limit in \eqref{9HE}. For every $q\in[0,T]$,
\eqref{4saq}, \eqref{aqw}, and \eqref{F1O1R1} give
\begin{align}
&
\widetilde{\mathbb E}
\left\lvert
\left(M_{u,n}^{\phi}(q)\right)^2
-
\left(M_u^\phi(q)\right)^2
\right\rvert
\leq
\left(
\widetilde{\mathbb E}
\lvert M_{u,n}^{\phi}(q)-M_u^\phi(q)\rvert^2
\right)^{1/2}
\left(
\widetilde{\mathbb E}
\lvert M_{u,n}^{\phi}(q)+M_u^\phi(q)\rvert^2
\right)^{1/2}
\longrightarrow0.
\label{34sa}
\end{align}
To pass to the limit in the remaining term in \eqref{9HE}, we write
\begin{align}
\Pi_n^{(1)}
\sigma_1^\ast
(s,\widetilde u^n,\widetilde B^n)P_n\phi
-
\sigma_1^\ast
(s,\widetilde u,\widetilde B)\phi
&=
\Pi_n^{(1)}
\left[
\sigma_1^\ast(s,\widetilde u^n,\widetilde B^n)
-
\sigma_1^\ast(s,\widetilde u,\widetilde B)
\right]P_n\phi
\nonumber\\
&\quad+
\Pi_n^{(1)}
\sigma_1^\ast(s,\widetilde u,\widetilde B)
(P_n\phi-\phi)
+
(\Pi_n^{(1)}-I)
\sigma_1^\ast(s,\widetilde u,\widetilde B)\phi.
\label{9xc}
\end{align}
Since
\begin{align}
\lVert\sigma\rVert_{L_2(\mathfrak U_1;H^{-1})}
=
\lVert\sigma^\ast\rVert_{L_2(H^1;\mathfrak U_1)},
\label{2qa}
\end{align}
 \eqref{NlH},
\eqref{TUL2H} and \eqref{TBL2H2}, and the boundedness of
$P_n\phi$ in $H^1$ give
\begin{align}
&\int_0^T
\left\lVert
\left[
\sigma_1^\ast(s,\widetilde u^n,\widetilde B^n)
-
\sigma_1^\ast(s,\widetilde u,\widetilde B)
\right]P_n\phi
\right\rVert_{\mathfrak U_1}^2\,ds
\nonumber\\
&\qquad\leq
C\lVert P_n\phi\rVert_{H^1}^2
\int_0^T
\left(
\lVert\widetilde u^n-\widetilde u\rVert_H^2
+
\lVert\widetilde B^n-\widetilde B\rVert_H^2
\right)\,ds
\longrightarrow0.
\end{align}
Moreover, $P_n\phi\to\phi$ in $H^1$ and \eqref{NH1} imply that the
second term in \eqref{9xc} tends to zero, while
$\Pi_n^{(1)}\to I$, \eqref{NH1}, and dominated convergence imply
the same for the third term. Consequently,
\begin{align}
&
\int_0^T
\left\lVert
\Pi_n^{(1)}
\sigma_1^\ast
(s,\widetilde u^n,\widetilde B^n)P_n\phi
-
\sigma_1^\ast
(s,\widetilde u,\widetilde B)\phi
\right\rVert_{\mathfrak U_1}^2\,ds
\longrightarrow0,
\qquad
\widetilde{\mathbb P}\text{-a.s.}
\label{erX}
\end{align}
Using $\lvert a^2-b^2\rvert=\lvert a-b\rvert(a+b)$ and
Cauchy-Schwarz in time, \eqref{erX} gives, for every $q\in[0,T]$,
\begin{align}
\int_0^q
\left\lVert
\Pi_n^{(1)}
\sigma_1^\ast
(s,\widetilde u^n,\widetilde B^n)P_n\phi
\right\rVert_{\mathfrak U_1}^2\,ds
\longrightarrow
\int_0^q
\left\lVert
\sigma_1^\ast
(s,\widetilde u,\widetilde B)\phi
\right\rVert_{\mathfrak U_1}^2\,ds,
\qquad
\widetilde{\mathbb P}\text{-a.s.}
\label{dx23}
\end{align}
Moreover, \eqref{NH1} and \eqref{L^4-mo} imply
\begin{align}
\sup_{n\geq1}
\widetilde{\mathbb E}
\left(
\int_0^q
\left\lVert
\Pi_n^{(1)}
\sigma_1^\ast
(s,\widetilde u^n,\widetilde B^n)P_n\phi
\right\rVert_{\mathfrak U_1}^2\,ds
\right)^2
<\infty.
\label{12az}
\end{align}
Thus \eqref{dx23}, \eqref{12az}, and Vitali's theorem yield
\begin{align}
&
\widetilde{\mathbb E}
\left\lvert
\int_0^q
\left\lVert
\Pi_n^{(1)}
\sigma_1^\ast
(s,\widetilde u^n,\widetilde B^n)P_n\phi
\right\rVert_{\mathfrak U_1}^2\,ds
-
\int_0^q
\left\lVert
\sigma_1^\ast
(s,\widetilde u,\widetilde B)\phi
\right\rVert_{\mathfrak U_1}^2\,ds
\right\rvert
\longrightarrow0.
\label{7rq}
\end{align}
Combining \eqref{34sa} at $q=t,r$ with \eqref{7rq} at $q=t,r$
gives
\begin{align}
&
\widetilde{\mathbb E}
\left\lvert
\left[
(M_{u,n}^{\phi}(t))^2-(M_{u,n}^{\phi}(r))^2
-
\int_r^t
\left\lVert
\Pi_n^{(1)}
\sigma_1^\ast
(s,\widetilde u^n,\widetilde B^n)P_n\phi
\right\rVert_{\mathfrak U_1}^2\,ds
\right]
\right.
\nonumber\\
&\hspace{20mm}\left.
-
\left[
(M_u^\phi(t))^2-(M_u^\phi(r))^2
-
\int_r^t
\left\lVert
\sigma_1^\ast
(s,\widetilde u,\widetilde B)\phi
\right\rVert_{\mathfrak U_1}^2\,ds
\right]
\right\rvert
\longrightarrow0.
\label{Hjk5s1}
\end{align}
Together with \eqref{M2C}, the boundedness of $g$, and
\eqref{Hjk5s1}, this allows passage to the
limit in \eqref{9HE}, giving
\begin{align}
&
\widetilde{\mathbb E}
\left[
\left(
(M_u^\phi(t))^2-(M_u^\phi(r))^2
-
\int_r^t
\left\lVert
\sigma_1^\ast
(s,\widetilde u,\widetilde B)\phi
\right\rVert_{\mathfrak U_1}^2\,ds
\right)
g(\widetilde u,\widetilde B)
\right]
=0.
\label{KL1ev3}
\end{align}
By the same monotone-class argument used in
\eqref{MCV}-\eqref{KHJ54c}, it follows from \eqref{KL1ev3} and \eqref{65xd} that
\begin{align*}
&
\widetilde{\mathbb E}
\left[
\left.
(M_u^\phi(t))^2-(M_u^\phi(r))^2
-
\int_r^t
\left\lVert
\sigma_1^\ast
(s,\widetilde u,\widetilde B)\phi
\right\rVert_{\mathfrak U_1}^2\,ds
\,\right\rVert,
\widetilde{\mathcal F}_r
\right]
=0,
\end{align*}
consequently,
\begin{align}
\left(M_u^\phi(t)\right)^2
-
\int_0^t
\left\lVert
\sigma_1^\ast
(s,\widetilde u(s),\widetilde B(s))\phi
\right\rVert_{\mathfrak U_1}^2\,ds
\quad\text{is an }
\{\widetilde{\mathcal F}_t\}_{t\in[0,T]}
\text{-martingale}.
\label{C1M12C1}
\end{align}
Since \eqref{CM2L1},
\eqref{ISLAMR1}, and
\eqref{GHTL2} show that $M_u^\phi$ is a continuous
square-integrable martingale, it follows from 
\eqref{C1M12C1} that
\begin{align}
\left\langle\!\left\langle M_u^\phi
\right\rangle\!\right\rangle_t
=
\int_0^t
\left\lVert
\sigma_1^\ast
(s,\widetilde u(s),\widetilde B(s))\phi
\right\rVert_{\mathfrak U_1}^2\,ds.
\label{A1U13}
\end{align}
\textbf{Part 4(ii). Proof of $\left\langle\!\left\langle M_B^\psi\right\rangle\!\right\rangle_t=\int_0^t\left\lVert\sigma_2^\ast(s,\widetilde u(s),\widetilde B(s))\psi\right\rVert_{\mathfrak U_2}^2\,ds$.}

\noindent For $M_B^{\psi}$, applying the same argument from \eqref{8bc} to \eqref{C1M12C1}, we obtain
\begin{align}
\left(M_B^\psi(t)\right)^2
-
\int_0^t
\left\lVert
\sigma_2^\ast
(s,\widetilde u(s),\widetilde B(s))\psi
\right\rVert_{\mathfrak U_2}^2\,ds
\quad\text{is an }
\{\widetilde{\mathcal F}_t\}_{t\in[0,T]}
\text{-martingale}.
\label{S1B1M2}
\end{align}
Since \eqref{L1M1A1},
\eqref{HJKT12}, and \eqref{lr2q} show that
$M_B^\psi$ is a continuous square-integrable martingale, it follows from \eqref{S1B1M2} that
\begin{align}
\left\langle\!\left\langle M_B^\psi
\right\rangle\!\right\rangle_t
=
\int_0^t
\left\lVert
\sigma_2^\ast
(s,\widetilde u(s),\widetilde B(s))\psi
\right\rVert_{\mathfrak U_2}^2\,ds.
\label{L1sec2}
\end{align}
\textbf{Part 4(iii) Proof of $\left\langle\!\left\langle M_u^\phi,M_B^\psi\right\rangle\!\right\rangle_t=0$.}

\noindent Fix divergence-free
$\phi,\psi\in C^\infty(\mathbb T^2;\mathbb R^2)$.
By \eqref{lkc}, $M_{u,n}^{\phi}$ is a stochastic integral with
respect to $W_1$, and similarly $M_{B,n}^{\psi}$ is a stochastic
integral with respect to $W_2$. Hence
\begin{align}
\left\langle\!\left\langle
M_{u,n}^{\phi},
M_{B,n}^{\psi}
\right\rangle\!\right\rangle_t
=
0,
\qquad
0\leq t\leq T.
\label{G1s1n1}
\end{align}
Since $M_{u,n}^{\phi}(0)=M_{B,n}^{\psi}(0)=0$, It\^o's product
formula, together with \eqref{G1s1n1}, gives
\begin{align}
M_{u,n}^{\phi}(t)M_{B,n}^{\psi}(t)
&=
\int_0^t
M_{B,n}^{\psi}(s)
\left(
\Pi_n^{(1)}
\sigma_1^\ast(s,u^n,B^n)P_n\phi,
dW_1(s)
\right)_{\mathfrak U_1}
\nonumber\\
&\quad+
\int_0^t
M_{u,n}^{\phi}(s)
\left(
\Pi_n^{(2)}
\sigma_2^\ast(s,u^n,B^n)P_n\psi,
dW_2(s)
\right)_{\mathfrak U_2}.
\label{h540}
\end{align}
The fourth-moment bound \eqref{h640}, together with
\eqref{h840}, \eqref{940h}, and \eqref{L^4-mo}, shows that the first
stochastic integral on the right-hand side of \eqref{h540} is a
square-integrable martingale, and similarly for the second one.
. Thus, as in
\eqref{Cx3zk4}-\eqref{L1zf4c"},
for $0\leq r<t\leq T$, it follows from
\eqref{h540} that
\begin{align}
\mathbb E
\left[
\left(
M_{u,n}^{\phi}(t)M_{B,n}^{\psi}(t)
-
M_{u,n}^{\phi}(r)M_{B,n}^{\psi}(r)
\right)
g(u^n,B^n)
\right]
=
0.
\label{5tc3z}
\end{align}
By $\mathcal L_{\widetilde{\mathbb P}} (\widetilde u^n,\widetilde B^n) = \mathcal L_{\mathbb P} (u^n,B^n)$, we obtain from \eqref{5tc3z} as follows
\begin{align}
\widetilde{\mathbb E}
\left[
\left(
M_{u,n}^{\phi}(t)M_{B,n}^{\psi}(t)
-
M_{u,n}^{\phi}(r)M_{B,n}^{\psi}(r)
\right)
g(\widetilde u^n,\widetilde B^n)
\right]
=
0.
\label{T1P1M2}
\end{align}
For every fixed $q\in[0,T]$, the convergences
\eqref{4saq} and \eqref{GHr12}, together with
\eqref{ISLAMR1}-\eqref{HJKT12}, yield
\begin{align}
&
\widetilde{\mathbb E}
\left\lvert
M_{u,n}^{\phi}(q)M_{B,n}^{\psi}(q)
-
M_u^\phi(q)M_B^\psi(q)
\right\rvert
\nonumber\\
&\leq
\left(
\widetilde{\mathbb E}
\left\lvert
M_{u,n}^{\phi}(q)-M_u^\phi(q)
\right\rvert^2
\right)^{1/2}
\left(
\widetilde{\mathbb E}
\left\lvert
M_{B,n}^{\psi}(q)
\right\rvert^2
\right)^{1/2}
\nonumber\\
&\quad+
\left(
\widetilde{\mathbb E}
\left\lvert
M_u^\phi(q)
\right\rvert^2
\right)^{1/2}
\left(
\widetilde{\mathbb E}
\left\lvert
M_{B,n}^{\psi}(q)-M_B^\psi(q)
\right\rvert^2
\right)^{1/2}
\longrightarrow0,
\end{align}
which implies, more precisely, \begin{align}
M_{u,n}^{\phi}(q)M_{B,n}^{\psi}(q)
\longrightarrow
M_u^\phi(q)M_B^\psi(q)
\qquad
\text{in }L^1(\widetilde\Omega).
\label{JK42A1}
\end{align}
Applying \eqref{JK42A1} at $q=t,r$, we write
\begin{align}
&
\left(
M_{u,n}^{\phi}(t)M_{B,n}^{\psi}(t)
-
M_{u,n}^{\phi}(r)M_{B,n}^{\psi}(r)
\right)
g(\widetilde u^n,\widetilde B^n)
\nonumber\\
&\quad-
\left(
M_u^\phi(t)M_B^\psi(t)
-
M_u^\phi(r)M_B^\psi(r)
\right)
g(\widetilde u,\widetilde B)
\nonumber\\
&=
\left[
M_{u,n}^{\phi}(t)M_{B,n}^{\psi}(t)
-
M_u^\phi(t)M_B^\psi(t)
\right]
g(\widetilde u^n,\widetilde B^n)
\nonumber\\
&\quad-
\left[
M_{u,n}^{\phi}(r)M_{B,n}^{\psi}(r)
-
M_u^\phi(r)M_B^\psi(r)
\right]
g(\widetilde u^n,\widetilde B^n)
\nonumber\\
&\quad+
\left(
M_u^\phi(t)M_B^\psi(t)
-
M_u^\phi(r)M_B^\psi(r)
\right)
\left[
g(\widetilde u^n,\widetilde B^n)
-
g(\widetilde u,\widetilde B)
\right].
\end{align}
By \eqref{JK42A1} at $q=t,r$ and the boundedness of
$g$, the first two terms on the right-hand side converge to zero in
$L^1(\widetilde\Omega)$. Moreover, by
\eqref{ISLAMR1}-\eqref{HJKT12} and H\"older's inequality,
\begin{align}
M_u^\phi(t)M_B^\psi(t)
-
M_u^\phi(r)M_B^\psi(r)
\in L^1(\widetilde\Omega).
\end{align}
Therefore, \eqref{M2C} gives the almost sure convergence of the
last factor in the third term, while \eqref{M2Lo} provides the
integrable domination. Hence, by the dominated convergence theorem,
the third term also converges to zero in $L^1(\widetilde\Omega)$.
Consequently, passing to the limit in
\eqref{T1P1M2} gives
\begin{align}
\widetilde{\mathbb E}
\left[
\left(
M_u^\phi(t)M_B^\psi(t)
-
M_u^\phi(r)M_B^\psi(r)
\right)
g(\widetilde u,\widetilde B)
\right]
=
0.
\label{j45t1}
\end{align}
Using the same
monotone-class argument used in passing from \eqref{MCV} to
\eqref{KHJ54c}, together with \eqref{KLCx}, \eqref{0oL3},
\eqref{ISLAMR1}, and
\eqref{HJKT12}, yields
\begin{align}
M_u^\phi M_B^\psi
\quad\textrm{is an }
\{\widetilde{\mathcal F}_t\}_{t\in[0,T]}\textrm{-martingale}.
\label{Mh5}
\end{align}
By \eqref{GHTL2}, \eqref{lr2q},
\eqref{ISLAMR1}, and
\eqref{HJKT12}, $M_u^\phi$ and $M_B^\psi$ are
square-integrable martingales. Since \eqref{Mh5} gives
\begin{align}
M_u^\phi M_B^\psi-0
\quad\textrm{is a martingale},
\end{align}
and the zero process is predictable and of finite variation, the
uniqueness assertion in
\cite[Theorem~I.4.2]{jacodshiryaev} gives
\begin{align}
\left\langle\!\left\langle
M_u^\phi,M_B^\psi
\right\rangle\!\right\rangle_t
=
0,
\qquad
0\leq t\leq T,
\qquad
\widetilde{\mathbb P}\text{-a.s.}
\label{jb23}
\end{align}
Combining \eqref{jb23} with
\eqref{A1U13} and
\eqref{L1sec2}, we obtain
\begin{equation}
\begin{aligned}
&\left\langle\!\left\langle
M_u^\phi
\right\rangle\!\right\rangle_t
=
\int_0^t
\left\lVert
\sigma_1^\ast
(s,\widetilde u(s),\widetilde B(s))\phi
\right\rVert_{\mathfrak U_1}^2\,ds,
\\
&\left\langle\!\left\langle
M_B^\psi
\right\rangle\!\right\rangle_t
=
\int_0^t
\left\lVert
\sigma_2^\ast
(s,\widetilde u(s),\widetilde B(s))\psi
\right\rVert_{\mathfrak U_2}^2\,ds,
\\
&\left\langle\!\left\langle
M_u^\phi,M_B^\psi
\right\rangle\!\right\rangle_t
=
0.
\label{OPF}
\end{aligned}
\end{equation}
\textit{We now transfer the verification of {\rm (M2)} from
$\widetilde{\mathbb P}$ to $\widehat{\mathbb P}$.}
From 
 \eqref{GHTL2}, \eqref{lr2q},
\eqref{CM2L1}, 
\eqref{L1M1A1}, \eqref{ISLAMR1}, and \eqref{HJKT12} together  with \eqref{LLaW},  we obtain 
\begin{align}
M_u^\phi\; \textrm{and}\; M_B^\psi\; \textrm{are continuous square-integrable}\; (\mathcal F_t,\widehat{\mathbb P})\;\textrm{-martingales on}\; \mathbb R_+.
\label{AS1X1}
\end{align}
By \eqref{C1M12C1}, we obtain
\begin{align}
&
\widetilde{\mathbb E}
\left[
\left(
(M_u^\phi(t))^2
-
(M_u^\phi(r))^2
-
\int_r^t
\left\lVert 
\sigma_1^\ast
\bigl(
s,\widetilde u(s),\widetilde B(s)
\bigr)\phi
\right\rVert_{\mathfrak U_1}^2\,ds
\right)
g(\widetilde u,\widetilde B)
\right]
=
0.
\label{degJ5}
\end{align}
Hence, using \eqref{LLaW}, from \eqref{degJ5} we obtain
\begin{align}
&
\widehat{\mathbb E}
\left[
\left(
(M_u^\phi(t))^2
-
(M_u^\phi(r))^2
-
\int_r^t
\left\lVert 
\sigma_1^\ast
\bigl(
s,u(s),B(s)
\bigr)\phi
\right\rVert_{\mathfrak U_1}^2\,ds
\right)
g(u,B)
\right]
\nonumber\\
&\qquad
=
\widetilde{\mathbb E}
\left[
\left(
(M_u^\phi(t))^2
-
(M_u^\phi(r))^2
-
\int_r^t
\left\lVert 
\sigma_1^\ast
\bigl(
s,\widetilde u(s),\widetilde B(s)
\bigr)\phi
\right\rVert_{\mathfrak U_1}^2\,ds
\right)
g(\widetilde u,\widetilde B)
\right]
\nonumber\\
&\qquad
=
0,
\end{align}
consequently,
\begin{align}
\left(M_u^\phi(t)\right)^2
-
\int_0^t
\left\lVert 
\sigma_1^\ast
\bigl(
s,u(s),B(s)
\bigr)\phi
\right\rVert_{\mathfrak U_1}^2\,ds
\label{JK5t}
\end{align}
is an
$(\mathcal F_t,\widehat{\mathbb P})$-martingale. By \eqref{GHTL2}, \eqref{CM2L1} and \eqref{ISLAMR1} with \eqref{LLaW}, 
$M_u^\phi$ is a continuous square-integrable martingale, and using the same argument in \eqref{jb23} gives
\begin{align}
\left\langle\!\left\langle
M_u^\phi
\right\rangle\!\right\rangle_t
=
\int_0^t
\left\lVert 
\sigma_1^\ast
\bigl(
s,u(s),B(s)
\bigr)\phi
\right\rVert_{\mathfrak U_1}^2\,ds,
\qquad t \ge 0 \qquad
\widehat{\mathbb P}\text{-a.s.}
\label{L1M1J1}
\end{align}
Similarly, by \eqref{S1B1M2}, we obtain
\begin{align}
&
\widetilde{\mathbb E}
\left[
\left(
(M_B^\psi(t))^2
-
(M_B^\psi(r))^2
-
\int_r^t
\left\lVert 
\sigma_2^\ast
\bigl(
s,\widetilde u(s),\widetilde B(s)
\bigr)\psi
\right\rVert_{\mathfrak U_2}^2\,ds
\right)
g(\widetilde u,\widetilde B)
\right]
=
0.
\end{align}
Therefore, by \eqref{LLaW}, we obtain
\begin{align}
&
\widehat{\mathbb E}
\left[
\left(
(M_B^\psi(t))^2
-
(M_B^\psi(r))^2
-
\int_r^t
\left\lVert 
\sigma_2^\ast
\bigl(
s,u(s),B(s)
\bigr)\psi
\right\rVert_{\mathfrak U_2}^2\,ds
\right)
g(u,B)
\right]
\nonumber\\
&\qquad
=
\widetilde{\mathbb E}
\left[
\left(
(M_B^\psi(t))^2
-
(M_B^\psi(r))^2
-
\int_r^t
\left\lVert 
\sigma_2^\ast
\bigl(
s,\widetilde u(s),\widetilde B(s)
\bigr)\psi
\right\rVert_{\mathfrak U_2}^2\,ds
\right)
g(\widetilde u,\widetilde B)
\right]
\nonumber\\
&\qquad
=
0,
\end{align}
consequently,
\begin{align}
\left(M_B^\psi(t)\right)^2
-
\int_0^t
\left\lVert 
\sigma_2^\ast
\bigl(
s,u(s),B(s)
\bigr)\psi
\right\rVert_{\mathfrak U_2}^2\,ds
\label{yt43}
\end{align}
is an
$(\mathcal F_t,\widehat{\mathbb P})$-martingale. Hence
\begin{align}
\left\langle\!\left\langle
M_B^\psi
\right\rangle\!\right\rangle_t
=
\int_0^t
\left\lVert 
\sigma_2^\ast
\bigl(
s,u(s),B(s)
\bigr)\psi
\right\rVert_{\mathfrak U_2}^2\,ds,
\qquad t\ge 0 \qquad
\widehat{\mathbb P}\text{-a.s.}
\label{GHR52}
\end{align}
Finally, by \eqref{Mh5}, we obtain
\begin{align}
\widetilde{\mathbb E}
\left[
\left(
M_u^\phi(t)M_B^\psi(t)
-
M_u^\phi(r)M_B^\psi(r)
\right)
g(\widetilde u,\widetilde B)
\right]
=
0.
\end{align}
Using \eqref{LLaW}, we obtain
$$
\widehat{\mathbb E}
\left[
\left(
M_u^\phi(t)M_B^\psi(t)
-
M_u^\phi(r)M_B^\psi(r)
\right)
g(u,B)
\right]
=
\widetilde{\mathbb E}
\left[
\left(
M_u^\phi(t)M_B^\psi(t)
-
M_u^\phi(r)M_B^\psi(r)
\right)
g(\widetilde u,\widetilde B)
\right]
=
0,
$$
hence
\begin{align}
M_u^\phi M_B^\psi
\quad\text{is an }
(\mathcal F_t,\widehat{\mathbb P})\text{-martingale}.
\end{align}
Since \eqref{GHTL2}, \eqref{lr2q},
\eqref{CM2L1}, 
\eqref{L1M1A1}, \eqref{ISLAMR1}, and \eqref{HJKT12} with \eqref{LLaW} show that $M_u^\phi$ and
$M_B^\psi$ are continuous square-integrable martingales, while
\eqref{Mh5} shows that $M_u^\phi M_B^\psi$ is a martingale, the uniqueness
assertion in \cite[Theorem~I.4.2]{jacodshiryaev} yields
\begin{align}
\left\langle\!\left\langle
M_u^\phi,M_B^\psi
\right\rangle\!\right\rangle_t
=
0,
\qquad
t\ge 0,
\qquad
\widehat{\mathbb P}\text{-a.s.}
\label{CC1Mw1}
\end{align}
Combining
\eqref{AS1X1},
\eqref{L1M1J1},
\eqref{GHR52},
and
\eqref{CC1Mw1},
we conclude that
$\widehat{\mathbb P}$
satisfies Definition~\ref{DEF}{\rm (M2)}.
\medskip

\noindent We now turn to the
verification of Definition~\ref{DEF}{\rm (M3)}.

\paragraph{Verification of Definition~\ref{DEF} $(M3)$.}

By \eqref{olTilde} and weak lower semicontinuity,
\begin{align}
\widetilde{\mathbb E}
\Bigg[
\sup_{0\le t\le T}
\bigl(
\lVert\widetilde u(t)\rVert_{L^2}^2
+
\lVert\widetilde B(t)\rVert_{L^2}^2
\bigr)
+
\int_0^T
\bigl(
\nu\lVert\partial_1\widetilde u\rVert_{L^2}^2
+
\eta\lVert\nabla\widetilde B\rVert_{L^2}^2
\bigr)\,dt
\Bigg]
\le C_T.
\label{M3M}
\end{align}
Since \eqref{LLaW} gives
$\widehat{\mathbb P}
=
\mathcal L_{\widetilde{\mathbb P}}(\widetilde u,\widetilde B)$,
the  estimate \eqref{M3M} yields
\begin{align}
\mathbb E^{\widehat{\mathbb P}}
\Bigg[
\sup_{0\le t\le T}
\bigl(
\lVert u(t)\rVert_{L^2}^2
+
\lVert B(t)\rVert_{L^2}^2
\bigr)
+
\int_0^T
\bigl(
\nu\lVert\partial_1u\rVert_{L^2}^2
+
\eta\lVert\nabla B\rVert_{L^2}^2
\bigr)\,dt
\Bigg]
\le C_T,
\end{align}
which shows that $(M3)$ holds.
\noindent

\medskip
\noindent
We now use the martingale representation theorem to pass from the
martingale solution established above to a global probabilistically
weak solution of \eqref{MHD} on $[0,\infty)$ in the sense of
Definition~\ref{d-w-s}.\qed
\subsection*{Proof of Theorem~\ref{thm10} (Part II: Existence of a probabilistically weak solution)}

\noindent We divide the argument into four parts.  In Part~1, we construct a
continuous square-integrable $H\times H$-valued martingale
$\mathbf M$. In Part~2, we construct a predictable $L_2^0$-valued
process $\Phi$ and prove that
$\widehat{\mathbb E}\int_0^T
\lVert\Phi(s)\rVert_{L_2^0}^2\,ds<\infty$ for every $T>0$. In Part~3,
we identify the operator-valued quadratic variation of $\mathbf M$ by
showing that
$$
\left\langle\!\left\langle\mathbf M\right\rangle\!\right\rangle_t
=
\int_0^t
\bigl(\Phi(s)Q^{1/2}\bigr)
\bigl(\Phi(s)Q^{1/2}\bigr)^\ast\,ds,
\qquad t\geq0.
$$
Finally, in Part~4, we apply the martingale representation theorem and
recover a global probabilistically weak solution of \eqref{MHD}.

\medskip
\noindent 
\textit{Part 1. Construction of an $H\times H$-valued martingale.}

\noindent Fix an arbitrary $T>0$.  We work on the stochastic basis
$(\bar\Omega,\mathcal B(\bar\Omega),
\{\mathcal F_t\}_{t\geq0},\widehat{\mathbb P})$, where
$\mathcal F_t=\sigma(u(s),B(s):0\leq s\leq t)$ by \eqref{MF1}.
By \eqref{GHTL2}, \eqref{lr2q},
\eqref{CM2L1}, 
\eqref{L1M1A1}, \eqref{ISLAMR1}, and \eqref{HJKT12} with \eqref{LLaW},
$M_u^\phi$ and $M_B^\psi$ are continuous square-integrable
$(\mathcal F_t,\widehat{\mathbb P})$-martingales for every
divergence-free $\phi,\psi\in C^\infty$. 
Using the divergence-free Fourier orthonormal basis $\{e_k\}_{k\geq1}$ 
of $H$ given in section \ref{5GHT}, we define
\begin{align}
M_{u,N}(t)&:=\sum_{k=1}^N M_u^{e_k}(t)e_k,
\label{L2p2w2v}\\
M_{B,N}(t)&:=\sum_{k=1}^N M_B^{e_k}(t)e_k,
\qquad
\mathbf M_N(t):=(M_{u,N}(t),M_{B,N}(t)).
\label{L1P1w1}
\end{align}
Since these sums are finite, $\mathbf M_N$ is a continuous
square-integrable $H\times H$-valued martingale and
\begin{align}
\widehat{\mathbb E}
[\mathbf M_N(t)\mid\mathcal F_r]
=\mathbf M_N(r),
\qquad
\widehat{\mathbb E}
\lVert\mathbf M_N(t)\rVert_{H\times H}^2<\infty,
\qquad 0\leq r\leq t.
\label{k1w1xg}
\end{align}
For $N>M$, orthonormality, Doob's $L^2$-inequality \cite[Theorem~20, p.~11]{protter2004stochastic}, 
\eqref{L1M1J1} and \eqref{GHR52} imply
\begin{align}
\widehat{\mathbb E}\sup_{0\leq t\leq T}
\lVert\mathbf M_N(t)-\mathbf M_M(t)\rVert_{H\times H}^2
&\leq
4\widehat{\mathbb E}\int_0^T
\sum_{k=M+1}^N
\lVert\sigma_1^\ast(s,u,B)e_k\rVert_{\mathfrak U_1}^2\,ds
\nonumber\\
&\quad+
4\widehat{\mathbb E}\int_0^T
\sum_{k=M+1}^N
\lVert\sigma_2^\ast(s,u,B)e_k\rVert_{\mathfrak U_2}^2\,ds.
\label{Y1x1La}
\end{align}
Since
\begin{align*}
\sum_{k=1}^\infty
\lVert\sigma_i^\ast(s,u,B)e_k\rVert_{\mathfrak U_i}^2
=
\lVert\sigma_i(s,u,B)\rVert_{L_2(\mathfrak U_i;H)}^2,
\qquad i=1,2,
\end{align*}
the sums on the right-hand side of \eqref{Y1x1La} converge to zero and
are dominated by \eqref{HGRT1}. Hence,
by dominated convergence,
\begin{align}
\lim_{M,N\to\infty}
\widehat{\mathbb E}\sup_{0\leq t\leq T}
\lVert\mathbf M_N(t)-\mathbf M_M(t)\rVert_{H\times H}^2=0.
\label{lw12xa}
\end{align}
Hence, by \eqref{lw12xa} and the completeness of
$L^2(\bar\Omega;C([0,T];H\times H))$, there exists
$\mathbf M=(M_u,M_B)\in L^2(\bar\Omega;C([0,T];H\times H))$ such that
\begin{align}
\mathbf M_N\longrightarrow\mathbf M
\quad\text{in }
L^2(\bar\Omega;C([0,T];H\times H)).
\label{zo1p5x}
\end{align}
Since $T>0$ is arbitrary and the limit in
$L^2(\widehat\Omega;C([0,T];H\times H))$ is unique,
\eqref{zo1p5x} defines a 
$\mathbf M=(\mathbf M(t))_{t\geq0}$ on $\mathbb R_+$ such that
\begin{align}
\widehat{\mathbb E}\sup_{0\leq t\leq T}
\lVert\mathbf M(t)\rVert_{H\times H}^2<\infty,
\qquad T>0.
\label{L1e45.}
\end{align}
For $0\leq r\leq t\leq T$,
\eqref{zo1p5x} gives
$\mathbf M_N(r)\to\mathbf M(r)$ and
$\mathbf M_N(t)\to\mathbf M(t)$ in
$L^2(\bar\Omega;H\times H)$. Since $\mathbf M_N(t)$ is
$\mathcal F_t$-measurable, closedness of
$L^2(\bar\Omega,\mathcal F_t,\widehat{\mathbb P};H\times H)$ gives
\begin{align}
\mathbf M(t)\quad\text{is }\mathcal F_t\text{-measurable}.
\label{A1p1w1d}
\end{align}
Moreover, for $G\in L^\infty(\mathcal F_r)$,
\eqref{k1w1xg} and
\eqref{zo1p5x} give
\begin{align}
\widehat{\mathbb E}
\bigl[(\mathbf M(t)-\mathbf M(r))G\bigr]=0.
\label{k3e1a}
\end{align}
Therefore, \eqref{A1p1w1d}, \eqref{zo1p5x},
\eqref{L1e45.},  and \eqref{k3e1a} show that 
\begin{align}
\mathbf M\, \textrm{is a
continuous square-integrable}\, H\times H\textrm{-valued}
(\mathcal F_t,\widehat{\mathbb P})\textrm{-martingale.}
\label{SW1}
\end{align}
By \eqref{L2p2w2v}-\eqref{L1P1w1}, the definition of $P_N$,   give
\begin{align}
(M_{u,N}(t),\phi)_H=M_u^{P_N\phi}(t),
\qquad
(M_{B,N}(t),\psi)_H=M_B^{P_N\psi}(t).
\label{X1Y1}
\end{align}
Since $P_N\phi\to\phi$ and $P_N\psi\to\psi$ in $H^1$, the integrability of $F_1(u,B)$ and $F_2(u,B)$ established in
\eqref{F1IC} and \eqref{F2IC} and
\eqref{zo1p5x} allow $N\to\infty$ in
\eqref{X1Y1}, yielding
\begin{align}
(M_u(t),\phi)_H=M_u^\phi(t),
\qquad
(M_B(t),\psi)_H=M_B^\psi(t),
\qquad t\geq0.
\label{X2Y2z}
\end{align}

\medskip
\noindent
\textit{Part 2. Proof of $\Phi\in L_2^0$, $\Phi$ predictable, and
$\widehat{\mathbb E}\int_0^T
\lVert\Phi(s)\rVert_{L_2^0}^2\,ds<\infty$.}

\noindent  Since \eqref{L1C2}-\eqref{K1M3} hold for every $T>0$, the equality in law \eqref{LLaW} yields
$u,B\in C(\mathbb R_+;H^{-1})$ $\widehat{\mathbb P}$-a.s.
Since $u,B\in C(\mathbb R_+;H^{-1})$ $\widehat{\mathbb P}$-a.s. and  $u$ and $B$ are $\{\mathcal F_t\}_{t\geq0}$-adapted by
\eqref{MF1}, they are
predictable as $H^{-1}$-valued processes. Hence,
\begin{align}
u,B\text{ are predictable as }H^{-1}\text{-valued processes}.
\label{Kl12Sq}
\end{align}
On the other hand, \eqref{kMMwrk} gives
\begin{align}
\widehat{\mathbb E}\int_0^T
\bigl(
\lVert u(s)\rVert_{H^{1,1}}^2+
\lVert B(s)\rVert_{H^{1,1}}^2
\bigr)\,ds\leq C_T,
\qquad T>0.
\label{vh2s}
\end{align}
Since $H^{1,1}\hookrightarrow H^{-1}$ continuously and injectively,
Kuratowski's theorem \cite[Theorem 8.3.7]{cohn2013measure} gives
\begin{align}
H^{1,1}\in\mathcal B(H^{-1}),
\qquad
\mathcal B(H^{1,1})
=
\{A\cap H^{1,1}:A\in\mathcal B(H^{-1})\}.
\label{kjl67x}
\end{align}
By \eqref{kjl67x}, \eqref{Kl12Sq}, and \eqref{vh2s},  it follows from Remark~5.7 in 
\cite{scheutzow2019stochastic} and Exercise~4.2.3(2) in 
\cite[p.~74]{prevot2007concise}, with
$V=H^{1,1}$, $H=H^{-1}$, and $\alpha=2$, that
\begin{align}
u,B\quad\text{are predictable }H^{1,1}\text{-valued processes}.
\label{X3y4s}
\end{align}
Hypothesis~\ref{HYp1} and
\eqref{X3y4s} imply that
$\sigma_i(\cdot,u,B)$ is predictable with values in
$L_2(\mathfrak U_i;H^{1,1})$. Since $H^{1,1}\hookrightarrow H$
continuously,
\begin{align}
\sigma_i(\cdot,u,B)
\quad\text{is predictable with values in }
L_2(\mathfrak U_i;H),
\qquad i=1,2.
\label{m1a2o1}
\end{align}
Set $\mathfrak U:=\mathfrak U_1\oplus\mathfrak U_2$ and choose
positive, self-adjoint, injective trace-class operators $Q_i$ on
$\mathfrak U_i$. Define
\begin{align}
Q:=Q_1\oplus Q_2,
\qquad
Q(\xi_1,\xi_2):=(Q_1\xi_1,Q_2\xi_2) \qquad \textrm{with}\;
(\xi_1,\xi_2)\in\mathfrak U,
\label{J1S23d}
\end{align}
then $Q$ is positive, self-adjoint, injective and trace class, with
$\operatorname{Tr}Q=\operatorname{Tr}Q_1+\operatorname{Tr}Q_2<\infty$,
and
$$
Q^{1/2}=Q_1^{1/2}\oplus Q_2^{1/2}.
$$
Set
$
\mathfrak U_0:=Q^{1/2}\mathfrak U$, and
$
L_2^0:=L_2(\mathfrak U_0;H\times H),
$
where
$(Q^{1/2}\xi,Q^{1/2}\zeta)_{\mathfrak U_0}
:=(\xi,\zeta)_{\mathfrak U}$.
If $\{\psi_j^{(i)}\}_{j\geq1}$ is an orthonormal basis of
$\mathfrak U_i$, then
\begin{align}
\mathcal E_0
:=
\{Q^{1/2}(\psi_j^{(1)},0):j\geq1\}
\cup
\{Q^{1/2}(0,\psi_j^{(2)}):j\geq1\}
\label{2p123}
\end{align}
is an orthonormal basis of $\mathfrak U_0$. Define $\Phi(s):\mathfrak U_0\to H\times H$ by
\begin{align}
\Phi(s)Q^{1/2}(\xi_1,\xi_2)
:=
\bigl(
\sigma_1(s,u(s),B(s))\xi_1,
\sigma_2(s,u(s),B(s))\xi_2
\bigr).
\label{P12f}
\end{align}
Using \eqref{2p123}, we obtain
\begin{align}
\lVert\Phi(s)\rVert_{L_2^0}^2
&=
\lVert\sigma_1(s,u(s),B(s))\rVert_{L_2(\mathfrak U_1;H)}^2
+
\lVert\sigma_2(s,u(s),B(s))\rVert_{L_2(\mathfrak U_2;H)}^2.
\label{P1q1c9}
\end{align}
To verify the predictability of $\Phi$, for
$S_i\in L_2(\mathfrak U_i;H)$, $i=1,2$, define the linear map
$\mathcal J:L_2(\mathfrak U_1;H)\times L_2(\mathfrak U_2;H)\to L_2^0$
by
$$
\mathcal J(S_1,S_2)Q^{1/2}(\xi_1,\xi_2)
:=
(S_1\xi_1,S_2\xi_2).
$$
Using \eqref{2p123}, we have
\begin{align}
\lVert\mathcal J(S_1,S_2)\rVert_{L_2^0}^2
=
\lVert S_1\rVert_{L_2(\mathfrak U_1;H)}^2
+
\lVert S_2\rVert_{L_2(\mathfrak U_2;H)}^2.
\label{J1w1o}
\end{align}
Thus, \eqref{J1w1o} shows that $\mathcal J$ is bounded
linear and hence continuous. Moreover, by
\eqref{P12f},
$\Phi=\mathcal J\bigl(\sigma_1(\cdot,u,B),\sigma_2(\cdot,u,B)\bigr)$.
Therefore, the continuity of $\mathcal J$ together with
\eqref{m1a2o1} yields
\begin{align}
\Phi\quad\text{is a predictable }L_2^0\text{-valued process}.
\label{R5D5X5g}
\end{align}
Finally, \eqref{P1q1c9} and
\eqref{HGRT1} yield
\begin{align}
\widehat{\mathbb E}\int_0^T
\lVert\Phi(s)\rVert_{L_2^0}^2\,ds<\infty,
\qquad T>0.
\label{R1DA1}
\end{align}
\textit{Part 3: Proof of
$\left\langle\!\left\langle\mathbf M\right\rangle\!\right\rangle_t
=
\int_0^t
\bigl(\Phi(s)Q^{1/2}\bigr)
\bigl(\Phi(s)Q^{1/2}\bigr)^\ast\,ds$.}

\noindent Fix $T>0$ and $h_i=(\phi_i,\psi_i)\in H\times H$, $i=1,2$, and set
$h_{i,N}:=(P_N\phi_i,P_N\psi_i)$. Since $P_N$ is the orthogonal
Fourier projection,
\begin{align}
h_{i,N}\longrightarrow h_i\quad\text{in }H\times H,
\qquad
\lVert h_{i,N}\rVert_{H\times H}
\leq
\lVert h_i\rVert_{H\times H}.
\label{R6D6X6h}
\end{align}
Set $X_{i,N}:=(\mathbf M,h_{i,N})_{H\times H}$ and
$X_i:=(\mathbf M,h_i)_{H\times H}$. By
\eqref{L1e45.} and
\eqref{R6D6X6h},
\begin{align}
X_{i,N}\longrightarrow X_i
\quad\text{in }L^2(\bar\Omega;C([0,T])),
\qquad i=1,2.
\label{P1W1a}
\end{align}
Moreover, \eqref{X2Y2z} gives
\begin{align}
X_{i,N}
=
M_u^{P_N\phi_i}+M_B^{P_N\psi_i}.
\label{P2W1b}
\end{align}
By the polarization identity for predictable covariation
\cite[\S~3.4]{da2014stochastic} together with \eqref{CC1Mw1}, we have
$$
\left\langle\!\left\langle
M_u^{P_N\phi_1},M_B^{P_N\psi_2}
\right\rangle\!\right\rangle_t
=
\left\langle\!\left\langle
M_B^{P_N\psi_1},M_u^{P_N\phi_2}
\right\rangle\!\right\rangle_t
=
0.
$$
Hence, by \eqref{P2W1b} and the
bilinearity of predictable covariation
\cite[\S~3.4]{da2014stochastic}, we obtain
\begin{align}
\left\langle\!\left\langle X_{1,N},X_{2,N}\right\rangle\!\right\rangle_t
&=
\int_0^t
\bigl(
\sigma_1^\ast(s,u,B)P_N\phi_1,
\sigma_1^\ast(s,u,B)P_N\phi_2
\bigr)_{\mathfrak U_1}\,ds
\nonumber\\
&\quad+
\int_0^t
\bigl(
\sigma_2^\ast(s,u,B)P_N\psi_1,
\sigma_2^\ast(s,u,B)P_N\psi_2
\bigr)_{\mathfrak U_2}\,ds.
\label{12PW}
\end{align}
Set $A(s):=\Phi(s)Q^{1/2}$. By
\eqref{P12f}, we obtain
\begin{align*}
A(s)(\xi_1,\xi_2)
=
\bigl(
\sigma_1(s,u,B)\xi_1,
\sigma_2(s,u,B)\xi_2
\bigr),
\end{align*}
and hence
\begin{align}
A(s)^\ast(\phi,\psi)
=
\bigl(
\sigma_1^\ast(s,u,B)\phi,
\sigma_2^\ast(s,u,B)\psi
\bigr).
\label{Pw12}
\end{align}
Using \eqref{Pw12} in
\eqref{12PW}, we obtain
\begin{align}
\left\langle\!\left\langle X_{1,N},X_{2,N}\right\rangle\!\right\rangle_t
=
\int_0^t
\bigl(
A(s)A(s)^\ast h_{1,N},
h_{2,N}
\bigr)_{H\times H}\,ds.
\label{z1t1}
\end{align}
By \eqref{P1q1c9} and
\eqref{R1DA1},
\begin{align}
\lVert A(s)\rVert_{L_2(\mathfrak U;H\times H)}^2
&=
\lVert\Phi(s)\rVert_{L_2^0}^2,
&
\widehat{\mathbb E}\int_0^T
\lVert A(s)\rVert_{L_2(\mathfrak U;H\times H)}^2\,ds
&<\infty.
\label{y1v1r}
\end{align}
Define
\begin{align}
C_N(t):=
\int_0^t
\bigl(
A(s)A(s)^\ast h_{1,N},
h_{2,N}
\bigr)_{H\times H}\,ds,
\quad
C(t):=
\int_0^t
\bigl(
A(s)A(s)^\ast h_1,
h_2
\bigr)_{H\times H}\,ds.
\label{g12sq}
\end{align}
Since
$
\lvert(AA^\ast x,y)_{H\times H}\rvert
\leq
\lVert A\rVert_{L_2}^2
\lVert x\rVert_{H\times H}
\lVert y\rVert_{H\times H},
$ using 
\eqref{R6D6X6h} and
\eqref{y1v1r}, it follows  from \eqref{g12sq} that
\begin{align}
\widehat{\mathbb E}
\sup_{0\leq t\leq T}
\lvert C_N(t)-C(t)\rvert
\longrightarrow0.
\label{l1z2}
\end{align}
Likewise, \eqref{P1W1a} and
\eqref{L1e45.} give
\begin{align}
\widehat{\mathbb E}
\sup_{0\leq t\leq T}
\lvert
X_{1,N}(t)X_{2,N}(t)-X_1(t)X_2(t)
\rvert
\longrightarrow0.
\label{t1z2}
\end{align}
By \eqref{z1t1} and the
definition of predictable covariation in
\cite[\S~3.4]{da2014stochastic},
$
Y_N(t):=X_{1,N}(t)X_{2,N}(t)-C_N(t)
$
is an $(\mathcal F_t,\widehat{\mathbb P})$-martingale. Set
$
Y(t):=X_1(t)X_2(t)-C(t).
$
By \eqref{l1z2} and
\eqref{t1z2},
$Y_N\to Y$ in $L^1(\bar\Omega;C([0,T]))$. Thus, for
$0\leq r<t\leq T$ and every bounded
$\mathcal F_r$-measurable random variable $G$,
\begin{align}
\widehat{\mathbb E}
\left[
\bigl(Y(t)-Y(r)\bigr)G
\right]
&=
\lim_{N\to\infty}
\widehat{\mathbb E}
\left[
\bigl(Y_N(t)-Y_N(r)\bigr)G
\right]
=
0,
\end{align}
which shows that $Y$ is an
$(\mathcal F_t,\widehat{\mathbb P})$-martingale. Since
$X_i=(\mathbf M,h_i)_{H\times H}$, the uniqueness of predictable
covariation in \cite[\S~3.4]{da2014stochastic} gives
\begin{align}
\left\langle\!\left\langle
(\mathbf M,h_1)_{H\times H},
(\mathbf M,h_2)_{H\times H}
\right\rangle\!\right\rangle_t
=
\int_0^t
\bigl(
A(s)A(s)^\ast h_1,
h_2
\bigr)_{H\times H}\,ds.
\label{R4D4X4}
\end{align}
Define
\begin{align}
\mathcal Q_{\mathbf M}(t)
:=
\int_0^tA(s)A(s)^\ast\,ds
=
\int_0^t
\bigl(\Phi(s)Q^{1/2}\bigr)
\bigl(\Phi(s)Q^{1/2}\bigr)^\ast\,ds.
\label{R3D3X3}
\end{align}
Since $A(s)$ is Hilbert-Schmidt, $A(s)A(s)^\ast$ is a positive
trace-class operator and
\begin{align}
\operatorname{Tr}\bigl(A(s)A(s)^\ast\bigr)
=
\lVert A(s)\rVert_{L_2(\mathfrak U;H\times H)}^2
=
\lVert\Phi(s)\rVert_{L_2^0}^2.
\label{R2DA1X}
\end{align}
Thus, by \eqref{R1DA1} and
\eqref{R2DA1X},
$\mathcal Q_{\mathbf M}$ is a well-defined
$L_1(H\times H)$-valued process with
$\mathcal Q_{\mathbf M}(0)=0$. Moreover,
\eqref{R5D5X5g} and
\eqref{R3D3X3} show that
$\mathcal Q_{\mathbf M}$ is adapted, while
\eqref{R1DA1} and
\eqref{R2DA1X} give continuity of
$t\mapsto\mathcal Q_{\mathbf M}(t)$ in $L_1(H\times H)$.

\noindent
For $0\leq r\leq t\leq T$ and $h\in H\times H$,
\eqref{R3D3X3} gives
$$
\bigl(
(\mathcal Q_{\mathbf M}(t)-\mathcal Q_{\mathbf M}(r))h,h
\bigr)_{H\times H}
=
\int_r^t
\lVert A(s)^\ast h\rVert_{\mathfrak U}^2\,ds
\geq0,
$$
so $\mathcal Q_{\mathbf M}$ is increasing. Moreover,
\eqref{R4D4X4} and
\eqref{R3D3X3} give
\begin{align}
\bigl(
\mathcal Q_{\mathbf M}(t)h_1,h_2
\bigr)_{H\times H}
=
\left\langle\!\left\langle
(\mathbf M,h_1)_{H\times H},
(\mathbf M,h_2)_{H\times H}
\right\rangle\!\right\rangle_t.
\label{c12zd8}
\end{align}
By \eqref{c12zd8}, for every
$h_1,h_2\in H\times H$, we obtain
$$
(\mathbf M(t),h_1)_{H\times H}
(\mathbf M(t),h_2)_{H\times H}
-
\bigl(
\mathcal Q_{\mathbf M}(t)h_1,h_2
\bigr)_{H\times H}
$$
is an $(\mathcal F_t,\widehat{\mathbb P})$-martingale. Hence
\eqref{R3D3X3} and
\eqref{c12zd8} show that
$\mathcal Q_{\mathbf M}$ satisfies the definition of the quadratic
variation of the $H\times H$-valued martingale $\mathbf M$ in
\cite[\S~3.4]{da2014stochastic}. Its uniqueness
\cite[Proposition~3.13]{da2014stochastic} yields
\begin{align}
\left\langle\!\left\langle
\mathbf M
\right\rangle\!\right\rangle_t
=
\int_0^t
\bigl(\Phi(s)Q^{1/2}\bigr)
\bigl(\Phi(s)Q^{1/2}\bigr)^\ast\,ds,
\qquad
0\leq t\leq T.
\label{3x2zg.}
\end{align}
Since $T>0$ was arbitrary, \eqref{3x2zg.}
holds for every $t\geq0$.

\medskip
\noindent
\medskip
\noindent
\textit{Part 4. Martingale representation and recovery of a
probabilistically weak solution.}
\medskip
\noindent 
By \eqref{SW1},
\eqref{R5D5X5g},
\eqref{R1DA1}, and
\eqref{3x2zg.},
the martingale representation theorem
\cite[Theorem~8.2]{da2014stochastic} yields the enlarged stochastic basis as follows
\begin{align}
(\Omega^\ast,\mathcal F^\ast,\{\mathcal F_t^\ast\}_{t\geq0},
\mathbb P^\ast),
\qquad
\Omega^\ast=\bar\Omega\times\widetilde\Omega,
\qquad
\mathbb P^\ast=\widehat{\mathbb P}\otimes\widetilde{\mathbb P},
\label{J1d2e1x}
\end{align}
together with an $\mathfrak U$-valued $Q$-Wiener process
$\widehat W$. Extending $u,B,\mathbf M$, and $\Phi$ to $\Omega^\ast$
via the projection onto $\bar\Omega$, and denoting the extensions by
$u^\ast,B^\ast,\mathbf M^\ast,\Phi^\ast$, respectively,
\eqref{3x2zg.} and
\cite[Theorem~8.2]{da2014stochastic} yield
\begin{align}
\mathbf M^\ast(t)
=
\int_0^t
\Phi^\ast(s)\,d\widehat W(s),
\qquad t\geq0.
\label{HJ12P}
\end{align}
Since $\mathbb P^\ast=\widehat{\mathbb P}\otimes
\widetilde{\mathbb P}$ in \eqref{J1d2e1x} and
$u^\ast,B^\ast$ depend only on the $\bar\Omega$-variable,
\eqref{LLaW} yields
\begin{align}
\mathcal L_{\mathbb P^\ast}(u^\ast,B^\ast)
=
\mathcal L_{\widehat{\mathbb P}}(u,B)
=
\widehat{\mathbb P}.
\label{P4w4F}
\end{align}
Since $Q=Q_1\oplus Q_2$ by
\eqref{J1S23d}, choose orthonormal eigenbases
$\{e_k^{(i)}\}_{k\geq1}$ of $\mathfrak U_i$ such that
\begin{align}
Q_ie_k^{(i)}
=
\lambda_k^{(i)}e_k^{(i)},
\qquad
\lambda_k^{(i)}>0,
\qquad
\sum_{k=1}^\infty\lambda_k^{(i)}<\infty,
\qquad i=1,2.
\label{S1P12x}
\end{align}
By \eqref{J1S23d} and
\eqref{S1P12x}, the $Q$-Wiener process $\widehat W$ admits the
expansion
\begin{align}
\widehat W(t)
&=
\sum_{k=1}^\infty
\sqrt{\lambda_k^{(1)}}\,
\beta_k^{(1)}(t)(e_k^{(1)},0)
+
\sum_{k=1}^\infty
\sqrt{\lambda_k^{(2)}}\,
\beta_k^{(2)}(t)(0,e_k^{(2)}),
\label{l1swpW}
\end{align}
where
$\{\beta_k^{(1)}\}_{k\geq1}$ and
$\{\beta_k^{(2)}\}_{k\geq1}$ are two independent families of
independent standard real Brownian motions. Define the cylindrical
Wiener processes $W_i$ on $\mathfrak U_i$ by
\begin{align}
W_i(t)\xi
:=
\sum_{k=1}^\infty
\beta_k^{(i)}(t)
(\xi,e_k^{(i)})_{\mathfrak U_i},
\qquad
\xi\in\mathfrak U_i,
\qquad i=1,2,
\label{C1w1Q1}
\end{align}
hence
\eqref{l1swpW} gives
\begin{align}
W_1\quad\text{and}\quad W_2
\quad\text{are independent cylindrical Wiener processes}.
\label{I1L1p1}
\end{align}
By \eqref{l1swpW},
\eqref{C1w1Q1}, and
\eqref{P12f}, we obtain
\begin{align}
\int_0^t
\Phi^\ast(s)\,d\widehat W(s)
&=
\sum_{k=1}^\infty
\int_0^t
\Phi^\ast(s)Q^{1/2}(e_k^{(1)},0)\,
d\beta_k^{(1)}(s)+
\sum_{k=1}^\infty
\int_0^t
\Phi^\ast(s)Q^{1/2}(0,e_k^{(2)})\,
d\beta_k^{(2)}(s)
\nonumber\\
&=
\left(
\sum_{k=1}^\infty
\int_0^t
\sigma_1(s,u^\ast(s),B^\ast(s))e_k^{(1)}
\,d\beta_k^{(1)}(s),
\right.
\nonumber\\
&\hspace{20mm}\left.
\sum_{k=1}^\infty
\int_0^t
\sigma_2(s,u^\ast(s),B^\ast(s))e_k^{(2)}
\,d\beta_k^{(2)}(s)
\right)
\nonumber\\
&=
\left(
\int_0^t
\sigma_1(s,u^\ast(s),B^\ast(s))\,dW_1(s),
\int_0^t
\sigma_2(s,u^\ast(s),B^\ast(s))\,dW_2(s)
\right).
\label{sdPW1}
\end{align}
Consequently, writing
$\mathbf M^\ast=(M_u^\ast,M_B^\ast)$,
\eqref{HJ12P} and
\eqref{sdPW1} give
\begin{align}
M_u^\ast(t)
&=
\int_0^t
\sigma_1(s,u^\ast(s),B^\ast(s))\,dW_1(s),
\label{U1P1W1}\\
M_B^\ast(t)
&=
\int_0^t
\sigma_2(s,u^\ast(s),B^\ast(s))\,dW_2(s).
\label{B1P1Q1}
\end{align}
By \eqref{X2Y2z}, after extending the
processes to $\Omega^\ast$, we have
$$
(M_u^\ast(t),\phi)_H=(M_u^\phi)^\ast(t),
\qquad
(M_B^\ast(t),\psi)_H=(M_B^\psi)^\ast(t),
$$
for every divergence-free $\phi,\psi\in C^\infty$.
Therefore, Definition~\ref{DEF}{\rm (M2)},
\eqref{U1P1W1}, and the first identity above
give
\begin{align}
\langle u^\ast(t),\phi\rangle
=
\langle u_0,\phi\rangle
+
\int_0^t
\langle F_1(u^\ast(s),B^\ast(s)),\phi\rangle\,ds+
\int_0^t
\left\langle
\sigma_1(s,u^\ast(s),B^\ast(s))\,dW_1(s),\phi
\right\rangle,
\label{U321s}
\end{align}
for every divergence-free $\phi\in C^\infty$.
Similarly, Definition~\ref{DEF}{\rm (M2)},
\eqref{B1P1Q1}, and the second identity above
give
\begin{align}
\langle B^\ast(t),\psi\rangle
&=
\langle B_0,\psi\rangle
+
\int_0^t
\langle F_2(u^\ast(s),B^\ast(s)),\psi\rangle\,ds+
\int_0^t
\left\langle
\sigma_2(s,u^\ast(s),B^\ast(s))\,dW_2(s),\psi
\right\rangle,
\label{P5W5z}
\end{align}
for every divergence-free $\psi\in C^\infty$.
By \eqref{P4w4F} and the initial condition of the martingale
solution in Definition~\ref{DEF},
\begin{align}
\mathbb P^\ast
\bigl(
u^\ast(0)=u_0,\ B^\ast(0)=B_0
\bigr)
=
1.
\label{P3W3I}
\end{align}
Moreover, \eqref{P4w4F} and
Definition~\ref{DEF}{\rm (M1)} give
\begin{align}
u^\ast
&\in
L^\infty_{\mathrm{loc}}(\mathbb R_+;H^{0,1})
\cap
L^2_{\mathrm{loc}}(\mathbb R_+;H^{1,1})
\cap
C(\mathbb R_+;H^{-1}),
\label{p7f7a}\\
B^\ast
&\in
L^\infty_{\mathrm{loc}}(\mathbb R_+;H^{0,1})
\cap
L^2_{\mathrm{loc}}(\mathbb R_+;H^1)
\cap
L^2_{\mathrm{loc}}(\mathbb R_+;H^{1,1})
\cap
L^2_{\mathrm{loc}}(\mathbb R_+;H^{0,2})
\cap
C(\mathbb R_+;H^{-1}),
\label{klwa1}
\end{align}
$\mathbb P^\ast$-a.s. Furthermore, by
\eqref{P4w4F} and Definition~\ref{DEF}{\rm (M1)} imply that,
for every $T>0$,
\begin{align}
&\int_0^T
\Bigl(
\lVert F_1(u^\ast(s),B^\ast(s))\rVert_{H^{-1}}
+
\lVert F_2(u^\ast(s),B^\ast(s))\rVert_{H^{-1}}
\Bigr)\,ds
\nonumber\\
&\quad+
\int_0^T
\Bigl(
\lVert
\sigma_1(s,u^\ast(s),B^\ast(s))
\rVert_{L_2(\mathfrak U_1;H)}^2
+
\lVert
\sigma_2(s,u^\ast(s),B^\ast(s))
\rVert_{L_2(\mathfrak U_2;H)}^2
\Bigr)\,ds
<\infty.
\label{I1P1W1}
\end{align}
Thus \eqref{p7f7a}-\eqref{klwa1} verify
Definition~\ref{d-w-s}{\rm (i)}, while
\eqref{I1P1W1} verifies
Definition~\ref{d-w-s}{\rm (ii)}.
Moreover,
\eqref{C1w1Q1}-
\eqref{I1L1p1},
\eqref{U321s}-\eqref{P5W5z}, and
\eqref{P3W3I} verify Definition~\ref{d-w-s}{\rm (iii)}.
Consequently,
$
\bigl((u^\ast,B^\ast),(W_1,W_2)\bigr)
$
is a probabilistically weak solution in the sense of
Definition~\ref{d-w-s} on
$
(\Omega^\ast,\mathcal F^\ast,
\{\mathcal F_t^\ast\}_{t\geq0},\mathbb P^\ast).$ This completes the proof of Theorem~\ref{thm10}\qed
\medskip

\noindent We next prove the uniqueness result stated in Theorem~\ref{thm1}.
\subsection*{Proof of Theorem \eqref{thm1} }

\noindent\noindent Let
$\bigl(u^{(1)},B^{(1)}\bigr)$ and
$\bigl(u^{(2)},B^{(2)}\bigr)$
be two solutions on the same stochastic basis, driven by the same
Wiener processes $W_1,W_2$, with the same initial data, and set
$
U:=u^{(1)}-u^{(2)},
\,
V:=B^{(1)}-B^{(2)},
\,
\bar q:=p^{(1)}-p^{(2)}.
$
For $i=1,2$, define
\begin{align}
\Sigma_i(t)
:=
\sigma_i
\bigl(t,u^{(1)}(t),B^{(1)}(t)\bigr)
-
\sigma_i
\bigl(t,u^{(2)}(t),B^{(2)}(t)\bigr).
\label{DE0}
\end{align}
Subtracting the two solutions in $\eqref{MHD}_1$ and using
\eqref{DE0}, we obtain
\begin{align}
dU
+
\Big[
(U\cdot\nabla)u^{(1)}
+
(u^{(2)}\cdot\nabla)U
-
(V\cdot\nabla)B^{(1)}
-
(B^{(2)}\cdot\nabla)V
-
\nu\partial_1^2U
+
\nabla \bar q
\Big]dt
&=
\Sigma_1\,dW_1,
\label{DE2}
\end{align}
while subtracting the two solutions in $\eqref{MHD}_2$ gives
\begin{align}
dV
+
\Big[
(U\cdot\nabla)B^{(1)}
+
(u^{(2)}\cdot\nabla)V
-
(V\cdot\nabla)u^{(1)}
-
(B^{(2)}\cdot\nabla)U
-
\eta\Delta V
\Big]dt
&=
\Sigma_2\,dW_2.
\label{DE3}
\end{align}
Since $\nabla\cdot U=0$, pairing \eqref{DE2} with $U$ eliminates the
pressure term through
$(\nabla \bar q,U)=-(\bar q,\nabla\cdot U)=0$.
Applying It\^o's formula to
$\lVert U(t)\rVert_{L^2}^{2}$ in \eqref{DE2} gives
\begin{align}
d\lVert U(t)\rVert_{L^2}^{2}
&=
2(U,dU)
+
\lVert \Sigma_1(t)\rVert_{L_2(\mathfrak U_1;L^2)}^{2}\,dt.
\label{I-U}
\end{align}
Substituting \eqref{DE2} into \eqref{I-U}, we obtain
\begin{align}
d\lVert U\rVert_{L^2}^{2}
&=
-2\bigl((U\cdot\nabla)u^{(1)},U\bigr)\,dt
-2\bigl((u^{(2)}\cdot\nabla)U,U\bigr)\,dt
\nonumber\\
&\quad
+
2\bigl((V\cdot\nabla)B^{(1)},U\bigr)\,dt
+
2\bigl((B^{(2)}\cdot\nabla)V,U\bigr)\,dt
\nonumber\\
&\quad
+
2\nu(\partial_1^2U,U)\,dt
-
2(\nabla q,U)\,dt
\nonumber\\
&\quad
+
2(U,\Sigma_1\,dW_1)
+
\lVert \Sigma_1\rVert_{L_2(\mathfrak U_1;L^2)}^{2}\,dt.
\label{E1ID}
\end{align}
Applying It\^o's formula to
$\lVert V(t)\rVert_{L^2}^{2}$ in \eqref{DE3} gives
\begin{align}
d\lVert V(t)\rVert_{L^2}^{2}
&=
2(V,dV)
+
\lVert \Sigma_2(t)\rVert_{L_2(\mathfrak U_2;L^2)}^{2}\,dt.
\label{Ito-H}
\end{align}
Substituting \eqref{DE3} into \eqref{Ito-H}, we obtain
\begin{align}
d\lVert V\rVert_{L^2}^{2}
&=
-2\bigl((U\cdot\nabla)B^{(1)},V\bigr)\,dt
-
2\bigl((u^{(2)}\cdot\nabla)V,V\bigr)\,dt
\nonumber\\
&\quad
+
2\bigl((V\cdot\nabla)u^{(1)},V\bigr)\,dt
+
2\bigl((B^{(2)}\cdot\nabla)U,V\bigr)\,dt
\nonumber\\
&\quad
+
2\eta(\Delta V,V)\,dt
+
2(V,\Sigma_2\,dW_2)
\nonumber\\
&\quad
+
\lVert \Sigma_2\rVert_{L_2(\mathfrak U_2;L^2)}^{2}\,dt.
\label{E2HA}
\end{align}
Adding \eqref{E1ID} and \eqref{E2HA}, and using
$
\nabla\cdot u^{(2)}=0,
$
Since
$\nabla\cdot B^{(2)}=0$, and 
$$
\bigl((B^{(2)}\cdot\nabla)V,U\bigr)
+
\bigl((B^{(2)}\cdot\nabla)U,V\bigr)
=
0,
$$ we obtain
\begin{align}
d\left(
\lVert U(t)\rVert_{L^2}^{2}
+
\lVert V(t)\rVert_{L^2}^{2}
\right)
&
+
2\nu\lVert \partial_1U(t)\rVert_{L^2}^{2}\,dt
+
2\eta\lVert \nabla V(t)\rVert_{L^2}^{2}\,dt
\nonumber\\
&=
\Big[
-2\bigl((U\cdot\nabla)u^{(1)},U\bigr)
+
2\bigl((V\cdot\nabla)B^{(1)},U\bigr)
\nonumber\\
&\qquad
-
2\bigl((U\cdot\nabla)B^{(1)},V\bigr)
+
2\bigl((V\cdot\nabla)u^{(1)},V\bigr)
\Big]dt
\nonumber\\
&\quad
+
\lVert \Sigma_1(t)\rVert_{L_2(\mathfrak U_1;L^2)}^{2}\,dt
+
\lVert \Sigma_2(t)\rVert_{L_2(\mathfrak U_2;L^2)}^{2}\,dt
\nonumber\\
&\quad
+
2(U(t),\Sigma_1(t)\,dW_1(t))
+
2(V(t),\Sigma_2(t)\,dW_2(t)).
\label{FIE1}
\end{align}
We now estimate each of the four nonlinear terms in
\eqref{FIE1} explicitly.
Applying H\"older's inequality, \eqref{Lia}, Young's inequality, we obtain
\begin{align}
\left\lvert
\bigl((U\cdot\nabla)u^{(1)},U\bigr)
\right\rvert
&\le
\lVert U^1\rVert_{L_h^\infty L_v^2}
\lVert\partial_1u^{(1)}\rVert_{L_h^2L_v^\infty}
\lVert U\rVert_{L^2}
+
\lVert U^2\rVert_{L_h^2L_v^\infty}
\lVert\partial_2u^{(1)}\rVert_{L_h^\infty L_v^2}
\lVert U\rVert_{L^2}
\nonumber\\
&\le
C
\left(
\lVert U\rVert_{L^2}^{1/2}
\lVert\partial_1U\rVert_{L^2}^{1/2}
+
\lVert U\rVert_{L^2}
\right)
\nonumber\\
&\times
\left(
\lVert\partial_1u^{(1)}\rVert_{L^2}^{1/2}
\lVert\partial_1\partial_2u^{(1)}\rVert_{L^2}^{1/2}
+
\lVert\partial_1u^{(1)}\rVert_{L^2}
\right)
\lVert U\rVert_{L^2}
\nonumber\\
&+
C
\left(
\lVert U\rVert_{L^2}^{1/2}
\lVert\partial_1U\rVert_{L^2}^{1/2}
+
\lVert U\rVert_{L^2}
\right)
\nonumber\\
&\times
\left(
\lVert\partial_2u^{(1)}\rVert_{L^2}^{1/2}
\lVert\partial_1\partial_2u^{(1)}\rVert_{L^2}^{1/2}
+
\lVert\partial_2u^{(1)}\rVert_{L^2}
\right)
\lVert U\rVert_{L^2}
\nonumber\\
&\le
C
\lVert U\rVert_{L^2}^{3/2}
\lVert\partial_1U\rVert_{L^2}^{1/2}
\lVert\partial_1u^{(1)}\rVert_{L^2}^{1/2}
\lVert\partial_1\partial_2u^{(1)}\rVert_{L^2}^{1/2}
\nonumber\\
&+
C
\lVert U\rVert_{L^2}^{3/2}
\lVert\partial_1U\rVert_{L^2}^{1/2}
\lVert\partial_1u^{(1)}\rVert_{L^2}
\nonumber\\
&+
C
\lVert U\rVert_{L^2}^{2}
\lVert\partial_1u^{(1)}\rVert_{L^2}^{1/2}
\lVert\partial_1\partial_2u^{(1)}\rVert_{L^2}^{1/2}
\nonumber\\
&+
C
\lVert U\rVert_{L^2}^{2}
\lVert\partial_1u^{(1)}\rVert_{L^2}
\nonumber\\
&+
C
\lVert U\rVert_{L^2}^{3/2}
\lVert\partial_1U\rVert_{L^2}^{1/2}
\lVert\partial_2u^{(1)}\rVert_{L^2}^{1/2}
\lVert\partial_1\partial_2u^{(1)}\rVert_{L^2}^{1/2}
\nonumber\\
&+
C
\lVert U\rVert_{L^2}^{3/2}
\lVert\partial_1U\rVert_{L^2}^{1/2}
\lVert\partial_2u^{(1)}\rVert_{L^2}
\nonumber\\
&+
C
\lVert U\rVert_{L^2}^{2}
\lVert\partial_2u^{(1)}\rVert_{L^2}^{1/2}
\lVert\partial_1\partial_2u^{(1)}\rVert_{L^2}^{1/2}
+
C
\lVert U\rVert_{L^2}^{2}
\lVert\partial_2u^{(1)}\rVert_{L^2}
\nonumber\\
&\le
\frac{\nu}{4}
\lVert\partial_1U\rVert_{L^2}^{2}
+
C_\nu
\Big[
\lVert\partial_1u^{(1)}\rVert_{L^2}^{2/3}
\lVert\partial_1\partial_2u^{(1)}\rVert_{L^2}^{2/3}
+
\lVert\partial_1u^{(1)}\rVert_{L^2}^{4/3}
\nonumber\\
&
+
\lVert\partial_2u^{(1)}\rVert_{L^2}^{2/3}
\lVert\partial_1\partial_2u^{(1)}\rVert_{L^2}^{2/3}
+
\lVert\partial_2u^{(1)}\rVert_{L^2}^{4/3}
\nonumber\\
&
+
\lVert\partial_1u^{(1)}\rVert_{L^2}^{1/2}
\lVert\partial_1\partial_2u^{(1)}\rVert_{L^2}^{1/2}
+
\lVert\partial_1u^{(1)}\rVert_{L^2}
\nonumber\\
&
+
\lVert\partial_2u^{(1)}\rVert_{L^2}^{1/2}
\lVert\partial_1\partial_2u^{(1)}\rVert_{L^2}^{1/2}
+
\lVert\partial_2u^{(1)}\rVert_{L^2}
\Big]
\lVert U\rVert_{L^2}^{2}
\nonumber\\
&\leq\frac{\nu}{4}\lVert \partial_1U\rVert_{L^2}^{2}
+ R_u(t)\lVert U\rVert_{L^2}^{2}. 
\label{DE5}
\end{align}
where 
\begin{align}
R_u(t)=&
C(1
+
\lVert \partial_1u^{(1)}(t)\rVert_{L^2}
+
\lVert \partial_2u^{(1)}(t)\rVert_{L^2}
+
\lVert \partial_1u^{(1)}(t)\rVert_{L^2}^{2}
+
\lVert \partial_2u^{(1)}(t)\rVert_{L^2}^{2}
+
\nonumber\\
&\lVert \partial_1\partial_2u^{(1)}(t)\rVert_{L^2}
+
\lVert \partial_1\partial_2u^{(1)}(t)\rVert_{L^2}^{2}).
\label{R_u}
\end{align}
Applying H\"older's inequality, \eqref{Lia}, Young's inequality, we obtain
\begin{align}
&
\left\lvert
\bigl((V\cdot\nabla)B^{(1)},U\bigr)
\right\rvert
\nonumber\\
&\le
C
\lVert V\rVert_{L^2}^{1/2}
\lVert \partial_1V\rVert_{L^2}^{1/2}
\lVert U\rVert_{L^2}
\lVert \partial_1B^{(1)}\rVert_{L^2}^{1/2}
\lVert \partial_1\partial_2B^{(1)}\rVert_{L^2}^{1/2}
\nonumber\\
&\quad+
C
\lVert V\rVert_{L^2}^{1/2}
\lVert \partial_1V\rVert_{L^2}^{1/2}
\lVert U\rVert_{L^2}
\lVert \partial_1B^{(1)}\rVert_{L^2}
\nonumber\\
&\quad+
C
\lVert V\rVert_{L^2}
\lVert U\rVert_{L^2}
\lVert \partial_1B^{(1)}\rVert_{L^2}^{1/2}
\lVert \partial_1\partial_2B^{(1)}\rVert_{L^2}^{1/2}
\nonumber\\
&\quad+
C
\lVert V\rVert_{L^2}
\lVert U\rVert_{L^2}
\lVert \partial_1B^{(1)}\rVert_{L^2}
\nonumber\\
&\quad+
C
\lVert V\rVert_{L^2}^{1/2}
\lVert \partial_1V\rVert_{L^2}^{1/2}
\lVert U\rVert_{L^2}
\lVert \partial_2B^{(1)}\rVert_{L^2}^{1/2}
\lVert \partial_1\partial_2B^{(1)}\rVert_{L^2}^{1/2}
\nonumber\\
&\quad+
C
\lVert V\rVert_{L^2}^{1/2}
\lVert \partial_1V\rVert_{L^2}^{1/2}
\lVert U\rVert_{L^2}
\lVert \partial_2B^{(1)}\rVert_{L^2}
\nonumber\\
&\quad+
C
\lVert V\rVert_{L^2}
\lVert U\rVert_{L^2}
\lVert \partial_2B^{(1)}\rVert_{L^2}^{1/2}
\lVert \partial_1\partial_2B^{(1)}\rVert_{L^2}^{1/2}
\nonumber\\
&\quad+
C
\lVert V\rVert_{L^2}
\lVert U\rVert_{L^2}
\lVert \partial_2B^{(1)}\rVert_{L^2}
\nonumber\\
&\le
\frac{\eta}{4}\lVert \partial_1V\rVert_{L^2}^{2}+
C_\eta
\Big[
\lVert \partial_1B^{(1)}\rVert_{L^2}^{2/3}
\lVert \partial_1\partial_2B^{(1)}\rVert_{L^2}^{2/3}
+
\lVert \partial_1B^{(1)}\rVert_{L^2}^{4/3}
\nonumber\\
&\qquad+
\lVert \partial_2B^{(1)}\rVert_{L^2}^{2/3}
\lVert \partial_1\partial_2B^{(1)}\rVert_{L^2}^{2/3}
+
\lVert \partial_2B^{(1)}\rVert_{L^2}^{4/3}
\nonumber\\
&\qquad+
\lVert \partial_1B^{(1)}\rVert_{L^2}^{1/2}
\lVert \partial_1\partial_2B^{(1)}\rVert_{L^2}^{1/2}
+
\lVert \partial_1B^{(1)}\rVert_{L^2}
\nonumber\\
&+
\lVert \partial_2B^{(1)}\rVert_{L^2}^{1/2}
\lVert \partial_1\partial_2B^{(1)}\rVert_{L^2}^{1/2}
+
\lVert \partial_2B^{(1)}\rVert_{L^2}
\Big]
\times
(\lVert U\rVert_{L^2}^{2}
+
\lVert V\rVert_{L^2}^{2})
\nonumber\\
&\le
\frac{\eta}{4}\lVert \nabla V\rVert_{L^2}^{2}
+
C R_B(t)
\left(
\lVert U\rVert_{L^2}^{2}
+
\lVert V\rVert_{L^2}^{2}
\right),
\label{H1E1}
\end{align}
where
\begin{align}
R_B(t)
:={}&
1
+
\lVert \partial_1B^{(1)}(t)\rVert_{L^2}
+
\lVert \partial_2B^{(1)}(t)\rVert_{L^2}
+
\lVert \partial_1B^{(1)}(t)\rVert_{L^2}^{2}
\nonumber\\
&+
\lVert \partial_2B^{(1)}(t)\rVert_{L^2}^{2}
+
\lVert \partial_1\partial_2B^{(1)}(t)\rVert_{L^2}
+
\lVert \partial_1\partial_2B^{(1)}(t)\rVert_{L^2}^{2}.
\label{RB}
\end{align}
Applying H\"older's inequality, \eqref{Lia}, Young's inequality, we obtain
\begin{align}
&
\left\lvert
\bigl((U\cdot\nabla)B^{(1)},V\bigr)
\right\rvert
\nonumber\\
&\le
C
\lVert U\rVert_{L^2}^{1/2}
\lVert \partial_1U\rVert_{L^2}^{1/2}
\lVert V\rVert_{L^2}
\lVert \partial_1B^{(1)}\rVert_{L^2}^{1/2}
\lVert \partial_1\partial_2B^{(1)}\rVert_{L^2}^{1/2}
\nonumber\\
&+
C
\lVert U\rVert_{L^2}^{1/2}
\lVert \partial_1U\rVert_{L^2}^{1/2}
\lVert V\rVert_{L^2}
\lVert \partial_1B^{(1)}\rVert_{L^2}
\nonumber\\
&+
C
\lVert U\rVert_{L^2}
\lVert V\rVert_{L^2}
\lVert \partial_1B^{(1)}\rVert_{L^2}^{1/2}
\lVert \partial_1\partial_2B^{(1)}\rVert_{L^2}^{1/2}
\nonumber\\
&+
C
\lVert U\rVert_{L^2}
\lVert V\rVert_{L^2}
\lVert \partial_1B^{(1)}\rVert_{L^2}
\nonumber\\
&+
C
\lVert U\rVert_{L^2}^{1/2}
\lVert \partial_1U\rVert_{L^2}^{1/2}
\lVert V\rVert_{L^2}
\lVert \partial_2B^{(1)}\rVert_{L^2}^{1/2}
\lVert \partial_1\partial_2B^{(1)}\rVert_{L^2}^{1/2}
\nonumber\\
&+
C
\lVert U\rVert_{L^2}^{1/2}
\lVert \partial_1U\rVert_{L^2}^{1/2}
\lVert V\rVert_{L^2}
\lVert \partial_2B^{(1)}\rVert_{L^2}
\nonumber\\
&+
C
\lVert U\rVert_{L^2}
\lVert V\rVert_{L^2}
\lVert \partial_2B^{(1)}\rVert_{L^2}^{1/2}
\lVert \partial_1\partial_2B^{(1)}\rVert_{L^2}^{1/2}
\nonumber\\
&+
C
\lVert U\rVert_{L^2}
\lVert V\rVert_{L^2}
\lVert \partial_2B^{(1)}\rVert_{L^2}
\nonumber\\
&\le
\frac{\nu}{4}\lVert \partial_1U\rVert_{L^2}^{2}
+
C
\Big[
\lVert \partial_1B^{(1)}\rVert_{L^2}^{2/3}
\lVert \partial_1\partial_2B^{(1)}\rVert_{L^2}^{2/3}
+
\lVert \partial_1B^{(1)}\rVert_{L^2}^{4/3}
\nonumber\\
&+
\lVert \partial_2B^{(1)}\rVert_{L^2}^{2/3}
\lVert \partial_1\partial_2B^{(1)}\rVert_{L^2}^{2/3}
+
\lVert \partial_2B^{(1)}\rVert_{L^2}^{4/3}
\nonumber\\
&+
\lVert \partial_1B^{(1)}\rVert_{L^2}^{1/2}
\lVert \partial_1\partial_2B^{(1)}\rVert_{L^2}^{1/2}
+
\lVert \partial_1B^{(1)}\rVert_{L^2}
\nonumber\\
&+
\lVert \partial_2B^{(1)}\rVert_{L^2}^{1/2}
\lVert \partial_1\partial_2B^{(1)}\rVert_{L^2}^{1/2}
+
\lVert \partial_2B^{(1)}\rVert_{L^2}
\Big]\times
\left(
\lVert U\rVert_{L^2}^{2}
+
\lVert V\rVert_{L^2}^{2}
\right)
\nonumber\\
&\le
\frac{\nu}{4}\lVert \partial_1U\rVert_{L^2}^{2}
+ R_B(t)
\left(
\lVert U\rVert_{L^2}^{2}
+
\lVert V\rVert_{L^2}^{2}
\right),
\label{U1H1}
\end{align}
where $R_B$ is defined in \eqref{RB}.
Similarly, applying H\"older's inequality, \eqref{Lia}, Young's inequality, we obtain
\begin{align}
&
\left\lvert
\bigl((V\cdot\nabla)u^{(1)},V\bigr)
\right\rvert
\nonumber\\
&\le
C
\lVert V\rVert_{L^2}^{3/2}
\lVert \partial_1V\rVert_{L^2}^{1/2}
\lVert \partial_1u^{(1)}\rVert_{L^2}^{1/2}
\lVert \partial_1\partial_2u^{(1)}\rVert_{L^2}^{1/2}
\nonumber\\
&+
C
\lVert V\rVert_{L^2}^{3/2}
\lVert \partial_1V\rVert_{L^2}^{1/2}
\lVert \partial_1u^{(1)}\rVert_{L^2}
\nonumber\\
&+
C
\lVert V\rVert_{L^2}^{2}
\lVert \partial_1u^{(1)}\rVert_{L^2}^{1/2}
\lVert \partial_1\partial_2u^{(1)}\rVert_{L^2}^{1/2}
\nonumber\\
&+
C
\lVert V\rVert_{L^2}^{2}
\lVert \partial_1u^{(1)}\rVert_{L^2}
\nonumber\\
&+
C
\lVert V\rVert_{L^2}^{3/2}
\lVert \partial_1V\rVert_{L^2}^{1/2}
\lVert \partial_2u^{(1)}\rVert_{L^2}^{1/2}
\lVert \partial_1\partial_2u^{(1)}\rVert_{L^2}^{1/2}
\nonumber\\
&+
C
\lVert V\rVert_{L^2}^{3/2}
\lVert \partial_1V\rVert_{L^2}^{1/2}
\lVert \partial_2u^{(1)}\rVert_{L^2}
\nonumber\\
&+
C
\lVert V\rVert_{L^2}^{2}
\lVert \partial_2u^{(1)}\rVert_{L^2}^{1/2}
\lVert \partial_1\partial_2u^{(1)}\rVert_{L^2}^{1/2}
\nonumber\\
&+
C
\lVert V\rVert_{L^2}^{2}
\lVert \partial_2u^{(1)}\rVert_{L^2}
\nonumber\\
&\le
\frac{\eta}{4}\lVert \nabla V\rVert_{L^2}^{2}
+ R_u(t)
\left(
\lVert U\rVert_{L^2}^{2}
+
\lVert V\rVert_{L^2}^{2}
\right),
\label{HER3x}
\end{align}
where $R_u$ is defined in \eqref{R_u}.
By
\eqref{DE5},
\eqref{H1E1},
\eqref{U1H1}, and
\eqref{HER3x},
we have
\begin{align}
&
2\left\lvert
\bigl((U\cdot\nabla)u^{(1)},U\bigr)
\right\rvert
+
2\left\lvert
\bigl((V\cdot\nabla)B^{(1)},U\bigr)
\right\rvert
\nonumber\\
&\quad+
2\left\lvert
\bigl((U\cdot\nabla)B^{(1)},V\bigr)
\right\rvert
+
2\left\lvert
\bigl((V\cdot\nabla)u^{(1)},V\bigr)
\right\rvert
\nonumber\\
&\le
\nu\lVert \partial_1U\rVert_{L^2}^{2}
+
\eta\lVert \nabla V\rVert_{L^2}^{2}
+
C
\bigl(
R_u(t)+R_B(t)
\bigr)
\left(
\lVert U\rVert_{L^2}^{2}
+
\lVert V\rVert_{L^2}^{2}
\right).
\label{N1F1l}
\end{align}
With $R(t):=R_u(t)+R_B(t)$, \eqref{2.28H} and \eqref{0.49A} give $R\in L^1(0,T)$, while \eqref{N1F1l} yields
\begin{align}
&
2\left\lvert
\bigl((U\cdot\nabla)u^{(1)},U\bigr)
\right\rvert
+
2\left\lvert
\bigl((V\cdot\nabla)B^{(1)},U\bigr)
\right\rvert
\nonumber\\
&\quad+
2\left\lvert
\bigl((U\cdot\nabla)B^{(1)},V\bigr)
\right\rvert
+
2\left\lvert
\bigl((V\cdot\nabla)u^{(1)},V\bigr)
\right\rvert
\nonumber\\
&\le
\nu\lVert \partial_1U\rVert_{L^2}^{2}
+
\eta\lVert \nabla V\rVert_{L^2}^{2}
+
CR(t)
\left(
\lVert U\rVert_{L^2}^{2}
+
\lVert V\rVert_{L^2}^{2}
\right).
\label{N1F1lR}
\end{align}
Using \eqref{N1F1lR} in \eqref{FIE1}, we obtain
\begin{align}
&
d\left(
\lVert U(t)\rVert_{L^2}^{2}
+
\lVert V(t)\rVert_{L^2}^{2}
\right)
+
\nu\lVert \partial_1U(t)\rVert_{L^2}^{2}\,dt
+
\eta\lVert \nabla V(t)\rVert_{L^2}^{2}\,dt
\nonumber\\
&\le
CR(t)
\left(
\lVert U(t)\rVert_{L^2}^{2}
+
\lVert V(t)\rVert_{L^2}^{2}
\right)\,dt
\nonumber\\
&\quad+
\left(
\lVert \Sigma_1(t)\rVert_{L_2(\mathfrak U_1;L^2)}^{2}
+
\lVert \Sigma_2(t)\rVert_{L_2(\mathfrak U_2;L^2)}^{2}
\right)\,dt
\nonumber\\
&\quad+
2(U(t),\Sigma_1(t)\,dW_1(t))
+
2(V(t),\Sigma_2(t)\,dW_2(t)).
\label{Lwdz}
\end{align}
By \eqref{LSI}, we obtain
\begin{align}
\lVert \Sigma_1(t)\rVert_{L_2(\mathfrak U_1;L^2)}^{2}
+
\lVert \Sigma_2(t)\rVert_{L_2(\mathfrak U_2;L^2)}^{2}
&\le
L_1
\left(
\lVert U(t)\rVert_{L^2}^{2}
+
\lVert V(t)\rVert_{L^2}^{2}
\right)
\nonumber\\
&+
L_2
\left(
\nu\lVert \partial_1U(t)\rVert_{L^2}^{2}
+
\eta\lVert \nabla V(t)\rVert_{L^2}^{2}
\right).
\label{z1L3da}
\end{align}
Substituting \eqref{z1L3da} into
\eqref{Lwdz}, we obtain
\begin{align}
&
d\left(
\lVert U(t)\rVert_{L^2}^{2}
+
\lVert V(t)\rVert_{L^2}^{2}
\right)
\nonumber\\
&\quad+
(1-L_2)
\left(
\nu\lVert \partial_1U(t)\rVert_{L^2}^{2}
+
\eta\lVert \nabla V(t)\rVert_{L^2}^{2}
\right)\,dt
\nonumber\\
&\le
\bigl(CR(t)+L_1\bigr)
\left(
\lVert U(t)\rVert_{L^2}^{2}
+
\lVert V(t)\rVert_{L^2}^{2}
\right)\,dt
\nonumber\\
&\quad+
2(U(t),\Sigma_1(t)\,dW_1(t))
+
2(V(t),\Sigma_2(t)\,dW_2(t)).
\label{bcd12a}
\end{align}
Define $q(t):=\int_0^t\bigl(CR(s)+L_1\bigr)\,ds$. Since $R \in L^1(0,T)$ due to \eqref{2.28H} and \eqref{0.49A} , $q(t)<\infty$ for every $0\leq t\leq T$, and
 $q'(t)=CR(t)+L_1$ for a.e.
$t\in[0,T]$;  and therefore
$\bigl[e^{-q},\,\lVert U\rVert_{L^2}^{2}
+\lVert V\rVert_{L^2}^{2}\bigr]_t=0$. Applying the product rule gives
\begin{align}
d\left[
e^{-q(t)}
\left(
\lVert U(t)\rVert_{L^2}^{2}
+
\lVert V(t)\rVert_{L^2}^{2}
\right)
\right]
&=
e^{-q(t)}
d\left(
\lVert U(t)\rVert_{L^2}^{2}
+
\lVert V(t)\rVert_{L^2}^{2}
\right)
\nonumber\\
&\quad
-
q'(t)e^{-q(t)}
\left(
\lVert U(t)\rVert_{L^2}^{2}
+
\lVert V(t)\rVert_{L^2}^{2}
\right)\,dt.
\label{bp1a2wz}
\end{align}
Substituting \eqref{bcd12a} into \eqref{bp1a2wz}, we obtain
\begin{align}
&
d\left[
e^{-q(t)}
\left(
\lVert U(t)\rVert_{L^2}^{2}
+
\lVert V(t)\rVert_{L^2}^{2}
\right)
\right]
\nonumber\\
&\quad+
(1-L_2)e^{-q(t)}
\left(
\nu\lVert \partial_1U(t)\rVert_{L^2}^{2}
+
\eta\lVert \nabla V(t)\rVert_{L^2}^{2}
\right)\,dt
\nonumber\\
&\le
2e^{-q(t)}
(U(t),\Sigma_1(t)\,dW_1(t))
+
2e^{-q(t)}
(V(t),\Sigma_2(t)\,dW_2(t)).
\label{d1wa2z}
\end{align}
Integrating \eqref{d1wa2z} from $0$ to $t$, and
using $U(0)=0=V(0)$, we obtain
\begin{align}
&
e^{-q(t)}
\left(
\lVert U(t)\rVert_{L^2}^{2}
+
\lVert V(t)\rVert_{L^2}^{2}
\right)
\nonumber\\
&\quad+
(1-L_2)
\int_0^t
e^{-q(s)}
\left(
\nu\lVert \partial_1U(s)\rVert_{L^2}^{2}
+
\eta\lVert \nabla V(s)\rVert_{L^2}^{2}
\right)\,ds
\nonumber\\
&\le
2\int_0^t
e^{-q(s)}
(U(s),\Sigma_1(s)\,dW_1(s))+
2\int_0^t
e^{-q(s)}
(V(s),\Sigma_2(s)\,dW_2(s)).
\label{y1e1w1}
\end{align}
Taking the supremum over $0\le r\le t$ and taking expectations in
\eqref{y1e1w1},  we obtain
\begin{align}
&
\mathbb E
\sup_{0\le r\le t}
e^{-q(r)}
\left(
\lVert U(r)\rVert_{L^2}^{2}
+
\lVert V(r)\rVert_{L^2}^{2}
\right)
\nonumber\\
&\quad+
(1-L_2)
\mathbb E
\int_0^t
e^{-q(s)}
\left(
\nu\lVert \partial_1U(s)\rVert_{L^2}^{2}
+
\eta\lVert \nabla V(s)\rVert_{L^2}^{2}
\right)\,ds
\nonumber\\
&\le
\mathbb E
\sup_{0\le r\le t}
\left\lvert
2\int_0^r
e^{-q(s)}
(U(s),\Sigma_1(s)\,dW_1(s))
\right.
\left.
+
2\int_0^r
e^{-q(s)}
(V(s),\Sigma_2(s)\,dW_2(s))
\right\rvert.
\label{GHRBDS}
\end{align}
Applying Burkholder-Davis-Gundy inequality applied to the right-hand
side of \eqref{GHRBDS} and using \eqref{LSI}, we obtain
 we obtain
\begin{align}
&
\mathbb E
\sup_{0\le r\le t}
e^{-q(r)}
\left(
\lVert U(r)\rVert_{L^2}^{2}
+
\lVert V(r)\rVert_{L^2}^{2}
\right)
\nonumber\\
&\quad+
(1-L_2)
\mathbb E
\int_0^t
e^{-q(s)}
\left(
\nu\lVert \partial_1U(s)\rVert_{L^2}^{2}
+
\eta\lVert \nabla V(s)\rVert_{L^2}^{2}
\right)\,ds
\nonumber\\
&\le
\beta
\mathbb E
\sup_{0\le r\le t}
e^{-q(r)}
\left(
\lVert U(r)\rVert_{L^2}^{2}
+
\lVert V(r)\rVert_{L^2}^{2}
\right)
\nonumber\\
&\quad+
\frac{9L_1}{\beta}
\mathbb E
\int_0^t
e^{-q(s)}
\left(
\lVert U(s)\rVert_{L^2}^{2}
+
\lVert V(s)\rVert_{L^2}^{2}
\right)\,ds
\nonumber\\
&\quad+
\frac{9L_2}{\beta}
\mathbb E
\int_0^t
e^{-q(s)}
\left(
\nu\lVert \partial_1U(s)\rVert_{L^2}^{2}
+
\eta\lVert \nabla V(s)\rVert_{L^2}^{2}
\right)\,ds.
\label{E1W2A1}
\end{align}
Moving the first and third terms on the right-hand side of
\eqref{E1W2A1} to the left gives
\begin{align}
&
(1-\beta)
\mathbb E
\sup_{0\le r\le t}
e^{-q(r)}
\left(
\lVert U(r)\rVert_{L^2}^{2}
+
\lVert V(r)\rVert_{L^2}^{2}
\right)
\nonumber\\
&\quad+
\left(
1-L_2-\frac{9L_2}{\beta}
\right)
\mathbb E
\int_0^t
e^{-q(s)}
\left(
\nu\lVert \partial_1U(s)\rVert_{L^2}^{2}
+
\eta\lVert \nabla V(s)\rVert_{L^2}^{2}
\right)\,ds
\nonumber\\
&\le
\frac{9L_1}{\beta}
\mathbb E
\int_0^t
e^{-q(s)}
\left(
\lVert U(s)\rVert_{L^2}^{2}
+
\lVert V(s)\rVert_{L^2}^{2}
\right)\,ds.
\label{WECD}
\end{align}
Dropping the second nonnegative term in
\eqref{WECD}, and  dividing by $1-\beta$ and using the definition of supremum  and Tonell's Theorem, we obtain
\begin{align}
\mathbb E
\sup_{0\le r\le t}
e^{-q(r)}
\left(
\lVert U(r)\rVert_{L^2}^{2}
+
\lVert V(r)\rVert_{L^2}^{2}
\right)\le
\frac{9L_1}{\beta(1-\beta)}
\int_0^t
\mathbb E
\sup_{0\le r\le s}
e^{-q(r)}
\left(
\lVert U(r)\rVert_{L^2}^{2}
+
\lVert V(r)\rVert_{L^2}^{2}
\right)\,ds.
\label{G1W1u}
\end{align}
Gronwall's inequality applied to \eqref{G1W1u} gives
$$
\mathbb E
\sup_{0\le s\le t}
e^{-q(s)}
\left(
\lVert U(s)\rVert_{L^2}^{2}
+
\lVert V(s)\rVert_{L^2}^{2}
\right)
=
0,
$$
therefore,
\begin{align}
u^{(1)}=u^{(2)},
\qquad
B^{(1)}=B^{(2)}
\qquad
\text{a.s.}
\end{align}
Thus pathwise uniqueness holds.\qed
\medskip
\bibliographystyle{plain}
\bibliography{reference}

@article{goldys2009martingale,
  title={Martingale solutions and Markov selections for stochastic partial differential equations},
  author={Goldys, Benjamin and R{\"o}ckner, Michael and Zhang, Xicheng},
  journal={Stochastic Processes and their Applications},
  volume={119},
  number={5},
  pages={1725--1764},
  year={2009},
  publisher={Elsevier}
}

@article{yamazaki2026remarks,
  title={Remarks on the two-dimensional magnetohydrodynamics system forced by space-time white noise},
  author={Yamazaki, Kazuo},
  journal={Stochastic Processes and their Applications},
  pages={104893},
  year={2026},
  publisher={Elsevier}
}

@article{yamazaki2016global,
  title={Global martingale solution to the stochastic nonhomogeneous magnetohydrodynamics system},
  author={Yamazaki, Kazuo},
  year={2016}
}

@article{yamazaki2015global,
  title={Global regularity of N-dimensional generalized MHD system with anisotropic dissipation and diffusion},
  author={Yamazaki, Kazuo},
  journal={Nonlinear Analysis: Theory, Methods \& Applications},
  volume={122},
  pages={176--191},
  year={2015},
  publisher={Elsevier}
}

@article{yamazaki2014remarks,
  title={Remarks on the global regularity of the two-dimensional magnetohydrodynamics system with zero dissipation},
  author={Yamazaki, Kazuo},
  journal={Nonlinear Analysis: Theory, Methods \& Applications},
  volume={94},
  pages={194--205},
  year={2014},
  publisher={Elsevier}
}

@book{cohn2013measure,
  title={Measure theory},
  author={Cohn, Donald L},
  volume={2},
  year={2013},
  publisher={Springer}
}

@article{liang2021deterministic,
  title={Deterministic and stochastic 2D Navier-Stokes equations with anisotropic viscosity},
  author={Liang, Siyu and Zhang, Ping and Zhu, Rongchan},
  journal={Journal of Differential Equations},
  volume={275},
  pages={473--508},
  year={2021},
  publisher={Elsevier}
}

@article{scheutzow2019stochastic,
  title={Stochastic partial differential equations},
  author={Scheutzow, Michael},
  journal={Lecture Notes, BMS Advanced Course},
  year={2019}
}

@book{da2014stochastic,
  title={Stochastic equations in infinite dimensions},
  author={Da Prato, Giuseppe and Zabczyk, Jerzy},
  year={2014},
  publisher={Cambridge university press}
}

@article{LiLiuTang2021,
  author  = {Li, Jingna and Liu, Hongxia and Tang, Hao},
  title   = {Stochastic MHD equations with fractional kinematic dissipation and partial magnetic diffusion in $\mathbb{R}^2$},
  journal = {Stochastic Processes and their Applications},
  volume  = {135},
  pages   = {139--182},
  year    = {2021},
  doi     = {10.1016/j.spa.2021.01.008}
}

@book{prevot2007concise,
  title={A concise course on stochastic partial differential equations},
  author={Pr{\'e}v{\^o}t, Claudia and R{\"o}ckner, Michael},
  year={2007},
  publisher={Springer}
}

@book{jacodshiryaev,
  author    = {Jacod, Jean and Shiryaev, Albert N.},
  title     = {Limit Theorems for Stochastic Processes},
  edition   = {2},
  series    = {Grundlehren der mathematischen Wissenschaften},
  volume    = {288},
  publisher = {Springer-Verlag},
  address   = {Berlin},
  year      = {2003}
}

@misc{protter2004stochastic,
  title={Stochastic Integration and Differential equations},
  author={Protter, Phillip E},
  year={2004},
  publisher={Citeseer}
}

@article{DuvautLions1972,
  author  = {Duvaut, Georges and Lions, Jacques-Louis},
  title   = {In{\'e}quations en thermo{\'e}lasticit{\'e} et magn{\'e}tohydrodynamique},
  journal = {Archive for Rational Mechanics and Analysis},
  volume  = {46},
  pages   = {241--279},
  year    = {1972},
  doi     = {10.1007/BF00250512}
}

@article{SermangeTemam1983,
  author  = {Sermange, Michel and Temam, Roger},
  title   = {Some mathematical questions related to the {MHD} equations},
  journal = {Communications on Pure and Applied Mathematics},
  volume  = {36},
  number  = {5},
  pages   = {635--664},
  year    = {1983},
  doi     = {10.1002/cpa.3160360506}
}

@article{Wu2003,
  author  = {Wu, Jiahong},
  title   = {Generalized {MHD} equations},
  journal = {Journal of Differential Equations},
  volume  = {195},
  number  = {2},
  pages   = {284--312},
  year    = {2003},
  doi     = {10.1016/j.jde.2003.07.007}
}

@article{CaoWu2011,
  author  = {Cao, Chongsheng and Wu, Jiahong},
  title   = {Global regularity for the 2{D} {MHD} equations with mixed partial dissipation and magnetic diffusion},
  journal = {Advances in Mathematics},
  volume  = {226},
  number  = {2},
  pages   = {1803--1822},
  year    = {2011},
  doi     = {10.1016/j.aim.2010.08.017}
}

@article{CaoRegmiWu2013,
  author  = {Cao, Chongsheng and Regmi, Dipendra and Wu, Jiahong},
  title   = {The 2{D} {MHD} equations with horizontal dissipation and horizontal magnetic diffusion},
  journal = {Journal of Differential Equations},
  volume  = {254},
  number  = {7},
  pages   = {2661--2681},
  year    = {2013},
  doi     = {10.1016/j.jde.2013.01.002}
}

@article{CaoWuYuan2014,
  author  = {Cao, Chongsheng and Wu, Jiahong and Yuan, Baoquan},
  title   = {The 2{D} incompressible magnetohydrodynamics equations with only magnetic diffusion},
  journal = {SIAM Journal on Mathematical Analysis},
  volume  = {46},
  pages   = {588--602},
  year    = {2014},
  doi     = {10.1137/130937718}
}

@article{DuZhou2015,
  author  = {Du, Lili and Zhou, Deqin},
  title   = {Global well-posedness of two-dimensional magnetohydrodynamic flows with partial dissipation and magnetic diffusion},
  journal = {SIAM Journal on Mathematical Analysis},
  volume  = {47},
  number  = {2},
  pages   = {1562--1589},
  year    = {2015},
  doi     = {10.1137/140959821}
}

@article{JiuNiuWuXuYu2015,
  author  = {Jiu, Quansen and Niu, Dongjuan and Wu, Jiahong and Xu, Xiaojing and Yu, Huan},
  title   = {The 2{D} magnetohydrodynamic equations with magnetic diffusion},
  journal = {Nonlinearity},
  volume  = {28},
  number  = {11},
  pages   = {3935--3955},
  year    = {2015},
  doi     = {10.1088/0951-7715/28/11/3935}
}

@article{DongJiaLiWu2018,
  author  = {Dong, Bo-Qing and Jia, Yan and Li, Jingna and Wu, Jiahong},
  title   = {Global regularity and time decay for the 2{D} magnetohydrodynamic equations with fractional dissipation and partial magnetic diffusion},
  journal = {Journal of Mathematical Fluid Mechanics},
  volume  = {20},
  number  = {4},
  pages   = {1541--1565},
  year    = {2018},
  doi     = {10.1007/s00021-018-0376-3}
}

@article{DongLiWu2019,
  author  = {Dong, Bo-Qing and Li, Jingna and Wu, Jiahong},
  title   = {Global regularity for the 2{D} {MHD} equations with partial hyper-resistivity},
  journal = {International Mathematics Research Notices},
  volume  = {2019},
  number  = {14},
  pages   = {4261--4280},
  year    = {2019},
  doi     = {10.1093/imrn/rnx240}
}

@article{SritharanSundar1999,
  author  = {Sritharan, S. S. and Sundar, P.},
  title   = {The stochastic magneto-hydrodynamic system},
  journal = {Infinite Dimensional Analysis, Quantum Probability and Related Topics},
  volume  = {2},
  number  = {2},
  pages   = {241--265},
  year    = {1999},
  doi     = {10.1142/S0219025799000138}
}

@article{BarbuDaPrato2007,
  author  = {Barbu, Viorel and Da Prato, Giuseppe},
  title   = {Existence and ergodicity for the two-dimensional stochastic magneto-hydrodynamics equations},
  journal = {Applied Mathematics \& Optimization},
  volume  = {56},
  number  = {2},
  pages   = {145--168},
  year    = {2007},
  doi     = {10.1007/s00245-007-0882-2}
}

@article{ChueshovMillet2010,
  author  = {Chueshov, Igor and Millet, Annie},
  title   = {Stochastic 2{D} hydrodynamical type systems: Well posedness and large deviations},
  journal = {Applied Mathematics \& Optimization},
  volume  = {61},
  pages   = {379--420},
  year    = {2010},
  doi     = {10.1007/s00245-009-9091-z}
}

@article{HuangShen2016,
  author  = {Huang, Jianhua and Shen, Tianlong},
  title   = {Well-posedness and dynamics of the stochastic fractional magneto-hydrodynamic equations},
  journal = {Nonlinear Analysis},
  volume  = {133},
  pages   = {102--133},
  year    = {2016},
  doi     = {10.1016/j.na.2015.12.001}
}

@article{HongLiLiu2021,
  author  = {Hong, Wei and Li, Shihu and Liu, Wei},
  title   = {Well-posedness and exponential mixing for stochastic magneto-hydrodynamic equations with fractional dissipations},
  journal = {Frontiers of Mathematics in China},
  volume  = {16},
  number  = {2},
  pages   = {425--457},
  year    = {2021},
  doi     = {10.1007/s11464-021-0910-0}
}
\end{document}